\documentclass[11pt]{amsart}

\usepackage{a4wide, mathdots, url}
\usepackage{graphicx}
\usepackage{amsopn}
\usepackage{amsmath, amsfonts, amssymb, amscd, amsthm}
\usepackage{yhmath}
\usepackage{setspace}
\usepackage{mathtools}
\usepackage[all]{xy}
\usepackage{verbatim}
\usepackage{enumerate}
\usepackage[usenames]{color}
\usepackage{varwidth}
\allowdisplaybreaks

\newtheorem{theorem}{Theorem}
\newtheorem*{theorem*}{Theorem}

\newtheorem{lemma}{Lemma}

\newtheorem{proposition}{Proposition}

\theoremstyle{definition}
\newtheorem{definition}{Definition}

\theoremstyle{remark}

\numberwithin{equation}{section}
\allowdisplaybreaks

\title[A REDUCIBILITY PROBLEM 
FOR EVEN UNITARY GROUPS: THE DEPTH ZERO CASE]{A REDUCIBILITY PROBLEM 
	FOR EVEN UNITARY GROUPS: THE DEPTH ZERO CASE}
\author{Subha Sandeep Repaka}
\address{Indian Institute of Science Education and Research Tirupati, Karakambadi Road, Rami Reddy Nagar, Mangalam, Tirupati, Andhra Pradesh-517507, INDIA}
\email{sandeep.repaka@gmail.com}

\begin{document}
	
	\date{\today}	
	\setcounter{tocdepth}{1}
	\date{\today}
	\keywords{}
	\subjclass[2010]{Primary 22E50, Secondary 11F70}
	
	\begin{abstract}
		We study a problem concerning parabolic induction in certain $p$-adic unitary groups. More precisely, for $E/F$ a quadratic extension of $p$-adic fields the associated unitary group $G=\mathrm{U}(n,n)$ contains a parabolic subgroup $P$ with Levi component $L$ isomorphic to $\mathrm{GL}_n$. Let $\pi$ be an irreducible supercuspidal representation of $L$ of depth zero. We use Hecke algebra methods to determine when the parabolically induced representation $\iota_P^G \pi$ is reducible.
	\end{abstract}
	
	\maketitle
	\tableofcontents
	
\section{Introduction}

A central feature of reductive $p$-adic groups is that there are irreducible representations that never occur as subrepresentations of parabolically induced representations. These are the supercuspidal representations. Suppose $\Pi$ is an irreducible representation of a reductive $p$-adic group $G$. By insights of Harish-Chandra and others it is known that there is a parabolic subgroup of $G$ and an irreducible supercuspidal representation $\pi$ of a Levi component of the parabolic subgroup such that $\Pi$ occurs in the representation obtained from $\pi$ via parabolic induction.\par

Thus, a core problem in $p$-adic representation theory is to understand when
and how parabolically induced representations decompose, especially when the
inducing representation is supercuspidal. The problem of when the parabolically induced representation decompose is what we study in a very special situation when $G=\mathrm{U}(n,n)$ is even unitary group over non- Archimedean local field $E$ and $\pi$ is a smooth irreducible supercuspidal  depth zero representation of Siegel Levi component $L \cong \mathrm{GL}_n(E)$ of Siegel parabolic subgroup $P$ of $G$. The terms $P, L, \pi, \mathrm{U}(n,n)$ are described in much detail later in the paper. We use Hecke algebra methods to determine when the parabolically induced representation $\iota_P^G \pi$ is reducible. Harish-Chandra tells us to  look not at an individual $\iota_P^G \pi$ but at the family $\iota_P^G (\pi \nu)$ as $\nu$ varies through the unramfied characters of $L \cong {\rm GL}_n(E)$. Note that the functor $\iota_P^G$ and unramified characters of $L$ are also described in greater detail later in the paper.\par  

The approach taken in solving the above problem in this paper is similar to the approach taken by Kutzko and Morris in \cite{MR2276353} for the split classical groups over non-Archimedean local fields, but the computations for the even unitary group $\mathrm{U}(n,n)$ over non-Archimedean local fields done in this paper are more involved than in \cite{MR2276353}.\par

Before going any further, let me describe how $\mathrm{U}(n,n)$ over  non-Archimedean local fields looks like. Let $E/F$ be a quadratic Galois extension of non-Archimedean local fields where char $F \neq 2$. Write $-$ for the non-trivial element of $\mathrm{Gal}(E/F)$. The group $G=\mathrm{U}(n,n)$ is given by

\[
\mathrm{U}(n,n)= \{g \in \mathrm{GL}_{2n}(E) \mid {^t}{\overline g}Jg=J\}
\]
for $J=
\begin{bmatrix}
0 & 1_n\\
1_n & 0 \\
\end {bmatrix}$ where each block is of size $n$ and for $g=(g_{ij})$ we write $\overline{g}= (\overline{g}_{ij})$. We write $\mathfrak{O}_E$ and $\mathfrak{O}_F$ for the ring of integers in $E$ and $F$ respectively. Similarly, $\mathbf{p}_E$ and $\mathbf{p}_F$ denote the maximal ideals in $\mathfrak{O}_E$ and $\mathfrak{O}_F$ and $k_E=\mathfrak{O}_E/\mathbf{p}_E$ and $k_F=\mathfrak{O}_F/\mathbf{p}_F$ denote the residue class fields of $\mathfrak{O}_E$ and $\mathfrak{O}_F$. Let $\left|k_F\right|= q =p^r$ for some odd prime $p$ and some integer $r \geqslant 1$.\par  

There are two kinds of extensions of $E$ over $F$. One is the unramified extension and the other one is the ramified extension. In the unramified case, we can choose uniformizers $\varpi_E,\varpi_F$ in $E, F$ such that $\varpi_E=\varpi_F$ so that we have $[k_E:k_F]=2 , \mathrm{Gal}(k_E/k_F) \cong \mathrm{Gal}(E/F)$. As $\varpi_E=\varpi_F$, so $\overline{\varpi}_E=\varpi_E$  since  $\varpi_F \in F$. As $k_F=\mathbb{F}_q$,  so $k_E= \mathbb{F}_{q^2}$ in this case. In the ramified case, we can choose uniformizers $\varpi_E,\varpi_F$ in $E, F$ such that $\varpi_E^2=\varpi_F$ so that we have $[k_E:k_F]=1,\mathrm{Gal}(k_E/k_F)=1$. As $\varpi_E^2=\varpi_F$, we can further choose $\varpi_E$ such that $\overline{\varpi}_E= -\varpi_E$. As $k_F=\mathbb{F}_q$, so $k_E= \mathbb{F}_{q}$ in this case. \par

We write $P$  for the Siegel parabolic subgroup of $G$. Write $L$ for the Siegel Levi component of $P$ and $U$ for the unipotent radical of $P$. Thus $P= L \ltimes U$ with  

\[
L= \Bigg\lbrace \begin{bmatrix}
a & 0\\
0 & {^t}{\overline {a}}{^{-1}}
\end{bmatrix} \mid a \in \mathrm{GL}_n(E) \Bigg \rbrace
\]

and

\[
U= \Bigg \lbrace \begin{bmatrix}
1_n & X \\
0  & 1_n 
\end{bmatrix} \mid  X \in \mathrm{M}_n(E), X + {^t}{\overline{X}}=0 \Bigg \rbrace.
\]

Let $\overline P= L \ltimes \overline U$ be the $L$-opposite of $P$ where 

\[
\overline{U}= \Bigg \lbrace \begin{bmatrix}
1_n & 0 \\
X  & 1_n 
\end{bmatrix} \mid  X \in \mathrm{M}_n(E), X + {^t}{\overline{X}}=0 \Bigg \rbrace.
\]\par

Let $K_0=\mathrm{GL}_n(\mathfrak{O}_E)$ and $K_1=1+\varpi_E \mathrm{M}_n(\mathfrak{O}_E)$. Note $K_1=1+\varpi_E \mathrm{M}_n(\mathfrak{O}_E)$ is the kernel of the surjective group homomorphism \[(g_{ij}) \longrightarrow (g_{ij} + \mathbf{p_E}) \colon \mathrm{GL}_n(\mathfrak{O}_E) \longrightarrow \mathrm{GL}_n(k_E)\] \par

Let

\[
\mathfrak{P}_{0,1}= \Bigg \lbrace
\begin{bmatrix}
a & 0\\ 
0 & {^t}{\overline{a}}{^{-1}}\\
\end{bmatrix} \mid a \in K_1=1 + \varpi_E \mathrm{M}_n(\mathfrak{O}_E)\Bigg \rbrace.
\]

Clearly, $\mathfrak{P}_{0,1} \cong K_1$. We say $\pi$ is a depth zero representation of Siegel Levi component $L$ of $P$ if $\pi^{\mathfrak{P}_{0,1}} \neq 0$. Let $(\rho, V)$ be a smooth representation of $H$ where $H$ is a subgroup of $G$. The smoothly induced representation from $H$ to $G$ is denoted by $Ind_H^G(\rho, V)$ or $Ind_H^G(\rho)$. Let us denote $c$-$Ind_H^G(\rho, V)$ or  $c$-$Ind_H^G(\rho)$ for smoothly induced compact induced representation from $H$ to $G$. The normalized induced representation from $P$ to $G$ is denoted by $\iota_P^G(\rho, V)$ or $\iota_P^G(\rho)$ where $\iota_P^G(\rho)=Ind_P^G(\rho \otimes \delta_P^{1/2})$. Note that here $\delta_P$ is a character of $P$ given by $\delta_P(p)=\|det(Ad \, p)|_{\mathrm{Lie} \, U}\|_F$ for $p \in P$ and $\mathrm{Lie} \, U$ is the Lie-algebra of $U$. We work with normalized induced representations rather than induced representations in this paper as results look more appealing.\par 

Write $G^{\circ}$ for the smallest subgroup of $G$ containing the compact open subgroups of $G$. We say a character $\nu \colon G \longrightarrow \mathbb{C}^{\times}$ is unramified if $\nu|_{G^{\circ}}=1$. Let the group of unramified characters of $G$ be denoted by $\mathrm{X}_{nr}(G)$.\par
 
\subsection{Question}
The question we answer in this paper is, given $\pi$ an irreducible supercuspidal representation of $L$ of depth zero, we look at the family of representations $\iota_P^G(\pi\nu)$  for $\nu \in \chi_{nr}(L)$ and we want to determine the set of such $\nu$ for which this induced representation is reducible for both ramified and unramified extensions. By general theory, this is a finite set.\par

As $\pi$ is an irreducible supercuspidal depth zero representation of $L$ which is isomorphic to $\mathrm{GL}_n(E)$, so $\pi|_{K_0}$ contains an irreducible representation $\rho_0$ of $K_0$ such that $\rho_0|_{K_1}$ is trivial. So $\rho_0$ can be viewed as an irreducible representation of $K_0/K_1\cong \mathrm{GL}_n(k_E)$  inflated to $K_0=\mathrm{GL}_n(\mathfrak{O}_E)$. The representation $\rho_0$ is supercuspidal by (a very special case of) A.1 Appendix \cite{MR1235019}.\par

We now define the Siegel parahoric subgroup $\mathfrak{P}$ of $G$ which is given by:

\begin{center}
	$\mathfrak{P} =
	\begin{bmatrix}
	\mathrm{GL}_n(\mathfrak{O}_E)    & \mathrm{M}_n(\mathfrak{O}_E) \\
	\mathrm{M}_n(\mathbf{p}_E)  & \mathrm{GL}_n(\mathfrak{O}_E) 
	\end{bmatrix}
	\bigcap \mathrm{U}(n,n).$\\
\end {center}

We have $\mathfrak{P}= (\mathfrak{P}\cap\overline{U})(\mathfrak{P}\cap L)(\mathfrak{P}\cap U) $(Iwahori factorization of $\mathfrak{P}$). Let us denote $(\mathfrak{P}\cap\overline{U})$ by $\mathfrak{P}_{-}$, $(\mathfrak{P}\cap U)$ by $\mathfrak{P}_{+}$, $(\mathfrak{P}\cap L)$ by $\mathfrak{P}_{0}$. Thus

	\begin{center}
		$ \mathfrak{P}_{0}=
		\Bigg\lbrace
		\begin{bmatrix}
		a & 0\\ 
		0 & {^t}{\overline{a}}{^{-1}}\\
		\end{bmatrix}\mid a \in \mathrm{GL}_n(\mathfrak{O}_E)
		\Bigg \rbrace$,
		\end {center}

		\begin{center}
			$\mathfrak{P}_{+}=
			\Bigg\lbrace
			\begin{bmatrix}
			1_n & X\\ 
			0 & 1_n\\
			\end{bmatrix} \mid X\in \mathrm{M}_n(\mathfrak{O}_E), X + {^t}{\overline{X}}=0
			\Bigg \rbrace$,
			\end {center}

               \begin{center}
	        	$\mathfrak{P}_{-}=
				\Bigg\lbrace
				\begin{bmatrix}
				1_n & 0\\ 
				X & 1_n\\
				\end{bmatrix} \mid X\in \mathrm{M}_n(\mathbf{p}_E), X + {^t}{\overline{X}}=0
				\Bigg \rbrace$.
				\end {center}\par
				
By Iwahori factorization of $\mathfrak{P}$ we have $\mathfrak{P}= (\mathfrak{P}\cap\overline{U})(\mathfrak{P}\cap L)(\mathfrak{P}\cap U)=\mathfrak{P}_{-}\mathfrak{P}_{0}\mathfrak{P}_{+}$. As $\rho_0$ is a representation of $K_0$, it can also be viewed as a representation of $\mathfrak{P}_{0}$. This is because $\mathfrak{P}_{0}\cong K_0$. In section \ref{sec_3}, we show that $\rho_0$ can be extended to a representation $\rho$ of $\mathfrak{P}$. If $j \in \mathfrak{P}$ has Iwahori factorization $j=j_{-}j_{0}j_{+}$ where $j_{-} \in \mathfrak{P}_{-}, j_{0} \in \mathfrak{P}_{0}, j_{+} \in \mathfrak{P}_{+}$ then $\rho$ is defined by $\rho(j)= \rho_0(j_0)$ for $j \in \mathfrak{P}$.\par

In section \ref{sec_4}, we compute the normalizer of $\mathfrak{P}$ in $G$ which is given by  $N_G(\mathfrak{P}_0)$ and it is independent of whether we are working in ramified or unramified case. The explicit calculation of $N_G(\mathfrak{P}_0)$ was not done in \cite{MR2276353}. By the work of Green or as a very special case of the Deligne-Lusztig construction, irreducible cuspidal representations of $\mathrm{GL}_n(k_E)$ are parametrized by the regular characters of degree $n$ extensions of $k_E$. We write $\tau_\theta$ for the irreducible cuspidal representation $\rho_0$ that corresponds to a regular character $\theta$.\par

Let $Z(L)$ denote the center of $L$. Hence 

\[
Z(L)=\Bigg\lbrace
\begin{bmatrix}
a1_n  & 0 \\
0    & \overline{a}^{-1}1_n \\
\end{bmatrix} \mid a \in E^{\times} \Bigg \rbrace.
\]\par

let \[ N_G(\rho_0)=\{ m \in N_G(\mathfrak{P}_0) \mid {\rho_0}\simeq \rho_0^m \}.\] 

We shall show in section \ref{Sec_5} that in the unramified case if $n$ is even then $N_G(\rho_0)= Z(L) \mathfrak{P}_0$ and if $n$ is odd then $N_G(\rho_0)= Z(L) \mathfrak{P}_0 \rtimes \langle J \rangle$. In the ramified case we show, if $n$ is odd then $N_G(\rho_0)= Z(L) \mathfrak{P}_0$ and if $n$ is even then $N_G(\rho_0)= Z(L) \mathfrak{P}_0 \rtimes \langle J \rangle$. Also the explicit calculation of $N_G(\rho_0)$ for the ramified and unramified cases was not done in \cite{MR2276353}. We use $N_G(\rho_0)$ to calculate the structure of the Hecke Algebra $\mathcal{H}(G, \rho)$ for unramified($n$ is odd) and ramified($n$ is even) cases in section \ref{sec_6}.\par

Let $w_0 =J$ and $w_1
=\begin{bmatrix}
0 & {\overline\varpi_E}^{-1}1_n\\
\varpi_E 1_n & 0
\end{bmatrix}$. Now 

\[ w_0 w_1=\zeta=
\begin{bmatrix}
\varpi_E 1_n & 0\\
0 & {\overline\varpi_E}^{-1}1_n
\end{bmatrix}.\]\par

The Hecke algebra $\mathcal{H}(G, \rho)$ in the unramified case and $n$ is odd is given by:  

\[
\mathcal{H}(G,\rho)= \left\langle \phi_i \colon G \to End_{\mathbb{C}}(\rho^{\vee})\; \middle| \;
\begin{varwidth}{\linewidth} 
$\phi_i$ is supported on 
$\mathfrak{P}w_i\mathfrak{P}$\\
$\,\text{and} \, \phi_i(pw_ip')= \rho^{\vee}(p)\phi_i(w_i)\rho^{\vee}(p')$\\
where $p,p' \in \mathfrak{P}, \, i=0,1$
\end{varwidth}
\right\rangle
\] where $\phi_i$ has support $\mathfrak{P} w_i \mathfrak{P}$  and $\phi_i$ satisfies the relation:

\begin{center}
	$\phi_i^2= q^{n}+ (q^{n}-1) \phi_i$ for $i=0,1$.
\end{center} \par

However, The Hecke algebra $\mathcal{H}(G, \rho)$ in the ramified case and $n$ is even is given by:  

\[
\mathcal{H}(G,\rho)= \left\langle \phi_i \colon G \to End_{\mathbb{C}}(\rho^{\vee})\; \middle| \;
\begin{varwidth}{\linewidth} 
$\phi_i$ is supported on 
$\mathfrak{P}w_i\mathfrak{P}$\\
$\,\text{and} \, \phi_i(pw_ip')= \rho^{\vee}(p)\phi_i(w_i)\rho^{\vee}(p')$\\
where $p,p' \in \mathfrak{P}, \, i=0,1$
\end{varwidth}
\right\rangle
\] where $\phi_i$ has support $\mathfrak{P} w_i \mathfrak{P}$  and $\phi_i$ satisfies the relation:

\begin{center}
	$\phi_i^2= q^{n/2}+ (q^{n/2}-1) \phi_i$ for $i=0,1$.
\end{center} \par

The approach taken in \cite{MR2276353} is by root sytems for the calculation of the Hecke algbera $\mathcal{H}(G, \rho)$  for the split classical groups.\par

Let 

\[
N_L(\rho_0)=\lbrace m \in N_L(\mathfrak{P}_0) \mid \rho_0^m \simeq \rho_0 \rbrace. 
\] \par

We use $N_L(\rho_0)$ to dermine the structure of the Hecke Algebra $\mathcal{H}(L, \rho_0)$ which is independent of whether we are in the ramified or unramfied case. Let $V$ be vector space corresponding to $\rho_0$. We shall show that there exists an element $\alpha$ in  $\mathcal{H}(L, \rho_0)$ such that $\text{supp}(\alpha)= \mathfrak{P}_0\zeta$ and $\alpha(\zeta)=1_{V^{\vee}}$. We show in section \ref{sec_7} that the Hecke algebra $\mathcal{H}(L, \rho_0)= \mathbb{C}[\alpha, \alpha^{-1}]$ which is isomorphic to the $\mathbb{C}$-algebra $\mathbb{C}[x,x^{-1}]$. Then we determine the structure of simple $\mathbb{C}[\alpha,\alpha^{-1}]$-modules in section \ref{sec_8}.\par 

The categories $\mathfrak{R}^{s_L}(L), \mathfrak{R}^{s}(G)$ where $s_L=[L,\pi]_L$ and $s= [L,\pi]_G$ are described later in this paper. The category $\mathcal{H}(G,\rho)-Mod$ is the category of $\mathcal{H}(G,\rho)$-modules and  $\mathcal{H}(L,\rho_0)-Mod$ be the category of $\mathcal{H}(L,\rho_0)$-modules. The functor $\iota_P^G$ was defined earlier. \par

For $s_L=[L,\pi]_L$, the representation $\pi \nu \in \mathfrak{R}^{s_L}(L)$. The functor $m_L \colon \mathfrak{R}^{s_L}(L) \longrightarrow  \mathcal{H}(L,\rho_0)-Mod$ given by $m_L(\pi \nu)= \mathrm{Hom}_{\mathfrak{P_0}}(\rho_0, \pi \nu)$ is an equivalance of categories. The representation $\pi \nu \in \mathfrak{R}^{s_L}(L)$  being irreducible, it corresponds to a simple $\mathcal{H}(L,\rho_0)$-module under the functor $m_L$. The simple $\mathcal{H}(L,\rho_0)$-module $m_L(\pi \nu)$ is calculated in section \ref{sec_9}. Let $f \in m_L(\pi \nu), \gamma \in \mathcal{H}(L,\rho_0)$ and $w \in W$. The action of $\mathcal{H}(L,\rho_0)$ on $m_L(\pi \nu)$ is given by $(\gamma. f)(w)= \int_L \pi(l) \nu(l)f(\gamma^{\vee}(l^{-1})w)dl$. Here $\gamma^{\vee}$ is defined on $L$ by $\gamma^{\vee}(l^{-1})=\gamma(l)^{\vee}$ for $l \in L$.  For $s= [L,\pi]_G$, the representation $\iota_P^G (\pi \nu) \in  \mathfrak{R}^{s}(G)$. The functor $m_G \colon \mathfrak{R}^{s}(G) \longrightarrow  \mathcal{H}(G,\rho)-Mod$ is given by $m_G( \iota_P^G (\pi \nu))= \mathrm{Hom}_{\mathfrak{P}}(\rho, \iota_P^G (\pi\nu))$ is an equivalance of categories. \par  

There is an algbera embedding $t_P \colon \mathcal{H}(L,\rho_0) \longrightarrow \mathcal{H}(G, \rho)$ described in section \ref{sec_9}. The map $t_P$ induces the  functor $(t_P)_* \colon  \mathcal{H}(L,\rho_L)-Mod \longrightarrow \mathcal{H}(G,\rho)-Mod$ which is given by:

For $M$ an $\mathcal{H}(L, \rho_0)$-module,

\[(t_P)_{*}(M)= \mathrm{Hom}_{\mathcal{H}(L, \rho_0)}(\mathcal{H}(G, \rho),M)\] where $\mathcal{H}(G, \rho)$ is viewed as a $\mathcal{H}(L, \rho_0)$-module via $t_P$. The action of $\mathcal{H}(G,\rho)$ on $(t_P)_*(M)$ is given by 

\[h'\psi(h_1)=\psi(h_1 h')\] where $\psi \in (t_P)_*(M), h_1, h' \in \mathcal{H}(G, \rho)$.\par.\par

From Corollary 8.4 in \cite{MR1643417}, the functors $m_L, m_G, Ind_P^G, (t_P)_*$  fit into the following commutative diagram:

\[
\begin{CD}
\mathfrak{R}^{[L,\pi]_G}(G)    @>m_G>>    \mathcal{H}(G,\rho)-Mod\\
@AInd_P^GAA                                    @A(t_P)_*AA\\
\mathfrak{R}^{[L,\pi]_L}(L)    @>m_L>>     \mathcal{H}(L,\rho_0)-Mod
\end{CD}
\] \par

We further show in section \ref{sec_9} that there is an algebra embedding $T_P \colon \mathcal{H}(L,\rho_0) \longrightarrow \mathcal{H}(G, \rho)$ given by $T_P(\phi)= t_P(\phi \delta_P^{-1/2})$ for $\phi \in \mathcal{H}(L, \rho_0)$ so that the following diagram commutes:

\[
\begin{CD}
\mathfrak{R}^{[L,\pi]_G}(G)    @>m_G>>    \mathcal{H}(G,\rho)-Mod\\
@A\iota_P^GAA                                    @A(T_P)_*AA\\
\mathfrak{R}^{[L,\pi]_L}(L)    @>m_L>>     \mathcal{H}(L,\rho_0)-Mod
\end{CD}
\]\par

For $M$ an $\mathcal{H}(L, \rho_0)$-module,

\[(T_P)_{*}(M)= \mathrm{Hom}_{\mathcal{H}(L, \rho_0)}(\mathcal{H}(G, \rho),M)\] where $\mathcal{H}(G, \rho)$ is viewed as a $\mathcal{H}(L, \rho_0)$-module via $T_P$. The action of $\mathcal{H}(G,\rho)$ on $(T_P)_*(M)$ is given by 

\[h'\psi(h_1)=\psi(h_1 h')\] where $\psi \in (T_P)_*(M), h_1, h' \in \mathcal{H}(G, \rho)$.\par

Define $g_i= q^{-n}\phi_i$ for $i=0,1$ in the unramified case where $n$ is odd and $g_i= q^{-n/2}\phi_i$ for $i=0,1$ in the ramified case where $n$ is even. We shall show in section \ref{sec_9} that $g_0*g_1= T_P(\alpha)$ for both the cases. From this result and Proposition 1.6 in \cite{MR2531913} Theorem \ref{the_1} follows.\par

As $Z(L)\mathfrak{P}_0= \coprod_{n \in \mathbb{Z}} \mathfrak{P}_0 \zeta^n$, so we can extend $\rho_0$ to a representation $\widetilde{\rho_0}$ of $Z(L)\mathfrak{P}_0$ via $\widetilde{\rho_0}(\zeta^k j)=\rho_0(j)$ for $j \in \mathfrak{P}_0, k \in \mathbb{Z}$. By the standard Mackey theory arguments, we show in the paper that $\pi$= $c$-$Ind_{Z(L)\mathfrak{P}_0}^L \widetilde{\rho_0}$ is a smooth irreducible supercuspidal depth zero representation of $L$. Also note that any arbitrary depth zero irreducible supercuspidal cuspidal representation of $L$ is an unramified twist of $\pi$. To that end, we will answer the question which we posed earlier in this paper and prove the following result.

\begin{theorem} \label{the_1}
	Let $G= \mathrm{U}(n,n)$. Let $P$ be the Siegel parabolic subgroup of $G$ and $L$ be the Siegel Levi component of $P$. Let $\pi$= $c$-$Ind_{Z(L)\mathfrak{P}_0}^L \widetilde{\rho_0}$ be a smooth irreducible supercuspidal depth zero representation of $L \cong \mathrm{GL}_n(E)$ where $\widetilde{\rho_0}(\zeta^k j)=\rho_0(j)$ for $j \in \mathfrak{P}_0, k \in \mathbb{Z}$ and  $\rho_0=\tau_{\theta}$ for some regular character $\theta$ of $l^{\times}$ with $[l:k_E]=n$ and $|k_F|= q$. Consider the family $\iota_P^G(\pi\nu)$ for $\nu \in \mathrm{X}_{nr}(L)$.
	\begin {enumerate}
	\item For $E/F$ is unramified, $\iota_P^G(\pi \nu)$ is reducible $\Longleftrightarrow n$ is odd, $\theta^{q^{n+1}}= \theta^{-q}$ and $\nu(\zeta) \in \{q^n,q^{-n},-1\}$.
	\item For $E/F$ is ramified, $\iota_P^G(\pi \nu)$ is reducible $\Longleftrightarrow n$ is even, $\theta^{q^{n/2}}=\theta^{-1}$ and $\nu(\zeta) \in \{q^{n/2},q^{-n/2},-1\}$.
\end{enumerate}
\end{theorem}

 Note that the composition series for our family of induced representations have length one or two. We induce from unramified twists of an irreducible unitary supercuspidal representation. If the unramified twist is unitary, then the induced representation is also unitary and so is a direct sum of its composition factors. If the unramified twist is not unitary and the induced representation is reducible, then there is a unique subrepresentation and a unique quotient.\par

 In \cite{MR1266747}, Goldberg computed the reducibility points of $\iota_{P}^{G}(\pi)$ by computing the poles of certain $L$-functions attached to the representation $\pi$ of $\mathrm{GL}_n(E)$. Note however that the base field $F$ is assumed to be of characteristic $0$ in \cite{MR1266747}, whereas we assumed characteristic of $F \neq 2$. In \cite{MR1266747} no restriction on the depth of the representation $\pi$ is there, while in this paper we have assumed depth of the representation $\pi$ to be zero. The final results obtained in \cite{MR1266747} are in terms of matrix coefficents of $\pi$ whereas our results are in terms of the unramified characters of $L$.\par

In \cite{heiermann_2011} and \cite{heiermann_2017}, Heiermann computed the structure of the Hecke algebras which we look at and much more and makes a connection with Langlands parameters. But his results are not explicit. They do not give the precise values of the parameters in the relevant Hecke algebras.\par  

Let $\eta$ be the automorphism of $G$ taking $x$ to ${}^t\overline{x}^{-1}$. In \cite{murnaghan_repka_1999}, Murnaghan and Repka computed the reducibility points of $\iota_{P}^{G}(\pi)$ for $\pi$ an irreducible unitary supercuspidal representation of $L$ and $\pi$ equivalent to $\pi \circ \eta$. Note that the base field $F$ is assumed to be of characteristic $0$ in \cite{murnaghan_repka_1999}, whereas we assumed characteristic of $F \neq 2$. In \cite{murnaghan_repka_1999} no restriction on the depth of the representation $\pi$ is there, while in this paper we have assumed depth of the representation $\pi$ to be zero.\par 

In \cite{Thomas2014} looked at a similar reducibility problem. Some of the technical results he derived (on Hecke algebras for finite groups of Lie type) are relevant to our calculations. We use them in this paper.

$\mathbf{Acknowledgments}$: This work is a part of author's thesis. He wishes
to thank his advisor Alan Roche from University of Oklahoma, USA for suggesting the problem studied in this work and guidance. Further, he wishes to thank C.G.Venketasubramanian and Arnab Mitra at IISER Tirupati, India for their interest and suggestions to improve the paper.\par

\section{Preliminaries}
\subsection{Bernstein Decomposition}

Let $G$ be the $F$-rational points of a reductive algebraic group defined over a non-Archimedean local field $F$. Let $(\pi, V)$ be an irreducible smooth representation of $G$. According to Theorem 3.3 in \cite{MR1486141}, there exists unique conjugacy class of cuspidal pairs $(L, \sigma)$ with the property that $\pi$ is isomorphic to a composition factor of $\iota_P^G\sigma$ for some parabolic subgroup $P$ of $G$. We call this conjugacy class of cuspidal pairs, the cuspidal support of $(\pi, V)$.\par

Given two cuspidal supports $(L_1, \sigma_1)$ and $(L_2, \sigma_2)$ of $(\pi, V)$, we say they are inertially equivalent if there exists $g \in G$ and $\chi \in \mathrm{X}_{nr}(L_2)$ such that $L_2= L_1^g$ and $\sigma_1^g\simeq \sigma_2\otimes\chi$. We write $[L, \sigma]_G$ for the inertial equivalence class or inertial support of $(\pi, V)$. Let $\mathfrak{B}(G)$ denote the set of inertial equivalence classes $[L, \sigma]_G$.\par

Let $\mathfrak{R}(G)$ denote the category of smooth representations of $G$. Let ${\mathfrak{R}}^s(G)$ be the full sub-category of smooth representations of $G$ with the property that $(\pi, V) \in ob({\mathfrak{R}}^s(G))\Longleftrightarrow$  every irreducible sub-quotient of $\pi$ has inertial support $s=[L, \sigma]_G$.\par

We can now state the Bernstein decomposition:

\[\mathfrak{R}(G)=\prod_{s \in \mathfrak{B}(G)}{\mathfrak{R}}^s(G).\]\par

\subsection{Types}

Let $G$ be the $F$-rational points of a reductive algebraic group defined over a non-Archimedean local field $F$. Let $K$ be a compact open subgroup of $G$. Let $(\rho, W)$ be an irreducible smooth representation of $K$ and $(\pi, V)$  be a  smooth representation of $G$. Let $V^{\rho}$ be the $\rho$-isotopic subspace of $V$.

\[
V^{\rho}= \sum\limits_{W'}W'
\] where the sum is over all $W'$ such that $(\pi|_K, W') \simeq (\rho, W)$.\par

Let $\mathcal{H}(G)$ be the space of all locally constant compactly supported functions $f \colon G \to \mathbb{C}$. This is a $\mathbb{C}$- algebra under convolution $*$. So for elements $f,g \in \mathcal{H}(G)$ we have 

\[(f * g)(x)= \int_G f(y)g(y^{-1}x)d\mu(y).\]

Here we have fixed a Haar measure $\mu$ on $G$. Let $(\pi, V)$ be a representation of $G$. Then  $\mathcal{H}(G)$ acts on $V$ via

\[hv=\int_G h(x)\pi(x)vd\mu(x)\] for $h \in \mathcal{H}(G), v \in V$. Let $e_\rho$ be the element in $\mathcal{H}(G)$ with support $K$ such that

\begin{center}
	$e_\rho(x)= \frac{\text{dim}\rho}{\mu(K)}tr_W(\rho(x^{-1})), x \in K.$\\
\end{center}

We have $e_\rho * e_\rho = e_\rho $ and $e_\rho V = V^{\rho}$ for any smooth representation $(\pi, V)$ of $G$. Let $\mathfrak{R}_\rho(G)$ be the full sub-category of $\mathfrak{R}(G)$ consisting of all representations $(\pi, V)$ where $V$ is generated by $\rho$-isotopic vectors. So $(\pi, V) \in \mathfrak{R}_\rho(G) \Longleftrightarrow  V = \mathcal{H}(G) * e_\rho V$. We now state the definition of a type.

\begin{definition}
	Let $ s \in \mathfrak{B}(G)$. We say that $(K, \rho)$ is an $s$-type in $G$ if $\mathfrak{R}_\rho(G)=
	\mathfrak{R}^s(G)$.
\end{definition}

\subsection{Hecke algebras}

Let $G$ be the $F$-rational points of a reductive algebraic group defined over a non-Archimedean local field $F$. Let $K$ be a compact open subgroup of $G$. Let $(\rho, W)$ be an irreducible smooth representation of $K$. Here we introduce the Hecke algebra $\mathcal{H}(G,\rho)$.

\[
\mathcal{H}(G,\rho)= \left\lbrace f \colon G \to End_{\mathbb{C}}(\rho^{\vee}) \; \middle|  \;
\begin{varwidth}{\linewidth}
supp($f$) is compact and \\
$f(k_1gk_2)= \rho^{\vee}(k_1)f(g)\rho^{\vee}(k_2)$\\
where $k_1,k_2 \in K, g \in G$
\end{varwidth}
\right \rbrace.
\]\par

Then $\mathcal{H}(G,\rho)$ is a $\mathbb{C}$-algebra with multiplication given by convolution $*$ with respect to some fixed Haar measure $\mu$ on $G$. So for elements $f,g \in \mathcal{H}(G)$ we have 
\[(f * g)(x)= \int_G f(y)g(y^{-1}x)d\mu(y).\] \par

The importance of types is seen from the following result. Let $\pi$ be a smooth representation in $\mathfrak{R}^{s}(G)$. Let $\mathcal{H}(G, \rho)- Mod$ denote the category of $\mathcal{H}(G, \rho)$-modules. If $(K, \rho)$ is an $s$-type then $ m_G \colon \mathfrak{R}^{s}(G) \longrightarrow \mathcal{H}(G, \rho)- Mod$  given by $m_G(\pi)= \mathrm{Hom}_K(\rho,\pi)$ is an equivalence of categories.\par

\subsection{Covers}

Let $G$ be the $F$-rational points of a reductive algebraic group defined over a non-Archimedean local field $F$. Let $K$ be a compact open subgroup of $G$. Let $(\rho, W)$ be an irreducible representation of $K$. Then we say $(K, \rho)$ is decomposed with respect to $(L,P)$ if the following hold:

\begin{enumerate}
	\item $K=(K\cap \overline U)( K \cap L)(K \cap U)$.
	\item $(K\cap \overline U),(K \cap U) \leqslant \text{ker}\rho$.
\end{enumerate}

Suppose $(K, \rho)$ is decomposed with respect  to $(L,P)$. We set $K_L=K \cap L$ and $\rho_L= \rho|_{K_L}$. We say an element $g \in G$ intertwines $\rho$ if $\mathrm{Hom}_{K^g \cap K} (\rho^g, \rho) \neq 0 $. Let $\mathfrak{I}_G(\rho)= \{ x \in G \mid x \,\text{intertwines} \, \rho \}$. We have the Hecke algebras $\mathcal{H}(G,\rho)$ and $\mathcal{H}(L,\rho_L)$. We write

\[
\mathcal{H}(G, \rho)_L =\{f \in \mathcal{H}(G, \rho) \mid \text{supp} (f) \subseteq KLK \}. 
\]\par

We recall some results and constructions from \cite{MR1643417}. These allow us to transfer questions about parabolic induction into questions concerning the module theory of appropriate Hecke algebras. 

\begin {proposition}[Bushnell and Kutzko, Proposition 6.3  \cite{MR1643417}]\label{pro_1}
Let $(K, \rho)$ decompose with respect to $(L,P)$ .Then

\begin{enumerate}
	\item $\rho_L$ is irreducible.
	\item $\mathfrak{I}_L(\rho_L)= \mathfrak{I}_G(\rho) \cap L$.
	\item There is an embedding  $T \colon \mathcal{H}(L, \rho_L) \longrightarrow \mathcal{H}(G, \rho)$ such that if $f \in \mathcal{H}(L, \rho_L)$ has support $K_LzK_L$ for some $ z\in L$, then $T(f)$ has support $KzK$.
	\item The map T induces an isomorphism of vector spaces:
	
	\begin{center}
		$\mathcal{H}(L, \rho_L) \xrightarrow{\simeq} \mathcal{H}(G, \rho)_L.$\\
	\end{center}
	
	\end {enumerate}
	
\end{proposition}

\begin{definition}
	
	An element $z \in L$ is called $(K,P)$-positive element if:
	
	\begin{enumerate}
		\item $z(K\cap \overline{U})z^{-1} \subseteq K \cap \overline{U}.$
		\item $z^{-1}(K \cap U)z \subseteq K \cap U.$
		\end {enumerate}
	\end{definition}
	
	\begin{definition}
		An element $z \in L$ is called strongly $(K,P)$-positive element if:
		
		\begin{enumerate}
			\item $z$ is $(K,P)$ positive.
			\item $z$ lies in center of $L$.
			\item For any compact open subgroups $K$ and $K'$ of $U$ there exists $m\geqslant 0$ and $m \in \mathbb{Z}$ such that $z^mKz^{-m} \subseteq K'$.
			\item For any compact open subgroups $K$ and $K'$ of $U$ there exists $m\geqslant 0$ and $m \in \mathbb{Z}$ such that $z^{-m}Kz \subseteq K'$. 
		\end{enumerate}

	\end{definition}\par

	\begin{proposition}[Bushnell and Kutzko, Lemma 6.14 \cite{MR1643417}, Proposition 7.1, \cite{MR1643417}]\label{pro_2}
		
		Strongly $(K,P)$-positive elements exist and given a strongly positive element $z \in L$ , there exists a unique function $\phi_z \in \mathcal{H}(L, \rho_L)$ with support $K_LzK_L$ such that $\phi_z(z)$ is identity function in $End_{\mathbb{C}}(\rho_L^{\vee})$.
	\end{proposition}
	
	\[
	\mathcal{H}^{+}(L,\rho_L)= \left\lbrace f \colon L \to End_{\mathbb{C}}(\rho_L^{\vee}) \; \middle|  \;
	\begin{varwidth}{\linewidth}
	supp($f$) is compact and consists\\ 
	of strongly $(K,P)$-positive elements \\ 
	and $f(k_1lk_2)= \rho_L^{\vee}(k_1)f(l)\rho_L^{\vee}(k_2)$\\
	where $k_1,k_2 \in K_L, l \in L$
	\end{varwidth}
	\right \rbrace.
	\]
	
	The isomorphism of vector spaces $T \colon \mathcal{H}(L,\rho_L) \longrightarrow \mathcal{H}(G,\rho)_L$ restricts to an embedding of algebras:

	\begin{center}
		$T^+ \colon \mathcal{H}^{+}(L,\rho_L) \longrightarrow \mathcal{H}(G,\rho)_L \hookrightarrow \mathcal{H}(G,\rho)$.\\
	\end{center}\par

	\begin{proposition}[Bushnell and Kutzko, Theorem 7.2.i \cite{MR1643417}]\label{pro_3}
		The embedding $T^+$ extends to an embedding of algebras \\ $t \colon \mathcal{H}(L,\rho_L)\longrightarrow \mathcal{H}(G,\rho) \Longleftrightarrow T^+(\phi_z)$ is invertible for some strongly $(K,P)$-positive element $z$, where  $\phi_z \in \mathcal{H}(L,\rho_L)$ has support $K_LzK_L$ with $\phi_z(z)=1$.
	\end{proposition}
	
	\begin{definition}
		Let $L$ be a proper Levi subgroup of $G$. Let $K_L$ be a compact open subgroup of $L$ and $\rho_L$ be an irreducible smooth representation of $K_L$. Let $K$ be a compact open subgroup of $G$ and $\rho$ be an irreducible, smooth representation of $K$. Then we say $(K, \rho)$ is a $G$-cover of $(K_L, \rho_L)$
		if
		
		\begin {enumerate}
		\item The pair $(K, \rho)$ is decomposed with respect to $(L,P)$ for every parabolic subgroup $P$ of $G$ with Levi component $L$.
		\item $K \cap L= K_L$  and $\rho|_L \simeq \rho_L$.
		\item The embedding $T^+ \colon \mathcal{H}^{+}(L,\rho_L) \longrightarrow \mathcal{H}(G,\rho)$ extends to an embedding of algebras $ t \colon \mathcal{H}(L,\rho_L) \longrightarrow \mathcal{H}(G,\rho)$.
	\end{enumerate}
\end{definition}\par

\begin{proposition}[Bushnell and Kutzko, Theorem 8.3 \cite{MR1643417}]\label{Types_Covers}
	Let $s_L=[L,\pi]_L \in \mathfrak{B}(L)$ and $s= [L,\pi]_G \in \mathfrak{B}(G)$ . Say $(K_L,\rho_L)$ is an $s_L$-type and $(K,\rho)$ is a $G$-cover of $(K_L,s_L)$. Then $(K,\rho)$ is an $s$-type.
\end{proposition}\par

Note that in this paper $\rho_L= \rho_0, K= \mathfrak{P}, K_L= \mathfrak{P}_0$. Recall the categories $\mathfrak{R}^{s_L}(L), \mathfrak{R}^{s}(G)$ where $s_L=[L,\pi]_L$ and $s= [L,\pi]_G$. Also recall $\mathcal{H}(G,\rho)-Mod$ is the category of $\mathcal{H}(G,\rho)$-modules. Let $\mathcal{H}(L,\rho_L)-Mod$ be the category of $\mathcal{H}(L,\rho_L)$-modules. The functors $\iota_P^G, m_G$ were defined earlier. Let $\pi \in \mathfrak{R}^{s_L}(L)$. Then the functor $m_L \colon \mathfrak{R}^{s_L}(L) \longrightarrow  \mathcal{H}(L,\rho_L)-Mod$ is given by $m_L(\pi)= \mathrm{Hom}_{K_L}(\rho_L, \pi)$. The functor $(T_P)_* \colon  \mathcal{H}(L,\rho_L)-Mod \longrightarrow \mathcal{H}(G,\rho)-Mod$ is defined later in this paper.\par

The importance of covers is seen from the following commutative diagram  which we will use in answering the question which we posed earlier in this paper.

\[
\begin{CD}
\mathfrak{R}^{s}(G)    @>m_G>>    \mathcal{H}(G,\rho)-Mod\\
@A\iota_P^GAA                                    @A(T_P)_*AA\\
\mathfrak{R}^{s_L}(L)    @>m_L>>     \mathcal{H}(L,\rho_L)-Mod
\end{CD}
\]\par

\subsection{Depth zero supercuspidal representations}

Suppose $\tau$ is an irreducible cuspidal representation of $\mathrm{GL}_n(k_E)$ inflated to a representation of $\mathrm{GL}_n(\mathfrak{O}_E)= K_0$. Then let $\widetilde{K_0}= ZK_0$ where $Z=Z(\mathrm{GL}_n(E))=\{\lambda \,1_n\mid \lambda \in E^{\times}\}$. As any element of $E^{\times}$ can be written as $u\varpi_E^n$ for some $u \in \mathfrak{O}_E^{\times}$ and $m \in \mathbb{Z}$. So in fact, $\widetilde{K_0}= <\varpi_E 1_n>K_0$. \par

Let $(\pi, V)$ be a representation of $\mathrm{GL}_n(E)$ and $1_V$ be the identity linear transformation of $V$. As $\varpi_E 1_n \in Z$, so $\pi(\varpi_E 1_n)= \omega_{\pi}(\varpi_E 1_n) 1_V$ where $\omega_{\pi}\colon Z \longrightarrow \mathbb{C}^{\times}$ is the central character of $\pi$.\par

Let $\widetilde\tau$ be a representation of $\widetilde K_0$ such that:
\begin{enumerate} 
	\item $\widetilde\tau(\varpi_E 1_n)=\omega_{\pi}(\varpi_E 1_n) 1_V,$
	\item $\widetilde\tau|_{K_0}= \tau.$
\end{enumerate}
\par

Say $\omega_{\pi}(\varpi_E 1_n)= z$ where $z \in \mathbb{C^{\times}}$. Now call $\widetilde\tau = \widetilde\tau_z$. We have extended $\tau$ to $\widetilde\tau_z$ which is a representation of $\widetilde K_0$, so that $Z$ acts by $\omega_\pi$. Hence $\pi|_{\widetilde K_0} \supseteq\widetilde \tau_z$ which implies that $\mathrm{Hom}_{\widetilde K_0} (\widetilde \tau_z, \pi|_{\widetilde K_0}) \neq 0$.\par

By Frobenius reciprocity for induction from open subgroups, 

\begin{center}
	$\mathrm{Hom}_{\widetilde K_0} (\widetilde \tau_z, \pi|_{\widetilde K_0}) \simeq \mathrm{Hom}_{\mathrm{GL}_n(E)}(c$-$Ind_{\widetilde K_0}^{\mathrm{GL}_n(E)} \widetilde\tau_z, \pi)$.
\end{center}

Thus $\mathrm{Hom}_{\mathrm{GL}_n(E)}(c$-$Ind_{\widetilde K_0}^{\mathrm{GL}_n(E)} \widetilde\tau_z, \pi) \neq 0$. So there exists a non-zero $\mathrm{GL}_n(E)$-map from $c$-$Ind_{\widetilde K_0}^G \widetilde\tau_z$ to $\pi$. As $\tau$ is cuspidal representation, using Cartan decompostion and Mackey's criteria we can show that $c$-$Ind_{\widetilde K_0}^{\mathrm{GL}_n(E)} \widetilde\tau_z$ is irreducible. So $\pi \simeq c$-$Ind_{\widetilde K_0}^{\mathrm{GL}_n(E)} \widetilde\tau_z$. As $c$-$Ind_{\widetilde K_0}^{\mathrm{GL}_n(E)} \widetilde\tau_z$ is irreducible supercuspidal representation of $\mathrm{GL}_n(E)$ of depth zero, so $\pi$ is irreducible supercuspidal representation of $\mathrm{GL}_n(E)$ of depth zero.\par

Conversely, let $\pi$ is a irreducible, supercuspidal, depth zero representation of  $\mathrm{GL}_n(E)$. So $\pi^{K_1} \neq \{0\}$. Hence $\pi|_{K_1} \supseteq 1_{K_1}$, where $1_{K_1}$ is trivial representation of $K_1$. This means $\pi|_{K_0} \supseteq \tau$, where $\tau$ is an irreducible representation of $K_0$ such that $\tau|_{K_1} \supseteq 1_{K_1}$. So $\tau$ is trivial on $K_1$. So $\pi|_{K_0}$ contains an irreducible representation $\tau$ of $K_0$ such that $\tau|_{K_1}$ is trivial. So $\tau$ can be viewed as an irreducible representation of $K_0/K_1\cong \mathrm{GL}_n(k_E)$  inflated to $K_0=\mathrm{GL}_n(\mathfrak{O}_E)$. The representation $\tau$ is cuspidal by (a very special case of) A.1 Appendix \cite{MR1235019}.\par

So we have the following bijection of sets:

\[
\left\lbrace 
\begin{varwidth}{\linewidth}
Isomorphism classes of irreducible \\ cuspidal 
representations of $\mathrm{GL}_n(k_E)$
\end{varwidth}
\right \rbrace		
\times \mathbb{C^{\times}} \longleftrightarrow 
\left\lbrace 
\begin{varwidth}{\linewidth}
Isomorphism classes \\of irreducible  \\
supercuspidal \\representations of \\ 
$\mathrm{GL}_n(E)$ of depth zero

\end{varwidth}
\right \rbrace.
\]
\par

\[(\tau, z)  \xrightarrow{\hspace*{6cm}}  c-Ind_{\widetilde K_0}^{\mathrm{GL}_n(E)} \widetilde\tau_z \]

\[(\tau, \omega_\pi(\varpi_E 1_n)) \xleftarrow{\hspace*{6cm}}  \pi \]\par

From now on we denote the representation $\tau$ by $\rho_0$. So $\rho_0$ is an irreducible cuspidal representation of $\mathrm{GL}_n(k_E)$ inflated to $K_0=\mathrm{GL}_n(\mathfrak{O}_E)$. 

\section{Representation $\rho$ of $\mathfrak{P}$} \label{sec_3}

 Let $V$ be the vector space associated with $\rho_0$. Now $\rho_0$ is extended to a map $\rho$ from $\mathfrak{P}$ to $GL(V)$ as follows. By Iwahori factorization, if $j \in \mathfrak{P}$ then $j$ can be written as $j_{-}j_{0}j_{+}$, where $j_{-}\in \mathfrak{P}_{-}, j_{+}\in \mathfrak{P}_{+}, j_{0}\in \mathfrak{P}_{0}$. Now the map $\rho$ on $\mathfrak{P}$ is defined as $\rho(j)= \rho_0(j_0)$.\par 
					
\begin{proposition}\label{pro_11}
$\rho$ is a homomorphism from $\mathfrak{P}$ to $GL(V)$. So $\rho$  becomes a representation of $\mathfrak{P}$.
\end{proposition}
					
\begin{proof}
Recall
						
\[
\mathfrak{P}_{0,1}= \Bigg \lbrace
\begin{bmatrix}
a & 0\\ 
0 & {^t}{\overline{a}}{^{-1}}\\
\end{bmatrix} \mid a \in K_1=1 + \varpi_E \mathrm{M}_n(\mathfrak{O}_E)\Bigg \rbrace.
\]

Clearly, $\mathfrak{P}_{0,1} \cong K_1$. Now let us define $\mathfrak{P}_{1}= \mathfrak{P}_{-}\mathfrak{P}_{0,1} \mathfrak{P}_{+}$. We can observe clearly that $\mathfrak{P}$ is a subgroup of $\mathrm{U}(n,n)\cap \mathrm{GL}_{2n}(\mathfrak{O}_E)$. We have the following group homomorphism:

\begin{center}
	$\phi \colon \mathfrak{P} \xrightarrow{mod \, \mathbf{p}_E} P(k_E)$. 
\end{center}

Here $P(k_E)$ is the Siegel parabolic subgroup of $\{ g\in \mathrm{GL}_{2n}(k_E)\mid {^t}{\overline g}Jg=J\}$. Now $P(k_E)=L(k_E)\ltimes U(k_E)$, where $L(k_E),U(k_E) $ are the Levi component and unipotent radical of the Siegel parabolic subgroup respectively.

\[
L(k_E)
=\Bigg\lbrace \begin{bmatrix}
a & 0\\
0 & {^t}{\overline{a}}{^{-1}}\\
\end{bmatrix}\mid a \in \mathrm{GL}_n(k_E)\Bigg \rbrace,
\]
\[
U(k_E)
=\Bigg\lbrace \begin{bmatrix}
1_n & X\\
0 &  1_n\\
\end{bmatrix}\mid X \in \mathrm{M}_n(k_E), X+^t\overline{X}=0\Bigg \rbrace.
\]

Observe that $\phi$ is a surjective homomorphism. Now let us find the inverse image of $U(k_E)$. Let $j \in \mathfrak{P}$ and $j=j_{-}j_{0}j_{+}$ be the Iwahori factorization of $j$, where $j_{+} \in \mathfrak{P}_{+}, j_{-} \in \mathfrak{P}_{-}, j_{0} \in \mathfrak{P}_{0}$. So $\phi(j) \in U(k_E) \Longleftrightarrow j_{0} \in \mathfrak{P}_{0,1}$. Therefore $\mathfrak{P}_{1}$ is the inverse image of $U(k_E)$ under $\phi$. So we have $\mathfrak{P}\diagup\mathfrak{P}_1 \cong P(k_E)\diagup U(k_E) \cong L(k_E) \cong \mathrm{GL}_n(k_E)$. As $\rho(j)= \rho_0(j_0)$, so $\rho$ is a representation of $\mathfrak{P}$ which is lifted from representation $\rho_0$ of $\mathfrak{P}_0$ that is trivial on $\mathfrak{P}_1$.
						
\end{proof}

\section{Calculation of $N_G(\mathfrak{P}_0)$} \label{sec_4}

We set $G=\mathrm{U}(n,n)$. To describe $\mathcal{H}(G,\rho)$ we need to determine $N_G(\rho_0)$ which is given by
\[ N_G(\rho_0)=\{ m \in N_G(\mathfrak{P}_0) \mid {\rho_0}\simeq \rho_0^m \}.\]  

Further, to find out $N_G(\rho_0)$ we need to determine $N_G(\mathfrak{P}_0)$. To that end we shall calculate $N_{\mathrm{GL}_n(E)}(K_0)$. Let $Z= Z(\mathrm{GL}_n(E))$. So $Z=\{\lambda 1_n \mid \lambda \in E^{\times} \}$.

\begin{lemma}\label{Normalizer_of_K_0_in_GL_n(E)} 
	$N_{\mathrm{GL}_n(E)}(K_0)= K_0Z$. 
\end{lemma}
\begin{proof}
	This follows from the Cartan decomposition by a direct matrix calculation.
\end{proof}\par

From now on let us denote $K_0$ by $K$. Now let us calculate $N_G(\mathfrak{P}_0)$. Note that 
$J=\begin{bmatrix}
0  & 1_n \\
1_n &  0 \\
\end{bmatrix} \in G$. Indeed, $J \in N_G(\mathfrak{P}_0)$. The center $Z(\mathfrak{P}_0)$ of $\mathfrak{P}_0$ is given by

\[
Z(\mathfrak{P}_0)=\Bigg\lbrace
\begin{bmatrix}
u1_n  & 0 \\
0    & \overline{u}^{-1}1_n \\
\end{bmatrix} \mid u \in \mathfrak{O}_E^{\times} \Bigg \rbrace.
\]\par

Recall the center $Z(L)$ of $L$ is given by

\[
Z(L)=\Bigg\lbrace
\begin{bmatrix}
a1_n  & 0 \\
0    & \overline{a}^{-1}1_n \\
\end{bmatrix} \mid a \in E^{\times} \Bigg \rbrace.
\]

\begin{proposition}\label{pro_12}
	$N_G(\mathfrak{P}_0)=\left<\mathfrak{P}_0 Z(L), J \right>= \mathfrak{P}_0Z(L) \rtimes \left<J \right>$.
\end{proposition}

\begin{proof}
	It easy to see that $N_G(\mathfrak{P}_0)\leqslant N_G(Z(\mathfrak{P}_0))$. Now suppose
	$g= 
	\begin{bmatrix}
	A & B\\
	C & D\\
	\end{bmatrix} \in N_G(Z(\mathfrak{P}_0))$, where $A,B,C,D \in \mathrm{M}_n(E)$. Let us choose $u \in \mathfrak{O}_E^{\times}$ such that $u \neq \overline{u}^{-1}$. Now such a $u$ exists in $\mathfrak{O}_E^{\times}$. Because if $u=\overline{u}^{-1}$ for all $u \in \mathfrak{O}_E^{\times}$ then $\overline{u}= u^{-1}$ for all $u \in \mathfrak{O}_E^{\times}$. But $\mathfrak{O}_E^{\times} \cap F^{\times}= \mathfrak{O}_F^{\times}$. Therefore $u= u^{-1}$ for all $u \in \mathfrak{O}_F^{\times}$ or $u^2=1$ for all $u \in \mathfrak{O}_F^{\times}$ which is a contradiction.

	\begin{center}
		As $\begin{bmatrix}
		A & B\\
		C & D
		\end{bmatrix} \in N_G(Z(\mathfrak{P}_0)),$
	\end{center}
	
	\[
	\begin{bmatrix}
	A & B\\
	C & D
	\end{bmatrix}
	\begin{bmatrix}
	u1_n & 0\\
	0   & \overline{u}^{-1}1_n
	\end{bmatrix}
	\begin{bmatrix}
	A & B\\
	C & D
	\end{bmatrix}^{-1}=
	\begin{bmatrix}
	v1_n & 0\\
	0   & \overline{v}^{-1}1_n
	\end{bmatrix}
	\] for some $v \in \mathfrak{O}_E^{\times}$. The left and right hand sides must have the same eigenvalues. So $u$= $v$ or $\overline{v}^{-1}$. Let $u=v$. Then we have
	
	\begin{center}
		$\begin{bmatrix}
		A & B\\
		C & D\\
		\end{bmatrix}
		\begin{bmatrix}
		u1_n & 0\\
		0   & \overline{u}^{-1}1_n\\
		\end{bmatrix}
		\begin{bmatrix}
		A & B\\
		C & D\\
		\end{bmatrix}^{-1}$=
		$\begin{bmatrix}
		v1_n & 0\\
		0   & \overline{v}^{-1}1_n\\
		\end{bmatrix}$
	\end{center}
	
	\begin{center}
		$\Longrightarrow$
		$\begin{bmatrix}
		Au & B\overline{u}^{-1}\\
		Cu & D\overline{u}^{-1}\\
		\end{bmatrix}=
		\begin{bmatrix}
		Av & Bv\\
		C\overline{v}^{-1} & D\overline{v}^{-1}\\
		\end{bmatrix}.$
	\end{center}\par
	
	As $u=v$, so $Au=Av, D\overline{u}^{-1}= D\overline{v}^{-1}$. Now as $u \neq \overline{v}^{-1}$ (i.e $v \neq \overline{u}^{-1}$), from the above matrix relation we can see that $B\overline{u}^{-1}= Bv$, $Cu=C\overline{v}^{-1}$ for arbitrary matrices $B$ and $C$. So this would imply that $B=C=0$. In a similar way, we can show that if $u= \overline{v}^{-1}$ then $A=D=0$. Hence any element of $N_G(Z(\mathfrak{P}_0))$ is of the form 
	$\begin{bmatrix}
	A & 0\\
	0 & D
	\end{bmatrix}$ or
	$\begin{bmatrix}
	0 & B\\
	C & 0
	\end{bmatrix}$ with $A,B,C,D \in \mathrm{GL}_n(E)$. As $N_G(\mathfrak{P}_0) \leqslant 
	N_G(Z(\mathfrak{P}_0))$, so any element which normalizes $\mathfrak{P}_0$ is also of the form 
	$\begin{bmatrix}
	A & 0\\
	0 & D
	\end{bmatrix}$ or
	$\begin{bmatrix}
	0 & B\\
	C & 0
	\end{bmatrix}$ with $A,B,C,D \in \mathrm{GL}_n(E)$.
	
	\begin{center}
		If 
		$\begin{bmatrix}
		A & 0\\
		0 & D
		\end{bmatrix}$ normalizes $\mathfrak{P}_0$ then
	\end{center}
	\begin{center}
		$\begin{bmatrix}
		A & 0\\
		0 & D
		\end{bmatrix}
		\begin{bmatrix}
		a & 0 \\
		0 & {^t}{\overline a}{^{-1}}
		\end{bmatrix}
		\begin{bmatrix}
		A^{-1} & 0   \\
		0      & D^{-1}
		\end{bmatrix}  \in \mathfrak{P}_0$ for all $a \in K$.\\
	\end{center}
	
	\begin{center}
		$\Longrightarrow
		\begin{bmatrix}
		AaA^{-1}  &  0 \\
		0         &  D{^t}{\overline a}{^{-1}}D^{-1}\\
		\end{bmatrix} \in \mathfrak{P}_0$ for all $a \in K$.
	\end{center}\par
	
	Hence $AaA^{-1}, D{^t}{\overline a}{^{-1}}D^{-1} \in K $ for all $a \in K$. So this implies that $A,D \in N_{\mathrm{GL}_n(E)}(K)=ZK=KZ$ from lemma \ref{Normalizer_of_K_0_in_GL_n(E)} and also ${^t}{\overline{(AaA^{-1})}}{^{-1}}= D{^t}{\overline a}{^{-1}}D^{-1}$ for all $a \in K$. If
	${^t}{\overline{(AaA^{-1})}}{^{-1}}= D{^t}{\overline a}{^{-1}}D^{-1}$ for all $a \in K$ then  ${^t}{\overline A}{^{-1}}  {^t}{\overline a}{^{-1}}{^t}{\overline A}= D{^t}{\overline a}{^{-1}}D^{-1}$ for all $a \in K \Longrightarrow A= {^t}{\overline D}{^{-1}}$ (i.e $D= {^t}{\overline A}{^{-1}}$). And as $A \in ZK$, so $A=zk$ for some $z \in Z, k \in K$. Hence
	\begin{center}
		$\begin{bmatrix}
		A & 0\\
		0 & D
		\end{bmatrix}=
		\begin{bmatrix}
		zk & 0\\
		0  & {^t}{\overline{(zk)}}{^{-1}}
		\end{bmatrix}$.
	\end{center}\par
	
	Similarly, we can show that if 
	$\begin{bmatrix}
	0 & B\\
	C & 0
	\end{bmatrix}$ normalizes $\mathfrak{P}_0$ then 
	\begin{center}
		$\begin{bmatrix}
		0 & B\\
		C & 0
		\end{bmatrix}=
		\begin{bmatrix}
		0 & z'k'\\
		{^t}{\overline{(z'k')}}{^{-1}} & 0
		\end{bmatrix}$ for some $z' \in Z, k' \in K$.
	\end{center}\par
	
	If 
	$\begin{bmatrix}
	A & 0\\
	0 & D
	\end{bmatrix} \in  N_G(\mathfrak{P}_0)$, we have shown that it looks like
	$\begin{bmatrix}
	zk & 0\\
	0 &  {^t}{\overline {(zk)}}{^{-1}}
	\end{bmatrix}$ and if 
	$\begin{bmatrix}
	0 & B\\
	C & 0
	\end{bmatrix} \in  N_G(\mathfrak{P}_0)$, we have shown that it looks like 
	$\begin{bmatrix}
	0 & z'k'\\
	{^t}{\overline {(z'k')}}{^{-1}} & 0
	\end{bmatrix}$ where $z,z' \in Z$, $k,k' \in K$. We know that $J \in N_G(\mathfrak{P}_0)$ and as 
	$\begin{bmatrix}
	0 & B\\
	C & 0
	\end{bmatrix}J=
	\begin{bmatrix}
	B & 0\\
	0 & C
	\end{bmatrix}$, so \[N_G(\mathfrak{P}_0)=\left< J; 
	\begin{bmatrix}
	zk & 0\\
	0  & {^t}{\overline{(zk)}}{^{-1}}\\
	\end{bmatrix} \mid z\in Z, k \in K \right>=\mathfrak{P}_0Z(L) \rtimes \left<J \right>\].
\end{proof}

\section{Calculation of $N_G(\rho_0)$} \label{Sec_5}

\subsection{Unramified case:}

We now calculate $N_G(\rho_0)$ in the unramified case. This will help in determining the structure of $\mathcal{H}(G,\rho)$. To do that, first we need to understand $\rho_0$. As $\rho_0$ is an irreducible cuspidal representation of ${\rm GL}_n(k_E)$, to understand it, we need the classification of the irreducible cuspidal representations of ${\rm GL}_n(k_E)$. 
This can be deduced from the Green parametrization \cite{green} or is a very special case of the Deligne-Lusztig construction. Let $l/k_E$ be a field extension of degree $n$. We set $\Gamma = {\rm Gal}(l/k_E)$.
 
Let 
\[
(l^\times)^\vee = {\rm Hom}(l^\times, \mathbb{C}^\times).
\]
Clearly, $\Gamma$ acts on $(l^\times)^\vee$ via
\[
\theta^\gamma (x) = \theta({}^\gamma x), \quad \theta \in (l^\times)^\vee, \,\,\gamma \in \Gamma, \,\,x \in l^\times.
\]     
We write $(l^\times)^\vee_{{\rm reg}}$ for the group of regular characters of $l^\times$ with respect to this action, that is, 
characters $\theta$ such that ${\rm Stab}_\Gamma (\theta)= \{ 1 \}$. 
We also write $l^\times_{\rm reg}$ for the regular elements in $l^\times$, that is, elements $x$ such that 
${\rm Stab}_\Gamma (x)= \{ 1 \}$. 
The set of $\Gamma$-orbits on  
$(l^\times)^\vee_{{\rm reg}}$ is then in canonical bijection with the set ${\rm Irr}_{\rm cusp} \, {\rm GL}_n(k_E)$ of equivalence 
classes of irreducible cuspidal representations of ${\rm GL}_n(k_E)$:
\begin{align*} 
\Gamma \backslash (l^\times)^\vee_{{\rm reg}}  \,\,\,     
&\longleftrightarrow \,\,\,  {\rm Irr}_{\rm cusp} \, {\rm GL}_n(k_E)  \\
\theta &\longleftrightarrow \tau_\theta.
\end{align*} 
The bijection is specified by a character relation         
\[
\tau_\theta (x) = c  \sum_{\gamma \in \Gamma}    \theta^\gamma(x), \quad x \in l^\times_{\rm reg}, 
\]
for a certain constant $c$ that is independent of $\theta$ and $x$. We denote $\rho_0$ by $\tau_{\theta}$.\par

Note that we have $k_E=\mathbb{F}_{q^2}$. So $l= \mathbb{F}_{q^{2n}}$.\par

As $\Gamma=\mathrm{Gal}(l/k_E)$, $\Gamma$ is generated by the Frobenius map $\Phi$ given by $\Phi(\lambda)=\lambda^{q^2}$ for $\lambda \in l$. Here $\Phi^n(\lambda)= \lambda^{q^{2n}}=\lambda$ (since $l^{\times}$ is a cyclic group of order $q^{2n}-1$) $\Longrightarrow \Phi^n= 1$. \par

Note that for two regular characters $\theta$ and $\theta'$ we have $\tau_{\theta}\simeq\tau_{\theta'} \Longleftrightarrow$ there exists $\gamma \in \Gamma$ such that $\theta^{\gamma}= \theta'$. \par

As we are in the unramified case, so $\mathrm{Gal}(k_E/k_F) \cong \mathrm{Gal}(E/F)$. Let $\iota\colon \mathrm{GL}_n(k_E) \longrightarrow \mathrm{GL}_n(k_E)$ be a group homomorphism given by: $\iota(g)= {^t}{\overline g}{^{-1}}$. Let us denote $\tau_\theta\circ\iota$ by ${\tau_\theta}^\iota$. So ${\tau_\theta}^\iota(g)=\tau_\theta(\iota(g))= \tau_{\theta}({^t}{\overline g}{^{-1}})$ for $g \in \mathrm{GL}_n(k_E)$. We also denote $\overline{\tau_\theta}(g)$ for $\tau_\theta(\overline g)$ for $g \in \mathrm{GL}_n(k_E)$. It can be observed clearly as $\theta$ is a character of $l^\times$, so $\theta(\lambda^m)= \theta^m(\lambda)$ for $m \in \mathbb{Z}, \lambda \in l^\times $. \par

Let $\tau_\theta^{\vee}$ be the dual representation of $\tau_\theta$. Let $V$ be the vector space corresponding to $\tau_\theta$ which is finite dimensional. Choose a basis $\{v_1, v_2, \ldots v_n\}$ of the vector space $V$. The dual basis $\{v_1^*, v_2^*, \ldots v_n^*\}$ for the dual space $V^*$ of $V$ can be constructed such that $v_i^*(v_j)= \delta_{ij}$ for $1 \leqslant i,j \leqslant n$. Suppose with respect to the above basis $\{v_1, v_2, \ldots v_n\}$, $\tau_\theta(g^{-1})$ represents matrix $A$ and  with respect to the dual basis $\{v_1^*, v_2^*, \ldots v_n^*\}$, $\tau_\theta^{\vee}(g)$ represents matrix $B$,then $A={^t}B$.\par 

From Proposition 3.5 in \cite{Thomas2014} we have $\overline\tau_\theta \simeq \tau_{\theta^q}$ and from Proposition 3.4 in \cite{Thomas2014} we have $\tau_\theta^{\vee} \simeq \tau_{\theta^{-1}}$.

\begin{proposition}\label{pro_5}
	
	Let $\theta$ be a regular character of $l^{\times}$. Then  $\tau_\theta^\iota \simeq \tau_\theta \Longleftrightarrow \theta^\gamma = \theta^{-q}$ for some $\gamma \in \mathrm{Gal}(l/k_E)$.
	
\end{proposition}
\begin{proof}
	$\Longrightarrow$
	As  $\tau_\theta^\iota \simeq \tau_\theta$, so $\chi_{\tau_\theta^\iota}(g)=\chi_{\tau_\theta}(g)$ for $g \in \mathrm{GL}_n(k_E)$. But $\chi_{\tau_\theta^\iota}(g)= \chi_{\tau_\theta}( {^t}{\overline g}{^{-1}})$,  since  $\chi_{\tau_\theta^\iota}(g)=\chi_{\tau_\theta}(\iota(g))$ for $ g \in \mathrm{GL}_n(k_E)$. As we know from the above discussion that $\tau_\theta^{\vee}(g)= (\tau_\theta(g^{-1}))^t $, so $trace(\tau_\theta^{\vee}(g))= trace(\tau_\theta(g^{-1}))^t$. Now $trace(\tau_\theta(g^{-1}))= trace(\tau_\theta(g^{-1}))^t$ as the trace of the matrix and it's transpose are same. So we have $trace(\tau_\theta(g^{-1}))=trace(\tau_\theta^{\vee}(g))$. Let us choose $h \in \mathrm{GL}_n(k_E)$ such that $h^{-1}{^t}g{^{-1}} h= g^{-1}$. So, $\chi_{\tau_\theta^{\vee}}(g)= \chi_{\tau_\theta}(g^{-1})= \chi_{\tau_\theta}(h^{-1}{^t}g{^{-1}} h) = \chi_{\tau_\theta}({^t}g{^{-1}})$. Let us denote $\tau_\theta^\eta(g)$ for  $\tau_\theta(\eta(g))$ where $\eta \colon g \longrightarrow {^t}g{^{-1}}$ is a group automorphism of $\mathrm{GL}_n(k_E)$. Hence $\chi_{\tau_\theta^\eta}(g)= \chi_{\tau_\theta}({^t}g{^{-1}})$. But we have already shown before that $\chi_{\tau_\theta}({^t}g{^{-1}})=\chi_{\tau_\theta^{\vee}}(g)$. so $\chi_{\tau_\theta^{\vee}}(g) = \chi_{\tau_\theta^\eta}(g)$. This implies $\tau_\theta^{\vee} \simeq \tau_\theta^\eta$. Hence $\tau_\theta^\iota= \overline\tau_\theta^\eta \simeq \overline\tau_\theta^{\vee} \simeq \tau_{\theta^q}^{\vee} \simeq \tau_{\theta^{-q}}$(since $\tau_\theta^{\vee} \simeq \tau_\theta^\eta, \overline\tau_\theta \simeq \tau_{\theta^q}, \tau_\theta^{\vee} \simeq \tau_{\theta^{-1}})$. Now from the hypothesis of Proposition, we know that $\tau_\theta^\iota \simeq \tau_\theta$, so this implies $\tau_\theta \simeq \tau_{\theta^{-q}}$  (since $ \tau_\theta^\iota \simeq \tau_{\theta^{-q}}$). But as $\theta$ is a regular character $\theta^\gamma= \theta^{-q}$ for some $\gamma \in \Gamma= \mathrm{Gal}(l/k_E)$ where $[l\colon k_E]= n$.\par
	
	$\Longleftarrow$ Now we can reverse the arguments and show that if $\theta^\gamma= \theta^{-q}$ for some
	$\gamma \in \Gamma= \mathrm{Gal}(l/k_E)$ then $ \tau_\theta^\iota \simeq \tau_{\theta^{-q}}$.
\end{proof}

\begin{proposition}\label{pro_6}
	If $\theta$ is a regular character of $l^{\times}$ such that $\theta^\gamma= \theta^{-q}$ for some $\gamma \in \Gamma$ then $n$ is odd. Conversely, if $n=2m+1$ is odd and $\theta$ is a regular character of $l^{\times}$ then $\theta^{\Phi^{m+1}}= \theta^{-q}$.
\end{proposition} 

\begin{proof}
	$\Longrightarrow$ Suppose $\theta$ is a regular character of $l^{\times}$ such that $\theta^\gamma= \theta^{-q}$ for some $\gamma \in \Gamma$. We know that $\Gamma=<\Phi>$ where $\Phi \colon l \longrightarrow l$ is the Frobenius map given by $\Phi(\lambda)=\lambda^{q^2}$ for $\lambda \in l$. Now $\Phi^n(\lambda)= \lambda^{q^{2n}}=\lambda$ for $\lambda \in l\Longrightarrow \Phi^n= 1$. Now we have $(\theta^{-q})^\gamma= (\theta^\gamma)^{-q}$. Hence $(\theta)^{\gamma^2}= (\theta^\gamma)^\gamma=(\theta^{-q})^\gamma= (\theta^\gamma)^{-q}= (\theta^{-q})^{-q} =\theta^{q^2}$. Now $\theta^{q^2}=\theta^\Phi$ because for $\lambda \in l^{\times}, \, \, \theta^{q^2}(\lambda)= \theta(\lambda^{q^2})= \theta(\Phi(\lambda))= \theta^\Phi(\lambda)$. As $\theta$ is a regular character and  $(\theta)^{\gamma^2}= \theta^\Phi$, so $\gamma^2= \Phi$. Let $\Phi$ be a generator of $\Gamma$ and $\gamma^2 = \Phi$. So $\gamma$ is also a generator of $\Gamma$.\par
	
	Hence order of $\gamma^2$= order of $\Phi \Longrightarrow \frac{n}{g.c.d (2,n)}=n  \Longrightarrow g.c.d (2,n)=1$. So $n$ is odd.\par
	
	$\Longleftarrow$ Suppose $n$ is odd. Let $n=2m+1$ where $m \in \mathbb{N}$. Now
	
	\begin{center} 
		$\mathrm{Hom}(l^{\times}, \mathbb{C^\times}) \cong l^{\times}$.\\
	\end{center}\par
	
	So $\mathrm{Hom}(l^{\times}, \mathbb{C^\times})$ is a cyclic group of order $(q^{2n}-1)$. Hence for every divisor $d$ of $(q^{2n}-1)$, there exists an element in  $\mathrm{Hom}(l^{\times}, \mathbb{C^\times})$ of order $d$. As $(q^n+1)$ is a divisor of $(q^{2n}-1)$, hence there exists an element $\theta$ in $\mathrm{Hom}(l^{\times}, \mathbb{C^\times})$ of order $(q^n+1)$. Hence $\theta^{q^n+1}=1 \Longrightarrow \theta^{q^n}= \theta^{-1} \Longrightarrow \theta^{q^{n+1}}= \theta^{-q} \Longrightarrow \theta^{q^{2m+2}}= \theta^{-q}\Longrightarrow \theta^{(q^{2})^{m+1}}= \theta^{-q} \Longrightarrow \theta^{\Phi^{m+1}} = \theta^{-q}\Longrightarrow \theta^\gamma= \theta^{-q}$, where $\gamma= \Phi^{m+1} \in \Gamma$. \par
	
	Now we claim that $\theta$ is a regular character in $\mathrm{Hom}(l^{\times}, \mathbb{C^\times})$. Suppose $ \theta^\gamma= \theta$ for some $\gamma \in \Gamma$. Let $\gamma= \Phi^k$ for some $k \in \mathbb{N}$. So we have $\theta^{\Phi^k}= \theta$. But $\theta^\Phi= \theta^{q^2}$, hence $\theta^{q^{2k}}=\theta$. That implies $\theta^{q^{2k}-1}=1$. As $\theta$ has order $(q^n+1)$, so $(q^n+1) \mid (q^{2k}-1)$. Let $l=2k$, so we have $(q^n+1) \mid (q^l-1)$ . If $l<n$ then it is a contradiction to the fact that $(q^n+1) \mid (q^l-1)$. Hence $l>n$. Now by applying Euclidean Algorithm for the integers $l,n$ we have $l=nd+r$ for some $0 \leqslant r < n$ and $d \geqslant 0 $ where $r,d \in \mathbb{Z}$. Now $d\neq 0$, because if $d=0$ then $l=r$ and that means $l<n$ which is a contradiction. So $d \in \mathbb{N}$. As we have  $(q^n+1) \mid (q^l-1) \Longrightarrow (q^n+1) \mid ((q^l-1)+(q^n+1)) \Longrightarrow (q^n+1) \mid (q^l+ q^n) \Longrightarrow (q^n+1) \mid q^n(q^r.q^{n(d-1)}+ 1)$. Now as $q^n$ and  $(q^n+1)$ are relatively prime, so  $(q^n+1) \mid (q^r.q^{n(d-1)}+ 1) \Longrightarrow (q^n+1) \mid ((q^r.q^{n(d-1)}+ 1)-(q^n+1)) \Longrightarrow (q^n+1) \mid q^n(q^r.q^{n(d-2)}- 1)$. As  $q^n$ and  $(q^n+1)$ are relatively prime, so  $(q^n+1) \mid (q^r.q^{n(d-2)}- 1)$. So continuing the above process we get, $(q^n+1)\mid (q^r+1)$ if $d$ is odd or $(q^n+1)\mid (q^r-1)$ if $d$ is even. But degree of $(q^n+1)$ is greater than degree of $(q^r+1)$ as $r<n$. So $r$ has to be equal to 0 and $l=2k= nd+r=nd$. And that implies $2 \mid nd$. But  $n$ is odd so $2 \mid d$. Now this means that $d$ is even and hence $(q^n+1)\mid (q^r-1)$. And  $(q^n+1)\mid (q^r-1)$ is not possible because $r=0$. So we have $2k=nd \Longrightarrow n \mid 2k$. But as $n$ is odd this implies $n \mid k$. And this further implies $k=np$ for some $p \in \mathbb{N}$. So $\gamma= \Phi^k= \Phi^{np}=1 \Longrightarrow \theta$ is regular character.
	
\end{proof}

Combining Proposition ~\ref{pro_5} and Proposition ~\ref{pro_6}, we have the following Proposition.

\begin{proposition}\label{pro_7}
	
	Let $\theta$ is a regular character of $l^{\times}$. Then $\tau_\theta^\iota \simeq \tau_\theta \Longleftrightarrow n$  is odd.
\end{proposition}\par

We know that $\rho_0$ is an irreducible supercuspidal representation of $K$. But $K \cong \mathfrak{P}_0$. So $\rho_0$ can be viewed as a representation of $\mathfrak{P}_0$. Now let us compute $N_G(\rho_0)$, where $N_G(\rho_0)=\{ m \in N_G(\mathfrak{P}_0) \mid {\rho_0}\simeq \rho_0^m \}$. Let $m \in N_G(\mathfrak{P}_0)$. Let $m$ be $J$ or of the form 
$\begin{bmatrix}
zk & 0\\
0 &  {^t}{\overline {(zk)}}{^{-1}}
\end{bmatrix}$ for some $z \in Z, k \in K$.

\begin{proposition}\label{pro_13}
	If $m=
	\begin{bmatrix}
	zk & 0\\
	0 &  {^t}{\overline {(zk)}}{^{-1}}
	\end{bmatrix}$ for some $z \in Z, k \in K$ then ${\rho_0}^m\simeq \rho_0$.
\end{proposition}

\begin{proof}
	As $\rho_0$ is an irreducible supercuspidal representation of $K$, so $K$ normalizes $\rho_0$. Clearly, $Z$ normalizes $\rho_0$. Thus $ZK$ normalizes $\rho_0$. As $\rho_0$ can also be viewed as a representation of $\mathfrak{P}_0$, so ${\rho_0}^m\simeq \rho_0$ where $m=\begin{bmatrix}
	zk & 0\\
	0 &  {^t}{\overline {(zk)}}{^{-1}}
	\end{bmatrix}$ for some $z \in Z, k \in K$.
\end{proof}

\begin{proposition}\label{pro_14}
	If $m=J=
	\begin{bmatrix}
	0 & 1_n \\
	1_n & 0
	\end{bmatrix}$ then ${\rho_0}^m\simeq \rho_0$ only when $n$ is odd.
\end{proposition}

\begin{proof}
	We know that $\iota\colon a \longrightarrow {^t}{\overline {a}}{^{-1}}$ is a group homomorphism of $\mathrm{GL}_n(k_E)$. Now $\iota\colon a \longrightarrow {^t}{\overline {a}}{^{-1}}$ can be inflated to a group homomorphism of $\mathrm{GL}_n(\mathfrak{O}_E)$. Further, $\iota$  can be viewed as a group homomorphism from $\mathfrak{P}_0$ to $\mathfrak{P}_0$ given by:\\
	\begin{center}
		$\iota\left(
		\begin{bmatrix}
		a & 0\\
		0 & {^t}{\overline {a}}{^{-1}}
		\end{bmatrix}\right)=
		\begin{bmatrix}
		{^t}{\overline {a}}{^{-1}} & 0\\
		0 & a
		\end{bmatrix}$\\
	\end{center} where $ a \in \mathrm{GL}_n(\mathfrak{O}_E)$. Let $g=
	\begin{bmatrix}
	a & 0\\
	0 & {^t}{\overline {a}}{^{-1}}
	\end{bmatrix}$. If $m=J$ then \[{\rho_0}^m(g)=\rho_0(JgJ^{-1})=\rho_0\left(
	\begin{bmatrix}
	{^t}{\overline {a}}{^{-1}} & 0\\
	0 & a
	\end{bmatrix}\right)= \rho_0(\iota(g))=\rho_0^{\iota}(g)\]. So ${\rho_0}^m(g)= \rho_0^{\iota}(g)$ for $g \in \mathfrak{P}_0 \Longrightarrow {\rho_0}^m=  \rho_0^{\iota}$.
	But from the hypothesis of Proposition, we know that $\rho_0^m \simeq \rho_0$. So we have $\rho_0\simeq\rho_0^{\iota}$. Now from Proposition ~\ref{pro_7}, $\rho_0\simeq\rho_0^{\iota}=\rho_0^m \Longleftrightarrow n$ is odd. 
\end{proof}\par

Thus we have the following conclusion about $N_G(\rho_0)$ for the unramified case:

If $n$ is even then $N_G(\rho_0)= Z(L) \mathfrak{P}_0$ and if $n$ is odd then $N_G(\rho_0)= Z(L) \mathfrak{P}_0 \rtimes \langle J \rangle$.

\subsection{Ramified case:}

Now that we have calculated $N_G(\mathfrak{P}_0)$, let us calculate $N_G(\rho_0)$ for the ramified case which would help us in determining the structure of $\mathcal{H}(G,\rho)$ in the ramified case. \par

As in section 5.1, $\rho_0=\tau_\theta$ for some regular character $\theta$ of $l^{\times}$ (where $l$ is a degree $n$ extension of $k_E$). We have $k_E=\mathbb{F}_{q}$. So $l=\mathbb{F}_{q^{n}}$.\par

Let $\Gamma=\mathrm{Gal}(l/k_E)$. The group $\Gamma$ is generated by Frobenius map $\Phi$  given by $\Phi(\lambda)=\lambda^{q}$ for $\lambda \in l$. Here $\Phi^n(\lambda)= \lambda^{q^{n}}=\lambda$ (since $l^{\times}$ is a cyclic group of order $q^{n}-1$) $\Longrightarrow \Phi^n= 1$. \par

For $\gamma \in \Gamma$ and $\theta \in \mathrm{Hom}(l^{\times}, {\mathbb{C}}^{\times})$, $\gamma$ acts on $\theta$ by $\gamma . \theta(\lambda)= \theta(\gamma(\lambda))$. Here $\gamma . \theta $ is also represented by $\theta^{\gamma}$.\par

As we are in the ramified case, so  $\mathrm{Gal}(k_E/k_F) = 1$. So $\overline g= g$ for $g \in k_E$. Let $\iota\colon \mathrm{GL}_n(k_E) \longrightarrow \mathrm{GL}_n(k_E)$ be a group homomorphism given by: $\iota(g)= {^t}{\overline g}{^{-1}}={^t}{g}{^{-1}} $. Let us denote $\tau_\theta\circ\iota$ by ${\tau_\theta}^\iota$. So ${\tau_\theta}^\iota(g)=\tau_\theta(\iota(g))= \tau_{\theta}({^t}{\overline g}{^{-1}})=\tau_{\theta}({^t}{g}{^{-1}})$ for $g \in \mathrm{GL}_n(k_E)$. We also denote $\overline{\tau_\theta}(g)$ for $\tau_\theta(\overline g)$ for $g \in \mathrm{GL}_n(k_E)$. But $\overline{\tau_\theta}(g)= \tau_\theta(\overline g)= \tau_\theta(g)$. It can be observed clearly as $\theta$ is a character of $l^{\times}$  , so $\theta(\lambda^m)= \theta^m(\lambda)$ for $m \in \mathbb{Z}, \lambda \in l^\times $. \par

\begin{proposition}\label{pro_8}
	Let $\theta$ be a regular character of $l^{\times}$. Then $\tau_\theta^\iota \simeq \tau_\theta \Longleftrightarrow \theta^\gamma = \theta^{-1}$ for some $\gamma \in \mathrm{Gal}(l/k_E)$  . 
\end{proposition}

\begin{proof}
	$\Longrightarrow$ As  $\tau_\theta^\iota \simeq \tau_\theta$, so $\chi_{\tau_\theta^\iota}(g)=\chi_{\tau_\theta}(g)$ for $g \in \mathrm{GL}_n(k_E)$. But $\chi_{\tau_\theta^\iota}(g)= \chi_{\tau_\theta}( {^t}{g}{^{-1}})$,  since  $\chi_{\tau_\theta^\iota}(g)=\chi_{\tau_\theta}(\iota(g))$ for $g \in \mathrm{GL}_n(k_E)$. As we know from the above discussion that $\tau_\theta^{\vee}(g)= (\tau_\theta(g^{-1}))^t $, so $trace(\tau_\theta^{\vee}(g))= trace(\tau_\theta(g^{-1}))^t$. Now $trace(\tau_\theta(g^{-1}))= trace(\tau_\theta(g^{-1}))^t$ as the trace of the matrix and it's transpose are same. So we have $trace(\tau_\theta(g^{-1}))=trace(\tau_\theta^{\vee}(g))$. Let us choose $h \in \mathrm{GL}_n(k_E)$ such that $h^{-1}{^t}g{^{-1}} h= g^{-1}$. So, $\chi_{\tau_\theta^{\vee}}(g)= \chi_{\tau_\theta}(g^{-1})= \chi_{\tau_\theta}(h^{-1}{^t}g{^{-1}} h) = \chi_{\tau_\theta}({^t}g{^{-1}})$.Let us denote $\tau_\theta^\eta(g)$ for  $\tau_\theta(\eta(g))$ where $\eta \colon g \longrightarrow {^t}g{^{-1}}$ is a group automorphism of $\mathrm{GL}_n(k_E)$. Hence $\chi_{\tau_\theta^\eta}(g)= \chi_{\tau_\theta}({^t}g{^{-1}})$. But we have already shown before that $\chi_{\tau_\theta}({^t}g{^{-1}})=\chi_{\tau_\theta^{\vee}}(g)$. so $\chi_{\tau_\theta^{\vee}}(g) = \chi_{\tau_\theta^\eta}(g)$. This implies  $\tau_\theta^{\vee} \simeq \tau_\theta^\eta$. Hence $ \tau_\theta^\iota= \tau_\theta^\eta \simeq \tau_\theta^{\vee}\simeq \tau_{\theta^{-1}}$ (since $\tau_\theta^{\vee} \simeq \tau_\theta^\eta, \tau_\theta^{\vee} \simeq \tau_{\theta^{-1}})$. Now from the hypothesis of Proposition, we know that $\tau_\theta^\iota \simeq \tau_\theta$, so this implies $\tau_\theta \simeq \tau_{\theta^{-1}}$ (since $ \tau_\theta^\iota \simeq \tau_{\theta^{-1}}$).But as $\theta$ is a regular character $\theta^\gamma= \theta^{-1}$ for some $\gamma \in \Gamma= \mathrm{Gal}(l/k_E)$ where $[l\colon k_E]= n$.\par
	
	$\Longleftarrow$ Now we can reverse the arguments and show that if $\theta^\gamma= \theta^{-1}$ for some $\gamma \in \Gamma= \mathrm{Gal}(l/k_E)$ then $ \tau_\theta^\iota \simeq \tau_{\theta^{-1}}$.
\end{proof}

\begin{proposition}\label{pro_9}
	If $\theta$ is a regular character of $l^{\times}$ such that $\theta^\gamma= \theta^{-1}$ for some $\gamma \in \Gamma$ then $n$ is even. Conversely, if $n=2m$ is even and $\theta$ is a regular character of $l^{\times}$ then $\theta^{\Phi^{m}}= \theta^{-1}$.
\end{proposition} 

\begin{proof}
	$\Longrightarrow$ Suppose $\theta$ is a regular character of $l^{\times}$ such that $\theta^\gamma= \theta^{-1}$ for some $\gamma \in \Gamma$. We know that $\Gamma=<\Phi>$ where $\Phi \colon l \longrightarrow l$ is the Frobenius map given by $\Phi(\lambda)=\lambda^{q}$ for $\lambda \in l$. Now $\Phi^n(\lambda)= \lambda^{q^{n}}=\lambda$ for $\lambda \in l\Longrightarrow \Phi^n= 1$. So for $\lambda \in l^{\times}$ we have $\theta^{\gamma^2}(\lambda)= \theta^{\gamma}(\gamma(\lambda))= \theta^{-1}(\gamma(\lambda))=\theta((\gamma(\lambda))^{-1})=\theta(\gamma(\lambda^{-1}))= \theta^{\gamma}(\lambda)^{-1}=\theta^{-1}(\lambda^{-1})= \theta((\lambda)^{-1})^{-1})= \theta(\lambda)$. So this implies $\theta^{\gamma^2}= \theta$. As $\theta$ is a regular character, so we have $\gamma^2=1$.
	Now for $\lambda \in l^{\times}$ we have $\theta^\Phi(\lambda)=\theta(\Phi(\lambda))=\theta(\lambda^q)=\theta^q(\lambda)$. That implies $\theta^\Phi= \theta^q$. As $\gamma^2=1 \Longrightarrow \gamma=1$ or $\gamma$ has order $2$. If $\gamma=1$ as $\theta^\gamma=\theta^{-1} \Longrightarrow \theta=\theta^{-1}\Longrightarrow \theta(\lambda)=\theta^{-1}(\lambda)$ for $\lambda \in l^{\times}  \Longrightarrow \theta(\lambda)=\theta(\lambda^{-1})\Longrightarrow \theta(\lambda^{2})=1\Longrightarrow
	(\theta(\lambda))^2=1\Longrightarrow \theta(\lambda)=\{\pm 1\}$ for $\lambda \in l^{\times} $.\par
	
	Let $q$ be an odd prime power. So for $\lambda \in l^{\times}$ we have $\theta^\Phi(\lambda)= \theta^q(\lambda)=(\theta(\lambda))^q= \theta(\lambda)$ (since $\theta(\lambda)=\{\pm 1\}$). So this implies $\theta^\Phi= \theta$ and that further implies $\Phi=1$ as $\theta$ is a regular character $\Longrightarrow n=1$ which contradicts our assumption that cardinality of $\Gamma$ is greater than $1$. Now suppose $q$ is a prime power of $2$. As the characteristic of $k_E=2$ that implies $+1=-1$ in $k_E$. So $\theta(\lambda)={\pm 1}={1}$ for $\lambda \in l^{\times}$. So we have for $\lambda \in l^{\times}$, $\theta^\Phi(\lambda)= \theta^q(\lambda)=(\theta(\lambda))^q=1$ (since $\theta(\lambda)=1$). And this implies $\theta^\Phi= \theta$ and that further implies $\Phi=1$ as $\theta$ is a regular character $\Longrightarrow n=1$ which contradicts our assumption that cardinality of $\Gamma$ is greater than $1$.\par
	
	Hence $\gamma^2=1$ or $\gamma$ has order $2$, since $\gamma \neq 1$. Now $\Gamma$ has order $n$ and $\gamma \in \Gamma$ has order 2. So $ 2\mid n \Longrightarrow n$ is even. \par  
	
	$\Longleftarrow$  Suppose $n$ is even. Let $n=2m$ where $m \in \mathbb{N}$. Now
	\begin{center} 
		$\mathrm{Hom}(l^{\times}, \mathbb{C^\times}) \cong l^{\times}$.\\
	\end{center}\par
	
	So $\mathrm{Hom}(l^{\times}, \mathbb{C^\times})$ is a cyclic group of order $(q^{n}-1)= (q^{2m}-1)$. Hence for every divisor $d$ of $(q^{2m}-1)$, there exists an element in  $\mathrm{Hom}(l^{\times}, \mathbb{C^\times})$ of order $d$. As $(q^m+1)$ is a divisor of $(q^{2m}-1)$, hence there exists an element $\theta$ in $\mathrm{Hom}(l^{\times}, \mathbb{C^\times})$ of order $(q^m+1)$. So $\theta^{q^m+1}= 1\Longrightarrow\theta^{q^m}=\theta^{-1}\Longrightarrow \theta^{\Phi^m}=\theta^{-1}$(since $\theta^{\Phi}=\theta^q$). Hence we have $\theta^{\gamma}=\theta^{-1}$, where
	$\gamma= \Phi^m$. \par
	
	Now we claim that the character $\theta$ is regular. Suppose  $\theta^{\gamma}=\theta$ for some $\gamma \in \Gamma$. Let $\gamma=\Phi^k$ for some $k \in \mathbb{Z}$. Then we have $\theta^{\Phi^k}=\theta\Longrightarrow\theta^{q^k}=\theta\Longrightarrow\theta^{q^k-1}=1$. As $\theta$ has order $(q^m+1)$ that means $(q^m+1)\mid (q^k-1)$. By Euclidean Algorithm, we have $k=md+r$ where $ r,d \in \mathbb{Z},0\leqslant r<m$ and $d \geqslant 0$. If $d=0$ then $k=r<m$ which contradicts the fact that $(q^m+1)\mid (q^k-1)$. So $d \geqslant 1$. Now as $(q^m+1)\mid (q^k-1)\Longrightarrow(q^m+1)\mid ((q^k-1)+(q^m+1))\Longrightarrow(q^m+1)\mid (q^k+q^m)\Longrightarrow(q^m+1)\mid (q^{md+r}+q^m) \Longrightarrow (q^m+1)\mid q^m(q^{m(d-1)+r}+1)$. But as $q^m$ and $(q^m+1)$ are relatively prime, so this implies $(q^m+1)\mid (q^{m(d-1)+r}+1)$. But $(q^m+1)\mid (q^{m(d-1)+r}+1) \Longrightarrow(q^m+1)\mid ((q^{m(d-1)+r}+1)-(q^m+1))\Longrightarrow (q^m+1)\mid q^m(q^{m(d-2)+r}-1)$. But as $q^m$ and $(q^m+1)$ are relatively prime, so this implies $(q^m+1)\mid (q^{m(d-2)+r}-1)$. Continuing the above process, we have $(q^m+1)\mid (q^r+1)$ if $d$ is odd and $(q^m+1)\mid (q^r-1)$if $d$ is even. As $r<m$, the above conditions are possible only when $r=0$. If $r=0$, then $k=md$. So if $d$ is odd then $(q^m+1)\mid 2\Longrightarrow (q^m+1)$ is either $1$ or $2$. If $(q^m+1)=1$ then $q=0$ which is a contradiction. So let $(q^m+1)=2$ then we have $q=1$ which is again a contradiction as $q$ is a prime power. So $d$ has to be even. Let $d$ be even and is greater than $2$. So $d$ can take values $4,6,8,\ldots$. But as $k=md$, so $k$ can take values $4m,6m,8m,\dots$. That is $k$ can take values $2n,3n,6n,\dots$ which is a contradiction as $k<n$. So $d=2$. Hence $k=2m=n$. So $\Phi^k= \Phi^n = 1\Longrightarrow\gamma=1$. So $\theta$ is a regular character.
	
\end{proof}
\par
Combining Proposition ~\ref{pro_8} and Proposition ~\ref{pro_9} we have the following Proposition.

\begin{proposition}\label{pro_10}
	
	Let $\theta$ be a regular character of $l^{\times}$. Then $\tau_\theta^\iota \simeq \tau_\theta \Longleftrightarrow n$ is even.
	
\end{proposition}\par

We know that $\rho_0$ is an irreducible supercuspidal representation of $K$. But $K \cong \mathfrak{P}_0$. So $\rho_0$ can be viewed as a representation of $\mathfrak{P}_0$. Now let us compute $N_G(\rho_0)$, where $N_G(\rho_0)=\{ m \in N_G(\mathfrak{P}_0) \mid {\rho_0}\simeq \rho_0^m \}$. Let $m \in N_G(\mathfrak{P}_0)$. Let $m$ be $J$ or of the form 
$\begin{bmatrix}
zk & 0\\
0 &  {^t}{\overline {(zk)}}{^{-1}}
\end{bmatrix}$ for some $z \in Z, k \in K$.

\begin{proposition}\label{pro_15}
	If $m=
	\begin{bmatrix}
	zk & 0\\
	0 &  {^t}{\overline {(zk)}}{^{-1}}
	\end{bmatrix}$ for some $z \in Z, k \in K$ then ${\rho_0}^m\simeq \rho_0$.
\end{proposition}
\begin{proof}
The proof is similar to the proof of Proposition ~\ref{pro_13}.
\end{proof}

\begin{proposition}\label{pro_16}
	If $m=J=
	\begin{bmatrix}
	0 & 1_n \\
	1_n & 0
	\end{bmatrix}$ then ${\rho_0}^m\simeq \rho_0$ only when $n$ is even.
\end{proposition}
\begin{proof}
The proof is similar to the proof of Proposition ~\ref{pro_14}.
\end{proof}\par
	
So we have the following conclusion about $N_G(\rho_0)$ for ramified case:
If $n$ is odd then $N_G(\rho_0)= Z(L) \mathfrak{P}_0$ and if $n$ is even then $N_G(\rho_0)= Z(L) \mathfrak{P}_0 \rtimes \langle J \rangle$.

\begin{lemma}\label{lem_2}
	When $n$ is odd in the unramified case or when $n$ is even in the ramified case, we have $N_G(\rho_0)= \left<\mathfrak{P}_0, w_0, w_1\right>$, where $w_0=J$ and $w_1=
	\begin{bmatrix}
	0 & {\overline\varpi_E}^{-1}1_n \\
	\varpi_E 1_n & 0
	\end{bmatrix}$.\\
\end{lemma}
\begin{proof}
	Let $\zeta= w_0w_1$. So $\zeta=
	\begin{bmatrix}
	\varpi_E 1_n & 0\\
	0 & {\overline\varpi_E}^{-1}1_n
	\end{bmatrix}$. We can clearly see that $w_0^2=1$. So $w_0= w_0^{-1}$ and $w_1=w_0^{-1}\zeta=w_0\zeta$. From the hypothesis of lemma, we have  $N_G(\rho_0)=  Z(L) \mathfrak{P}_0 \rtimes \langle J \rangle$. 
	As any element in $E^{\times}$ can be written as $u\varpi_E^n$ for some $n \in \mathbb{Z}, u \in \mathfrak{O}_E^{\times}$, so $Z(L)=Z(\mathfrak{P}_0)\langle\zeta\rangle $. So $Z(L)\mathfrak{P}_0= \left<\mathfrak{P}_0,\zeta\right>$. Hence $N_G(\rho_0)= \left<\mathfrak{P}_0,\zeta\right> \rtimes J$. But $J=w_0, w_1=w_0\zeta$. So $N_G(\rho_0)= \left<\mathfrak{P}_0,w_0, w_1 \right>$. 
\end{proof}

\section{Structure of $\mathcal{H}(G,\rho)$} \label{sec_6} 

\subsection{Unramified case:}

In this section we determine the structure of $\mathcal{H}(G,\rho)$ for the unramified case when $n$ is odd. Using cuspidality of $\rho_0$, it can be shown by Theorem 4.15 in \cite{MR1235019}, that $ \mathfrak{I}_G(\rho)= \mathfrak{P}N_G(\rho_0)\mathfrak{P}$. But from lemma \ref{lem_2}, $N_G(\rho_0)= \left<\mathfrak{P}_0, w_0, w_1\right>$. So $\mathfrak{I}_G(\rho)=\mathfrak{P}\left<\mathfrak{P}_0, w_0, w_1\right>\mathfrak{P}= \mathfrak{P}\left<w_0, w_1\right> \mathfrak{P}$. Let $V$ be the vector space corresponding to $\rho$. Let us recall that  $\mathcal{H}(G, \rho)$ consists of maps $f \colon G \to End_{\mathbb{C}}(V^{\vee})$ such that support of $f$ is compact and $f(pgp')= \rho^{\vee}(p)f(g)\rho^{\vee}(p')$ for $p,p' \in \mathfrak{P}, g \in G$. In fact $\mathcal{H}(G,\rho)$ consists of $\mathbb{C}$-linear combinations of maps $f \colon G \longrightarrow End_{\mathbb{C}}(V^{\vee})$ such that $f$ is supported on $\mathfrak{P}x\mathfrak{P}$ where $x \in \mathfrak{I}_G(\rho)$ and $f(pxp')= \rho^{\vee}(p)f(x)\rho^{\vee}(p')$ for $p,p' \in \mathfrak{P}$. We shall now show there exists $\phi_0 \in \mathcal{H}(G,\rho)$ with support $\mathfrak{P}w_0\mathfrak{P}$ and satisfies $\phi_0^2= q^n+ (q^n-1)\phi_0$. Let

\[
K(0)= \mathrm{U}(n,n) \cap \mathrm{GL}_{2n}(\mathfrak{O}_E)=\{g \in \mathrm{GL}_{2n}(\mathfrak{O}_E)\mid ^t\overline{g}Jg=J\},
\]
\[
K_1(0)=\{g \in 1+ \varpi_E \mathrm{M}_{2n}(\mathfrak{O}_E)\mid ^t\overline{g}Jg=J\},
\]
\[
\mathsf{G}= \{g \in \mathrm{GL}_{2n}(k_E) \mid ^t\overline{g}Jg=J\}.
\]

The map $r$ from $K(0)$ to $\mathsf{G}$ given by $r\colon K(0)\xrightarrow{\text{mod} \, p_E}\mathsf{G}$ is a surjective group homomorphism with kernel $K_1(0)$. So by the first isomorphism theorem of groups we have:

\begin{center}
	$\frac{K(0)}{K_1(0)}\cong \mathsf{G}.$
\end{center}

$r(\mathfrak{P})= \mathsf{P}=
\begin{bmatrix}
\mathrm{GL}_n(k_E) &  \mathrm{M}_n(k_E) \\
0         &  \mathrm{GL}_n(k_E)
\end{bmatrix} \bigcap \mathsf{G}$= Siegel parabolic subgroup of $\mathsf{G}$.\\\par

Now $\mathsf{P}= \mathsf{L} \ltimes \mathsf{U}$, where $\mathsf{L}$ is the Siegel Levi component of $\mathsf{P}$ and $\mathsf{U}$ is the unipotent radical of $\mathsf{P}$. Here

\[
\mathsf{L}= \Bigg \lbrace
\begin{bmatrix}
a & 0\\
0 & ^t\overline{a}^{-1}\\
\end{bmatrix} \mid a \in \mathrm{GL}_n(k_E) \Bigg \rbrace,
\]

\[
\mathsf{U}= \Bigg \lbrace
\begin{bmatrix}
1_n & X \\
0  & 1_n \\
\end{bmatrix} \mid X \in \mathrm{M}_n(k_E), X+{^t}\overline{X}=0 \Bigg \rbrace.
\]

Let $V$ be the vector space corresponding to $\rho$. The Hecke algebra $\mathcal{H}(K(0),\rho)$ is a sub-algebra of $\mathcal{H}(G,\rho)$.\par

Let $\overline{\rho}$ be the representation of $\mathsf{P}$ which when inflated to $\mathfrak{P}$ is given by $\rho$ and $V$ is also the vector space corresponding to $\overline{\rho}$. The Hecke algebra $\mathcal{H}(\mathsf{G},\overline{\rho})$ looks as follows:

\[
\mathcal{H}(\mathsf{G},\overline{\rho})= \left\lbrace f \colon \mathsf{G} \to End_{\mathbb{C}}(V^{\vee}) \; \middle|  \;
\begin{varwidth}{\linewidth}
$f(pgp')= \overline{\rho}^{\vee}(p)f(g)\overline{\rho}^{\vee}(p')$\\
where $p,p' \in \mathsf{P}, \, g \in \mathsf{G}$
\end{varwidth}
\right \rbrace.
\] \par

Now the homomorphism $r \colon K(0) \longrightarrow \mathsf{G}$ extends to a map from  $\mathcal{H}(K(0),\rho)$ to $\mathcal{H}(\mathsf{G},\overline{\rho})$ which we again denote by $r$. Thus $r \colon \mathcal{H}(K(0),\rho) \longrightarrow \mathcal{H}(\mathsf{G},\overline{\rho})$ is given by

\[
r(\phi)(r(x))=\phi(x)
\]
\[ 
\text{for} \, \phi \in \mathcal{H}(K(0),\rho) \, \text{and} \, x \in K(0).
\]
\begin{proposition}\label{pro_17}
	The map $r \colon \mathcal{H}(K(0),\rho) \longrightarrow \mathcal{H}(\mathsf{G},\overline{\rho})$ is an algebra isomorphism.
\end{proposition}
\begin{proof}
	To prove that the map $r$ is an isomorphism of algebras, we have to show that $r$ is a homomorphism of algebras and is a bijective map.\par
	
	In order to show that the map $r$ is a homomorphism, we need to show that it is $\mathbb{C}$-linear and it preserves convolution. It is obvious that the map $r$ is $\mathbb{C}$-linear. Let us now show that the map preserves convolution.\par 
	If $x \in K(0)$ and $\phi_1,\phi_2 \in \mathcal{H}(K(0),\rho)$ then
	\[
	(\phi_1 * \phi_2)(x)=\int_{K(0)} \phi_1(y)\phi_2(y^{-1}x)dy.
	\] 
	
	Now
	\[\int_{K(0)} \phi_1(y)\phi_2(y^{-1}x)dy=\sum_{y \in \mathfrak{P}/K(0)} \phi_1(xy^{-1})\phi_2(y).\] 
	
	Hence
	\begin{align*}
	r(\phi_1* \phi_2)(r(x))&= (\phi_1*\phi_2)(x)\\
	&=\sum_{y \in \mathfrak{P}/K(0)} \phi_1(xy^{-1})\phi_2(y)\\
	&= \sum_{y \in \mathfrak{P}/K(0)} (r(\phi_1)(r(xy^{-1})))(r(\phi_2)(r(y)))\\
	&=\sum_{r(y) \in \mathsf{P}/\mathsf{G}} (r(\phi_1)(r(x)(r(y))^{-1}))(r(\phi_2)(r(y)))\\
	&= (r(\phi_1)*r(\phi_2))(r(x)). 
	\end{align*}
	So we have $r(\phi_1* \phi_2)(r(x))= (r(\phi_1)*r(\phi_2))(r(x))$. But $r$ is a surjective group homomorphism from $K(0)$ to $\mathsf{G}$. Hence $r(\phi_1* \phi_2)(y)= (r(\phi_1)* r(\phi_2))(y)$ for $y \in \mathsf{G}$ which would imply that $r(\phi_1* \phi_2)= (r(\phi_1)* r(\phi_2))$. Hence $r$ is a homomorphism of algebras. \par  
	
	In order to show that $r$ is bijective map, we first show here that the map $r$ is a one-one map. Let $\phi_1, \phi_2 \in \mathcal{H}(K(0),\rho), y \in \mathsf{G}$. Suppose $r(\phi_1)(y)= r(\phi_2)(y)$. As $r$ is surjective map from $K(0)$ to $\mathsf{G}$, so there exists $x \in K(0)$ such that $r(x)=y$. So $ r(\phi_1)(r(x))= r(\phi_2)(r(x)) \Longrightarrow \phi_1(x)= \phi_2(x)$. As $r$ is a surjective map from $K(0)$ to $\mathsf{G}$, so when $y$ spans over $\mathsf{G}$, $x$ spans over $K(0)$. So $\phi_1(x)= \phi_2(x)$ for $x \in K(0) \Longrightarrow \phi_1= \phi_2$. So $r$ is a one-one map. \par

	Now we show that $r$ is a surjective map from $\mathcal{H}(K(0),\rho)$ to $\mathcal{H}(\mathsf{G},\overline{\rho})$. Let $\psi \in \mathcal{H}(\mathsf{G},\overline{\rho})$, then $\psi \colon \mathsf{G}\longrightarrow End_{\mathbb{C}}V^{\vee}$ is a map such that $\psi(pgp') = \overline{\rho}^{\vee}(p) \psi(g)\overline{\rho}^{\vee}(p')$ for $p,p' \in \mathsf{P}, g \in \mathsf{G}$. As $r$ is a surjective map from $K(0)$ to $\mathsf{G}$ , so $\psi \circ r$ makes sense. Now let us call $\psi \circ r$ as $\phi$. So $\phi$ is a map from $K(0)$ to $End_{\mathbb{C}}V^{\vee}$. Let $p,p' \in \mathfrak{P} , k \in K(0)$, so $\phi(pkp')= (\psi \circ r)(pkp')= \psi(r(pkp'))=\psi(r(p)r(k)r(p'))= \overline{\rho}^{\vee}(r(p))\psi(r(k))\overline{\rho}^{\vee}(r(p'))= \rho^{\vee}(p)(\psi \circ r)(k)\rho{^\vee}(p')=\rho^{\vee}(p) \phi(k)\rho^{\vee}(p')$. So $\phi \in  \mathcal{H}(K(0),\rho)$. Let $y \in \mathsf{G}$. So there exits $x \in K(0)$ such that $r(x)= y$. Now consider $\psi(y)= \psi(r(x))=(\psi \circ r)(x)= \phi(x)= r(\phi)(r(x))= r(\phi)(y)$. So $\psi(y)= r(\phi)(y)$ for $y \in \mathsf{G} \Longrightarrow \psi= r(\phi)$. Hence $r$ is a surjective map.\par
	
	As $r$ is both one-one and surjective map, hence it is a bijective map.\par
\end{proof}

Let $w=r(w_0)=r(
\begin{bmatrix}
0 & 1_n \\
1_n & 0
\end{bmatrix})=
\begin{bmatrix}
0 & 1_n \\
1_n & 0
\end{bmatrix} \in \mathsf{G}$. Clearly $K(0) \supseteq \mathfrak{P}\amalg\mathfrak{P}w_0\mathfrak{P}\Longrightarrow r(K(0))\supseteq r(\mathfrak{P}\amalg\mathfrak{P}w_0\mathfrak{P})\Longrightarrow\mathsf{G}\supseteq r(\mathfrak{P})\amalg r(\mathfrak{P}w_0\mathfrak{P})=\mathsf{P}\amalg \mathsf{P}w\mathsf{P}$. So $\mathsf{G}\supseteq \mathsf{P}\amalg \mathsf{P}w\mathsf{P}$.\par 

Now $Ind_{\mathsf{P}}^{\mathsf{G}}\overline{\rho}= \pi_1 \oplus \pi_2$, where $\pi_1, \pi_2$ are distinct irreducible representations of $\mathsf{G}$ with $\text{dim} \pi_2 \geqslant \text{dim} \pi_1$. Let $\lambda= \frac{\text{dim} \pi_2}{\text{dim} \pi_1}$. By Proposition 3.2 in \cite{MR2276353}, there exists a unique $\phi$ in $\mathcal{H}(\mathsf{G},\overline{\rho})$ with support $\mathsf{P}w\mathsf{P}$ such that $\phi^2= \lambda + (\lambda-1)\phi$. By Proposition ~\ref{pro_17}, there is a unique element $\phi_0$ in $\mathcal{H}(K(0),\rho)$ such that $r(\phi_0)=\phi$. Thus supp($\phi_0$)=$\mathfrak{P}w_0\mathfrak{P}$ and $\phi_0^2= \lambda+(\lambda-1)\phi_0$. From Lemma 3.12 in \cite{Thomas2014}, $\lambda= q^n$. Hence $\phi_0^2= q^n+ (q^n-1)\phi_0$. As support of $\phi_0= \mathfrak{P}w_0\mathfrak{P}\subseteq K(0)\subseteq G$, so $\phi_0$ can be extended to $G$ and viewed as an element of $\mathcal{H}(G,\rho)$. Thus $\phi_0$ satisfies the following relation in $\mathcal{H}(G,\rho)$:

\[
\phi_0^2= q^n+ (q^n-1)\phi_0.
\]
\par  

We shall now show there exists $\phi_1 \in \mathcal{H}(G, \rho)$ with support $\mathfrak{P}w_1\mathfrak{P}$ satisfying the same relation as $\phi_0$. Let $\eta= 
\begin{bmatrix}
0 & 1_n\\
\varpi_E 1_n & 0
\end{bmatrix}$. Now we can check that $\eta w_0 \eta^{-1}= w_1$. Recall that $\mathfrak{P}$ looks as follows:

\begin{center}
	$\mathfrak{P}=
	\begin{bmatrix}
	\mathrm{GL}_n(\mathfrak{O}_E) & \mathrm{M}_n(\mathfrak{O}_E)\\
	\mathrm{M}_n(p_E) & \mathrm{GL}_n(\mathfrak{O}_E)
	\end{bmatrix} \bigcap G.$
\end{center}

\begin{lemma}\label{lem_4}
	$\eta \mathfrak{P}\eta^{-1}= \mathfrak{P}.$
\end{lemma}

\begin{proof}
	\[
	\mathfrak{P}=
	\begin{bmatrix}
	\mathrm{GL}_n(\mathfrak{O}_E) & \mathrm{M}_n(\mathfrak{O}_E)\\
	\mathrm{M}_n(p_E) & \mathrm{GL}_n(\mathfrak{O}_E)
	\end{bmatrix} \bigcap G
	\]
	\begin{center}
		$\Longrightarrow \eta \mathfrak{P} \eta^{-1}=
		\eta \begin{bmatrix}
		\mathrm{GL}_n(\mathfrak{O}_E) & \mathrm{M}_n(\mathfrak{O}_E)\\
		\mathrm{M}_n(p_E) & \mathrm{GL}_n(\mathfrak{O}_E)
		\end{bmatrix} \eta^{-1} \bigcap \eta G \eta^{-1}$.
	\end{center}
	
	It is easy to show that 
	\[\eta \begin{bmatrix}
	\mathrm{GL}_n(\mathfrak{O}_E) & \mathrm{M}_n(\mathfrak{O}_E)\\
	\mathrm{M}_n(p_E) & \mathrm{GL}_n(\mathfrak{O}_E)
	\end{bmatrix} \eta^{-1}=
	\begin{bmatrix}
	\mathrm{GL}_n(\mathfrak{O}_E) & \mathrm{M}_n(\mathfrak{O}_E)\\
	\mathrm{M}_n(p_E) & \mathrm{GL}_n(\mathfrak{O}_E)
	\end{bmatrix}.\] 
	
	Now we claim that $\eta G \eta^{-1}= G$. To prove this let us consider \[G'=\lbrace g \in \mathrm{GL}_{2n}(E) \mid ^t\overline{g}Jg= \lambda(g)J \, \text{for some} \, \lambda(g) \in F^{\times}\rbrace.\] 
	
	Now $\eta \in G'$ clearly, as $^t\overline{\eta}J\eta= \varpi_E J= \varpi_F J$. And $\lambda \colon G' \longrightarrow F^{\times}$ is a homomorphism of groups with kernel $G$. So $G \unlhd G'$. As $\eta \in G'$ and $G \unlhd G'$, so $\eta G \eta^{-1}= G$. Hence $\eta \mathfrak{P}\eta^{-1}= \mathfrak{P}$.
\end{proof}\par

As $\mathfrak{P} \subseteq K(0)$ and $w_0 \in K(0)$, so $K(0) \supseteq \mathfrak{P} \amalg \mathfrak{P}w_0\mathfrak{P} \Longrightarrow \eta K(0) \eta^{-1} \supseteq \eta \mathfrak{P} \eta^{-1} \amalg \eta\mathfrak{P}w_0\mathfrak{P}\eta^{-1}$. But from lemma \ref{lem_4}, we know that  $\eta \mathfrak{P} \eta^{-1}= \mathfrak{P}$ and  $\eta\mathfrak{P}w_0\mathfrak{P}\eta^{-1}= (\eta\mathfrak{P}\eta^{-1})(\eta w_0 \eta^{-1})(\eta\mathfrak{P}\eta^{-1})=\mathfrak{P}w_1\mathfrak{P}$ (since $\eta w_0 \eta^{-1}=w_1$). So $\eta K(0) \eta^{-1} \supseteq \mathfrak{P} \amalg \mathfrak{P}w_1\mathfrak{P}$.\par

Let $r'$ be homomorphism of groups given by the map $r'\colon \eta k(0)\eta^{-1} \longrightarrow \mathsf{G}$ such that $r'(x)= (\eta^{-1}x\eta) mod p_E$ for $x \in \eta K(0) \eta^{-1}$. Observe that $r'$ is a surjective homomorphism of groups because $r'(\eta K(0) \eta^{-1})= (\eta^{-1} \eta K(0) \eta^{-1} \eta) mod p_E=  K(0) mod p_E =\mathsf{G}$. The kernel of group homomorphism is $\eta K_1(0) \eta^{-1}$. Now by the first isomorphism theorem of groups we have $\frac{\eta K(0) \eta^{-1}}{\eta K_1(0) \eta^{-1}} \cong \frac{K(0)}{K_1(0)} \cong \mathsf{G}$. Also $r'(\eta\mathfrak{P}\eta^{-1})=(\eta^{-1} \eta \mathfrak{P} \eta^{-1} \eta) mod p_E=\mathfrak{P}mod p_E= \mathsf{P}$. Let $\overline{\rho}$ be representation of $\mathsf{P}$ which when inflated to $\mathfrak{P}$ is given by $\rho$. The Hecke algebra of $\eta K(0) \eta^{-1}$ which we denote by $\mathcal{H}(\eta K(0) \eta^{-1}, \rho)$ is a sub-algebra of $\mathcal{H}(G,\rho)$.\par

The map $r':\eta K(0) \eta^{-1} \longrightarrow \mathsf{G}$ extends to a map from  $\mathcal{H}(\eta K(0) \eta^{-1},\rho)$ to $\mathcal{H}(\mathsf{G}, \overline{\rho})$ which we gain denote by $r'$. Thus $r' \colon \mathcal{H}(\eta K(0) \eta^{-1},\rho) \longrightarrow \mathcal{H}(\mathsf{G}, \overline{\rho})$ is given by
\[r'(\phi)(r'(x))= \phi(x) \]
\[\text{for} \, \phi \in \mathcal{H}(\eta K(0) \eta^{-1},\rho) \, \text{and} \, x \in \eta K(0) \eta^{-1}.\] \par

The proof that $r'$ is an isomorphism goes in the similar lines as Proposition ~\ref{pro_17}. We can observe that $r'(w_1)=w \in \mathsf{G}$, where $w$ is defined as before in this section. As we know from our previous discussion in this section, that there exists a unique $\phi$ in  $\mathcal{H}(\mathsf{G}, \overline{\rho})$ with support $\mathsf{P} w \mathsf{P}$ such that $\phi^2= q^n+ (q^n-1)\phi$. Hence  there is a unique element $\phi_1 \in \mathcal{H}(\eta K(0) \eta^{-1},\rho)$ such that $r'(\phi_1)=\phi$. Thus supp($\phi_1$)=$\mathfrak{P} w_1 \mathfrak{P}$ and $\phi_1^2= q^n+ (q^n-1)\phi_1$. Now $\phi_1$ can be extended to $G$ and viewed as an element in $\mathcal{H}(G, \rho)$ as $\mathfrak{P} w_1 \mathfrak{P} \subseteq \eta K(0) \eta^{-1}\subseteq G$. Thus $\phi_1$ satisfies the following relation in $\mathcal{H}(G,\rho)$:

\[
\phi_1^2= q^n+ (q^n-1)\phi_1.
\]\par 

Thus we have shown there exists $\phi_i \in \mathcal{H}(G,\rho)$ with supp($\phi_i$)=$\mathfrak{P}w_i\mathfrak{P}$ satisfying $\phi_i^2= q^n+ (q^n-1)\phi_i$  for $i=0,1$.
It can be further shown that $\phi_0$ and $\phi_1$ generate the Hecke algebra $\mathcal{H}(G, \rho)$. Let us denote the Hecke algebra $\mathcal{H}(G, \rho)$ by $\mathcal{A}$. So 
\[
\mathcal{A}= \mathcal{H}(G,\rho)= \left\langle \phi_i \colon G \to End_{\mathbb{C}}(\rho^{\vee})\; \middle| \;
\begin{varwidth}{\linewidth} 
$\phi_i$ is supported on 
$\mathfrak{P}w_i\mathfrak{P}$\\
and $\phi_i(pw_ip')= \rho^{\vee}(p)\phi_i(w_i)\rho^{\vee}(p')$\\
where $p,p' \in \mathfrak{P}, \, i=0,1$
\end{varwidth}
\right\rangle
\] where $\phi_i$ satisfies the relation:

\begin{center}
	$\phi_i^2= q^n+ (q^n-1) \phi_i$ for $i=0,1$.
\end{center}  

\begin{lemma}\label{lem_5}
	$\phi_0$ and $\phi_1$ are units in $\mathcal{A}$.
\end{lemma}
\begin{proof}
	As $\phi_i^2= q^n + (q^n-1) \phi_i$ for $i=0,1$. So $\phi_i( \frac{\phi_i+(1-q^n)1}{q^n})= 1$ for i=0,1. Hence $\phi_0$ and $\phi_1$ are units in $\mathcal{A}$.
\end{proof}\par

\begin{lemma}\label{lem_6}
	Let $\phi, \psi \in \mathcal{H}(G,\rho) $ with support of $\phi, \psi$ being $\mathfrak{P}x\mathfrak{P}, \mathfrak{P}y\mathfrak{P}$ respectively. Then supp($\phi * \psi$)=supp($\phi \psi$) $\subseteq$ (supp($\phi$))(supp($\psi$))=$\mathfrak{P}x\mathfrak{P}y\mathfrak{P}$.
\end{lemma}
\begin{proof}
	As supp($\phi$)= $\mathfrak{P}x\mathfrak{P}$ and  supp($\psi$)= $\mathfrak{P}y\mathfrak{P}$, so if $z \in \text{supp}(\phi * \psi)$ then $(\phi * \psi)(z)=\int_G \phi(zr^{-1})\psi(r)dr \neq 0$. So there exists $r \in G$ such that $\phi(zr^{-1})\psi(r) \neq 0$. Because if $\phi(zr^{-1})\psi(r)=0$ for $r \in G$ then $\int_G\phi(zr^{-1})\psi(r)=0 \Longrightarrow (\phi* \psi)(z)=0$ which is a contradiction. So $\phi(zr^{-1})\psi(r) \neq 0$ for some $r \in G$. As $\phi(zr^{-1}) \neq 0 \Longrightarrow zr^{-1} \in \text{supp}(\phi)= \mathfrak{P}x\mathfrak{P}$ and $\psi(r) \neq 0 \Longrightarrow r \in \text{supp}(\psi) = \mathfrak{P}y\mathfrak{P}$. Hence $(zr^{-1})(r)=z \in (\text{supp} (\phi))(\text{supp}(\psi))= (\mathfrak{P}x\mathfrak{P})(\mathfrak{P}y\mathfrak{P})= \mathfrak{P}x\mathfrak{P}y\mathfrak{P}$. Hence supp($\phi * \psi$)= supp($\phi \psi$) $\subseteq$ (supp($\phi$))(supp($\psi$))=$\mathfrak{P}x\mathfrak{P}y\mathfrak{P}$.
\end{proof}\par  

From B-N pair structure theory we can show that, $\mathfrak{P}x\mathfrak{P}y\mathfrak{P}= \mathfrak{P}xy\mathfrak{P}\Longleftrightarrow l(xy)= l(x)+l(y)$. From lemma \ref{lem_6}, we  have supp($\phi_0\phi_1$) $\subseteq \mathfrak{P}w_0\mathfrak{P}w_1\mathfrak{P}$. But $\mathfrak{P}w_0\mathfrak{P}w_1\mathfrak{P} =\mathfrak{P}w_0w_1\mathfrak{P}$ (since $l(w_0w_1)=l(w_0)+l(w_1)$). Thus supp($\phi_0\phi_1$)$\subseteq \mathfrak{P}w_0w_1\mathfrak{P}$. Let $\zeta=w_0w_1$, So 
\begin{center}
	$\zeta = 
	\begin{bmatrix}
	\varpi_E 1_n & 0\\
	0     & \varpi_E^{-1} 1_n
	\end{bmatrix}.$
\end{center}\par

As $\phi_0, \phi_1$ are units in algebra $\mathcal{A}$, so $\psi = \phi_0\phi_1$ is a unit too in $\mathcal{A}$ and $\psi^{-1}= \phi_1^{-1}\phi_0^{-1}$. Now as we have seen before that supp($\phi_0\phi_1$) $\subseteq \mathfrak{P}w_0w_1\mathfrak{P} \Longrightarrow supp(\psi) \subseteq \mathfrak{P}\zeta \mathfrak{P}\Longrightarrow supp(\psi)= \varnothing \,\text{or} \, \mathfrak{P}\zeta\mathfrak{P}$. If supp($\psi$)= $\varnothing \Longrightarrow \psi= 0$ which is a contradiction as $\psi$ is a unit in $\mathcal{A}$. So supp($\psi$) = $\mathfrak{P} \zeta \mathfrak{P}$. As $\psi$ is a unit in $\mathcal{A}$, we can show as before from B-N pair structure theory that supp($\psi^2$) = $\mathfrak{P} \zeta^2 \mathfrak{P}$. Hence by induction on $n \in \mathbb{N}$, we can further show from B-N pair structure theory that supp($\psi^n$)= $\mathfrak{P} \zeta^n \mathfrak{P}$ for $n \in \mathbb{N}$. \par

Now $\mathcal{A}$ contains a sub- algebra generated by 
$\psi, \psi^{-1}$ over $\mathbb{C}$ and we denote this sub-algebra by $\mathcal{B}$. So $\mathcal{B}=\mathbb{C}[\psi,\psi^{-1}]$ where

\[
\mathcal{B}=\mathbb{C}[\psi,\psi^{-1}] = \left\lbrace c_k\psi^k + \cdots +c_l \psi^l \; \middle| \;
\begin{varwidth}{\linewidth} 
$c_k, \ldots ,c_l \in \mathbb{C};$\\
$k <l ; k,l \in \mathbb{Z}$
\end{varwidth}
\right\rbrace.
\]
\par

\begin{proposition}\label{algebra_isomorphism}
	The unique algebra homomorphism $\mathbb{C}[x, x^{-1}] \longrightarrow \mathcal{B}$ given by $x \longrightarrow \psi$ is an isomorphism. So $\mathcal{B} \simeq \mathbb{C}[x, x^{-1}]$. 
\end{proposition}
\begin{proof}
	It is obvious that the map is an algebra homomorphism and is surjective as $\lbrace \psi^n \mid n \in \mathbb{Z} \rbrace$ spans $\mathcal{B}$. Now we show that the kernel of map is 0. Suppose $c_k \psi^k+ \cdots +c_l \psi^l=0$ with $c_k, \ldots , c_l \in \mathbb{C}; l> k \geqslant 0; l,k \in \mathbb{Z}$. Let $x \in \text{supp} \psi^s = \mathfrak{P}\zeta^{s} \mathfrak{P}$ where $0 \leqslant k \leqslant s \leqslant l$. As double cosets of a group are disjoint or equal, so $\psi^s(x) \neq 0$ and $\psi^i(x)=0$ for $0 \leqslant k \leqslant i \leqslant l, i \neq s$. Hence $c_k \psi^k(x)+ \cdots +c_l \psi^l(x)=0$ would imply that $c_s=0$. In a similar way we can show that $c_k=c_{k+1}=\ldots =c_l =0$. So $\lbrace \psi^k, \psi^{k+1}, \ldots ,\psi^l \rbrace$ is a linearly independent set when $0 \leqslant k <l ; k,l \in \mathbb{Z}$. Now suppose if $k<0$ and let $c_k \psi^k+ \cdots +c_l \psi^l=0$ with $c_k, \ldots , c_l \in \mathbb{C}; k,l \in \mathbb{Z}$. Let us assume without loss of generality that $k<l$. Multiplying throughout the above expression by $\psi^{-k}$, we have $c_k + \cdots +c_l \psi^{l-k}=0$. Now repeating the previous argument we have $c_k=c_{k+1}=\ldots =c_l =0$. So again $\lbrace \psi^k, \psi^{k+1}, \ldots ,\psi^l \rbrace$ is a linearly independent set when $k<0; k<l; k,l \in \mathbb{Z}$. Hence $\mathcal{B} \simeq \mathbb{C}[x, x^{-1}]$. \par
\end{proof}

\subsection{Ramified case:}

In this section we determine the structure of $\mathcal{H}(G,\rho)$ for the ramified case when $n$ is even. Recall $ \mathfrak{I}_G(\rho)= \mathfrak{P}N_G(\rho_0)\mathfrak{P}$. But from lemma \ref{lem_2}, $N_G(\rho_0)= \left<\mathfrak{P}_0, w_0, w_1\right>$. So $\mathfrak{I}_G(\rho)=\mathfrak{P}\left<\mathfrak{P}_0, w_0, w_1\right>\mathfrak{P}= \mathfrak{P}\left<w_0, w_1\right> \mathfrak{P}$, as $\mathfrak{P}_0$ is a subgroup of $\mathfrak{P}$. Let $V$ be the vector space corresponding to $\rho$. Let us recall that  $\mathcal{H}(G, \rho)$ consists of maps $f \colon G \to End_{\mathbb{C}}(V^{\vee})$ such that support of $f$ is compact and $f(pgp')= \rho^{\vee}(p)f(g)\rho^{\vee}(p')$ for $p,p' \in \mathfrak{P}, g \in G$. In fact $\mathcal{H}(G,\rho)$ consists of $\mathbb{C}$-linear combinations of maps $f \colon G \longrightarrow End_{\mathbb{C}}(V^{\vee})$ such that $f$ is supported on $\mathfrak{P}x\mathfrak{P}$ where $x \in \mathfrak{I}_G(\rho)$ and $f(pxp')= \rho^{\vee}(p)f(x)\rho^{\vee}(p')$ for $p,p' \in \mathfrak{P}$. We shall now show there exists $\phi_0 \in \mathcal{H}(G,\rho)$ with support $\mathfrak{P}w_0\mathfrak{P}$ and satisfies $\phi_0^2= q^{n/2}+ (q^{n/2}-1)\phi_0$. Let

\[
K(0)= \mathrm{U}(n,n) \cap \mathrm{GL}_{2n}(\mathfrak{O}_E)=\{g \in \mathrm{GL}_{2n}(\mathfrak{O}_E)\mid ^t\overline{g}Jg=J\},
\]
\[
K_1(0)=\{g \in 1+ \varpi_E \mathrm{M}_{2n}(\mathfrak{O}_E)\mid ^t\overline{g}Jg=J\},
\]
\[
\mathsf{G}= \{g \in \mathrm{GL}_{2n}(k_E) \mid ^t\overline{g}Jg=J\}.
\]

The map $r$ from $K(0)$ to $\mathsf{G}$ given by $r\colon K(0)\xrightarrow{\text{mod} \, p_E}\mathsf{G}$ is a surjective group homomorphism with kernel $K_1(0)$. So by the first isomorphism theorem of groups we have:

\begin{center}
	$\frac{K(0)}{K_1(0)}\cong \mathsf{G}$.
\end{center}

$r(\mathfrak{P})= \mathsf{P}=
\begin{bmatrix}
\mathrm{GL}_n(k_E) &  \mathrm{M}_n(k_E) \\
0         &  \mathrm{GL}_n(k_E)
\end{bmatrix} \bigcap \mathsf{G}$= Siegel parabolic subgroup of $\mathsf{G}$.\\

Now $\mathsf{P}= \mathsf{L} \ltimes \mathsf{U}$, where $\mathsf{L}$ is the Siegel Levi component of $\mathsf{P}$ and $\mathsf{U}$ is the unipotent radical of $\mathsf{P}$. Here

\[
\mathsf{L}= \Bigg \lbrace
\begin{bmatrix}
a & 0\\
0 & ^t\overline{a}^{-1}\\
\end{bmatrix} \mid a \in \mathrm{GL}_n(k_E) \Bigg \rbrace,
\]

\[
\mathsf{U}= \Bigg \lbrace
\begin{bmatrix}
1_n & X \\
0  & 1_n \\
\end{bmatrix} \mid X \in \mathrm{M}_n(k_E), X+{^t}\overline{X}=0 \Bigg \rbrace.
\]

Let $V$ be the vector space corresponding to $\rho$. The Hecke algebra $\mathcal{H}(K(0),\rho)$ is a sub-algebra of $\mathcal{H}(G,\rho)$.\par

Let $\overline{\rho}$ be the representation of $\mathsf{P}$ which when inflated to $\mathfrak{P}$ is given by $\rho$ and $V$ is also the vector space corresponding to $\overline{\rho}$. The Hecke algebra $\mathcal{H}(\mathsf{G},\overline{\rho})$ looks as follows:

\[
\mathcal{H}(\mathsf{G},\overline{\rho})= \left\lbrace f \colon \mathsf{G} \to End_{\mathbb{C}}(V^{\vee}) \; \middle|  \;
\begin{varwidth}{\linewidth}
$f(pgp')= \overline{\rho}^{\vee}(p)f(g)\overline{\rho}^{\vee}(p')$\\
where $p,p' \in \mathsf{P}, \, g \in \mathsf{G}$
\end{varwidth}
\right \rbrace.
\] \par

Now the homomorphism $r \colon K(0) \longrightarrow \mathsf{G}$ extends to a map from  $\mathcal{H}(K(0),\rho)$ to $\mathcal{H}(\mathsf{G},\overline{\rho})$ which we again denote by $r$. Thus $r \colon \mathcal{H}(K(0),\rho) \longrightarrow \mathcal{H}(\mathsf{G},\overline{\rho})$ is given by

\[
r(\phi)(r(x))=\phi(x)
\]
\[ 
\text{for} \, \phi \in \mathcal{H}(K(0),\rho) \, \text{and} \, x \in K(0).
\]

As in the unramified case, when $n$ is odd, we can show that $\mathcal{H}(K(0),\rho)$ is isomorphic to $\mathcal{H}(\mathsf{G},\overline{\rho})$ as algebras via $r$.\\\par

Let $w=r(w_0)=r(
\begin{bmatrix}
0 & 1_n \\
1_n & 0
\end{bmatrix})=
\begin{bmatrix}
0 & 1_n \\
1_n & 0
\end{bmatrix} \in \mathsf{G}$. Clearly $K(0) \supseteq \mathfrak{P}\amalg\mathfrak{P}w_0\mathfrak{P}\Longrightarrow r(K(0))\supseteq r(\mathfrak{P}\amalg\mathfrak{P}w_0\mathfrak{P})\Longrightarrow\mathsf{G}\supseteq r(\mathfrak{P})\amalg r(\mathfrak{P}w_0\mathfrak{P})=\mathsf{P}\amalg \mathsf{P}w\mathsf{P}$. So $\mathsf{G}\supseteq \mathsf{P}\amalg \mathsf{P}w\mathsf{P}$.\par

Now $\mathsf{G}$ is a finite group. In fact, it is the special orthogonal group consisting of matrices of size $2n \times 2n$ over finite field $k_E$ or $\mathbb{F}_q$. So $\mathsf{G} = SO_{2n}(\mathbb{F}_q)$. \par 

According to the Theorem 6.3 in \cite{MR2276353}, there exists a unique $\phi$ in  $\mathcal{H}(\mathsf{G}, \overline{\rho})$ with support $\mathsf{P} w \mathsf{P}$ such that $\phi^2= q^{n/2}+ (q^{n/2}-1)\phi$. Hence there is a unique element $\phi_0 \in \mathcal{H}(K(0),\rho)$ such that $r(\phi_0)=\phi$. Thus supp($\phi_0$)=$\mathfrak{P} w_0 \mathfrak{P}$ and $\phi_0^2= q^{n/2}+ (q^{n/2}-1)\phi_0$. Now $\phi_0$ can be extended to $G$ and viewed as an element in $\mathcal{H}(G, \rho)$ as $\mathfrak{P} w_0 \mathfrak{P} \subseteq K(0) \subseteq G$. Thus $\phi_0$ satisfies the following relation in $\mathcal{H}(G,\rho)$:

\[
\phi_0^2= q^{n/2}+ (q^{n/2}-1)\phi_0.
\]\par 

We shall now show there exists $\phi_1 \in \mathcal{H}(G, \rho)$ with support $\mathfrak{P}w_1\mathfrak{P}$ satisfying the same relation as $\phi_0$.\par

We know that $w_1=
\begin{bmatrix}
0 & \overline\varpi^{-1}_E 1_n \\
\varpi_E 1_n & 0
\end{bmatrix}, \overline\varpi^{-1}_E= -\varpi_E$. 
So $w_1= 
\begin{bmatrix}
0 & -\varpi^{-1}_E 1_n\\
\varpi_E 1_n & 0
\end{bmatrix}$. Let $\eta= 
\begin{bmatrix}
\varpi_E 1_n & 0 \\
0  & 1_n
\end{bmatrix}$. So, $\eta w_1 \eta^{-1}= J'=
\begin{bmatrix}
0 & -1_n\\
1_n & 0
\end{bmatrix}$. Recall that $\mathfrak{P}$ looks as follows:

\begin{center}
	$\mathfrak{P}=\begin{bmatrix}
	\mathrm{GL}_n(\mathfrak{O}_E) & \mathrm{M}_n(\mathfrak{O}_E)\\
	\mathrm{M}_n(p_E) & \mathrm{GL}_n(\mathfrak{O}_E)
	\end{bmatrix} \bigcap G.$
\end{center}\par

Now
\[
\eta
\begin{bmatrix}
\mathrm{GL}_n(\mathfrak{O}_E) & \mathrm{M}_n(\mathfrak{O}_E)\\
\mathrm{M}_n(p_E) & \mathrm{GL}_n(\mathfrak{O}_E)
\end{bmatrix}\eta^{-1}=
\begin{bmatrix}
\mathrm{GL}_n(\mathfrak{O}_E) & \mathrm{M}_n(p_E)\\
\mathrm{M}_n(\mathfrak{O}_E) & \mathrm{GL}_n(\mathfrak{O}_E)
\end{bmatrix},
\]

\[\eta G \eta^{-1}=G'= \lbrace g \in \mathrm{GL}_{2n}(E) \mid ^t \overline{g} J' g= J'\rbrace.\]

Hence
\begin{center}
	$\eta\mathfrak{P}\eta^{-1}=\begin{bmatrix}
	\mathrm{GL}_n(\mathfrak{O}_E) & \mathrm{M}_n(p_E)\\
	\mathrm{M}_n(\mathfrak{O}_E) & \mathrm{GL}_n(\mathfrak{O}_E)
	\end{bmatrix} \bigcap G'.$
\end{center}\par

Therefore $\eta\mathfrak{P}\eta^{-1}$ is the opposite of the Siegel Parahoric subgroup of $G'$. Let

\begin{center}
	$K'(0)= \langle\mathfrak{P}, w_1\rangle.$
\end{center}

And let
\begin{align*}
\mathsf{G}' &=\lbrace g \in \mathrm{GL}_{2n}(k_E) \mid ^t\overline{g}J'g=J'\rbrace\\
&= \lbrace g \in \mathrm{GL}_{2n}(k_E) \mid ^tgJ'g=J'\rbrace.
\end{align*} 

Let $r'\colon K'(0) \longrightarrow \mathsf{G}'$ be the group homomorphism given by
\[r'(x)= (\eta x \eta^{-1})mod p_E \,\text{where} \, x \in K'(0).\]

So we have $r'(K(0))= (\eta K'(0) \eta^{-1}) mod p_E= (\eta\langle \mathfrak{P}, w_1 \rangle\eta^{-1})mod p_E$. Let \[r'(\mathfrak{P})=(\eta \mathfrak{P} \eta^{-1})mod p_E = \overline{\mathsf{P}}'.\] We can see that  $r'(w_1)= (\eta w_1 \eta^{-1}) mod p_E=J' mod p_E=w'=\begin{bmatrix}
0 & -1_n\\
1_n & 0
\end{bmatrix}$. So $\overline {\mathsf{P}}'= r'(\mathfrak{P})= (\eta \mathfrak{P} \eta^{-1}) mod p_E=
\begin{bmatrix}
\mathrm{GL}_n(k_E) & 0\\
\mathrm{M}_n(k_E) & \mathrm{GL}_n(k_E)
\end{bmatrix} \bigcap \mathsf{G}'$. Clearly $\overline {\mathsf{P}}'$ is the opposite of Siegel parabolic subgroup of $\mathsf{G}'$. Hence  $r'(K(0))= \langle \overline {\mathsf{P}}', w' \rangle= \mathsf{G}'$, as $\overline {\mathsf{P}}'$ is a maximal subgroup of $\mathsf{G}'$ and $w'$ does not lie in $\overline {\mathsf{P}}'$. So $r'$ is a surjective homomorphism of groups.\par

Let $V$ be the vector space corresponding to $\rho$. The Hecke algebra $\mathcal{H}(K'(0),\rho)$ is a sub-algebra of $\mathcal{H}(G,\rho)$.\par

Let $\overline{\rho}'$ be the representation of $\overline{\mathsf{P}}'$ which when inflated to $^{\eta}\mathfrak{P}$ is given by $^\eta{\rho}$ and $V$ is also the vector space corresponding to $\overline{\rho}'$. Now the Hecke algebra $\mathcal{H}(\mathsf{G}',\overline{\rho}')$ looks as follows:

\[
\mathcal{H}(\mathsf{G}',\overline{\rho}')= \left\lbrace f \colon \mathsf{G}' \to End_{\mathbb{C}}(V^{\vee}) \; \middle|  \;
\begin{varwidth}{\linewidth}
$f(pgp')= \overline{\rho}'^{\vee}(p)f(g)\overline{\rho}'^{\vee}(p')$\\
where $p,p' \in \overline{\mathsf{P}}', \, g \in \mathsf{G}'$
\end{varwidth}
\right \rbrace.
\]\par

Now the homomorphism $r' \colon K'(0) \longrightarrow \mathsf{G'}$ extends to a map from  $\mathcal{H}(K'(0),\rho)$ to $\mathcal{H}(\mathsf{G}',\overline{\rho}')$ which we again denote by $r'$. Thus $r' \colon \mathcal{H}(K'(0),\rho) \longrightarrow \mathcal{H}(\mathsf{G}',\overline{\rho}')$ is given by

\[
r'(\phi)(r'(x))=\phi(x)
\]
\[ 
\text{for} \, \phi \in \mathcal{H}(K'(0),\rho) \, \text{and} \, x \in K'(0).
\]

As in the unramified case when $n$ is odd, we can show that $\mathcal{H}(K'(0),\rho)$ is isomorphic to $\mathcal{H}(\mathsf{G}',\overline{\rho}')$ as algebras via $r'$.\par

Clearly $K'(0) \supseteq \mathfrak{P}\amalg\mathfrak{P}w_1\mathfrak{P}\Longrightarrow r'(K'(0))\supseteq r'(\mathfrak{P}\amalg \mathfrak{P}w_1\mathfrak{P})\Longrightarrow \mathsf{G}'\supseteq r'(\mathfrak{P})\amalg r'(\mathfrak{P}w_1\mathfrak{P})=\overline{\mathsf{P}}'\amalg \overline{\mathsf{P}}'w'\overline{\mathsf{P}}'$. So $\mathsf{G}'\supseteq \overline{\mathsf{P}}'\amalg 
\overline{\mathsf{P}}'w'\overline{\mathsf{P}}'$.\par

Now $\mathsf{G}'$ is a finite group. In fact, it is the symplectic group consisting of matrices of size $2n \times 2n$ over finite field $k_E$ or $\mathbb{F}_q$. So $\mathsf{G}' = Sp_{2n}(\mathbb{F}_q)$. \par 

According to the Theorem 6.3 in \cite{MR2276353}, there exists a unique $\phi$ in  $\mathcal{H}(\mathsf{G}', \overline{\rho}')$ with support $\overline{\mathsf{P}}' w' \overline{\mathsf{P}}'$ such that $\phi^2= q^{n/2}+ (q^{n/2}-1)\phi$. Hence there is a unique element $\phi_1 \in \mathcal{H}(K'(0),\rho)$ such that $r'(\phi_1)=\phi$. Thus supp($\phi_1$)=$\mathfrak{P} w_1 \mathfrak{P}$ and $\phi_1^2= q^{n/2}+ (q^{n/2}-1)\phi_1$. Now $\phi_1$ can be extended to $G$ and viewed as an element in $\mathcal{H}(G, \rho)$ as $\mathfrak{P} w_1 \mathfrak{P} \subseteq K'(0) \subseteq G$. Thus $\phi_1$ satisfies the following relation in $\mathcal{H}(G,\rho)$:

\[
\phi_1^2= q^{n/2}+ (q^{n/2}-1)\phi_1.
\]\par 

Thus we have shown there exists $\phi_i \in \mathcal{H}(G,\rho)$ with supp($\phi_i$)=$\mathfrak{P}w_i\mathfrak{P}$ satisfying $\phi_i^2= q^{n/2}+ (q^{n/2}-1)\phi_i$  for $i=0,1$. It can be further shown that $\phi_0$ and $\phi_1$ generate the Hecke algebra $\mathcal{H}(G, \rho)$. Let us denote the Hecke algebra $\mathcal{H}(G, \rho)$ by $\mathcal{A}$. So 

\[
\mathcal{A}= \mathcal{H}(G,\rho)= \left\langle \phi_i \colon G \to End_{\mathbb{C}}(\rho^{\vee})\; \middle| \;
\begin{varwidth}{\linewidth} 
$\phi_i$ is supported on 
$\mathfrak{P}w_i\mathfrak{P}$\\
 $\,\text{and} \, \phi_i(pw_ip')= \rho^{\vee}(p)\phi_i(w_i)\rho^{\vee}(p')$\\
where $p,p' \in \mathfrak{P}, \,i=0,1$
\end{varwidth}
\right\rangle
\] where $\phi_i$ has support $\mathfrak{P} w_i \mathfrak{P}$  and $\phi_i$ satisfies the relation:

\begin{center}
	$\phi_i^2= q^{n/2}+ (q^{n/2}-1) \phi_i$ for $i=0,1$.
\end{center}  

\begin{lemma}\label{lem_7}
	$\phi_0$ and $\phi_1$ are units in $\mathcal{A}$.
\end{lemma}
\begin{proof}
	As $\phi_i^2= q^{n/2}+ (q^{n/2}-1) \phi_i$ for $i=0,1$. So $\phi_i( \frac{\phi_i+(1-q^{n/2})1}{q^{n/2}})= 1$ for i=0,1. Hence $\phi_0$ and $\phi_1$ are units in $\mathcal{A}$.
\end{proof}\par

As $\phi_0, \phi_1$ are units in $\mathcal{A}$ which is an algebra, so $\psi= \phi_0\phi_1$ is a unit too in $\mathcal{A}$ and $\psi^{-1}= \phi_1^{-1}\phi_0^{-1}$. As in the unramified case when $n$ is odd, we can show that $\mathcal{A}$ contains sub-algebra $\mathcal{B}=\mathbb{C}[\psi,\psi^{-1}]$ where

\[
\mathcal{B}=\mathbb{C}[\psi,\psi^{-1}] = \left\lbrace c_k\psi^k + \cdots +c_l \psi^l \; \middle| \;
\begin{varwidth}{\linewidth} 
$c_k, \ldots ,c_l \in \mathbb{C};$\\
$k <l ; k,l \in \mathbb{Z}$
\end{varwidth}
\right\rbrace.
\]
\par 

Further, as in the unramified case when $n$ is odd, we can show that $\mathbb{C}[\psi,\psi^{-1}] \simeq \mathbb{C}[x,x^{-1}]$ as $\mathbb{C}$-algebras. \par

\section{Structure of $\mathcal{H}(L,\rho_0)$} \label{sec_7}

In this section we describe the structure of $\mathcal{H}(L,\rho_0)$. Thus we need first to determine 

\[
N_L(\rho_0)=\lbrace m \in N_L(\mathfrak{P}_0) \mid \rho_0^m \simeq \rho_0 \rbrace. 
\] \par

We know from lemma \ref{Normalizer_of_K_0_in_GL_n(E)} that $N_{\mathrm{GL}_n(E)}(K_0)= K_0Z$, so we have $N_L(\mathfrak{P}_0)= Z(L)\mathfrak{P}_0$. Since $Z(L)$ clearly normalizes $\rho_0$  and $\rho_0$ is an irreducible supercuspidal representation of $\mathfrak{P}_0$, so $N_L(\rho_0)= Z(L)\mathfrak{P}_0$.\par

Now that we have calculated $N_L(\rho_0)$, we determine the structure of $\mathcal{H}(L,\rho_0)$. Using the cuspidality of $\rho_0$, it can be shown by A.1 Appendix \cite{MR1235019} that $\mathfrak{I}_L(\rho_0)= \mathfrak{P}_0 N_L(\rho_0)\mathfrak{P}_0$. As $N_L(\rho_0)= Z(L)\mathfrak{P}_0$, so $\mathfrak{I}_L(\rho_0)= \mathfrak{P}_0 Z(L)\mathfrak{P}_0 \mathfrak{P}_0 = Z(L)\mathfrak{P}_0$. Let $V$ be the vector space of $\rho_0$.\par

The Hecke algebra $\mathcal{H}(L,\rho_0)$ consists of $\mathbb{C}$-linear combinations of maps $f \colon L \longrightarrow End_{\mathbb{C}}(V^{\vee})$ such that each map $f$ is supported on $\mathfrak{P}_0x\mathfrak{P}_0$ where $x \in \mathfrak{I}_L(\rho_0)=Z(L)\mathfrak{P}_0$ and $f(pxp')=\rho_0^{\vee}(p)f(x)\rho_0^{\vee}(p')$ for $p,p' \in \mathfrak{P}_0$. It is clear that

\[
Z(L)\mathfrak{P}_0= \coprod_{n \in \mathbb{Z}} \mathfrak{P}_0 \zeta^n.
\]

So the Hecke algebra $\mathcal{H}(L,\rho_0)$ consists of $\mathbb{C}$-linear combinations of maps $f \colon L \longrightarrow End_{\mathbb{C}}(V^{\vee})$ such that each map $f$ is supported on $\mathfrak{P}_0x\mathfrak{P}_0$ where $x \in \mathfrak{P}_0 \zeta^n$ with $n \in \mathbb{Z}$ and  
$f(pxp')=\rho_0^{\vee}(p)f(x)\rho_0^{\vee}(p')$ for $p,p' \in \mathfrak{P}_0$.\par

Let $\phi_1, \phi_2 \in \mathcal{H}(L,\rho_0)$ with $\text{supp}(\phi_1) =\mathfrak{P}_0 z_1$ and $\text{supp}(\phi_2) =\mathfrak{P}_0 z_2$ respectively with $z_1, z_2 \in Z(L)$. As $\rho_0$ is an irreducible  supercuspidal representation of $\mathfrak{P}_0$. So if $f \in \mathcal{H}(L,\rho_0)$ with $\text{supp}(f)=\mathfrak{P}_0 z$ where $z \in Z(L)$ then from Schur's lemma $f(z)=c 1_{V^{\vee}}$ for some $c \in \mathbb{C}^{\times}$. Hence $\phi_1(z_1)= c_1 1_{V^{\vee}}$ and $\phi_2(z_2)=c_2 1_{V^{\vee}}$ where $c_1,c_2 \in \mathbb{C}^{\times}$.\par

We have $\text{supp}(\phi_1\phi_2) \subseteq (\text{supp}(\phi_1))(\text{supp}(\phi_2))= \mathfrak{P}_0 z_1 \mathfrak{P}_0 z_2= \mathfrak{P}_0 z_1 z_2$. The proof goes in the similar lines as lemma \ref{lem_6}.\par

We assume without loss of generality that $\text{vol}\mathfrak{P}_0= \text{vol}\mathfrak{P}_{-}= \text{vol}\mathfrak{P}_{+}=1$. Thus we have $\text{vol}\mathfrak{P}=1$.

\begin{lemma}\label{convolution_lemma}
	Let $\phi_1, \phi_2 \in \mathcal{H}(L,\rho_0)$ with $\text{supp}(\phi_1)=\mathfrak{P}_0z_1$ and  $\text{supp}(\phi_2)=\mathfrak{P}_0 z_2$ where $z_1, z_2 \in Z(L)$. Also let $\phi_1(z_1)=c_1 1_{V^{\vee}}$ and $\phi_2(z_2)=c_2 1_{V^{\vee}}$ where $c_1,c_2 \in \mathbb{C}^{\times}$. Then $(\phi_1*\phi_2)(z_1z_2)= \phi_1(z_1)\phi_2(z_2)=c_1c_2 1_{V^{\vee}}$.	
\end{lemma}
\begin{proof}
	
	\begin{align*}
	(\phi_1*\phi_2)(z_1z_2)&= \int_L \phi_1(z_1z_2y^{-1})\phi_2(y)dp\\
	&=\int_{\mathfrak{P}_0}\phi_1(z_1z_2z_2^{-1}p^{-1})\phi_2(z_2p)dp\\
	&=\int_{\mathfrak{P}_0}\phi_1(z_1p^{-1})\phi_2(pz_2)dp\\
	&=\int_{\mathfrak{P}_0}\phi_1(z_1)\rho_0^{\vee}(p^{-1})\rho_0^{\vee}(p)\phi_2(z_2)dp\\
	&=\int_{\mathfrak{P}_0}\phi_1(z_1)\phi_2(z_2)dp\\
	&=\int_{\mathfrak{P}_0}c_1c_2 1_{V^{\vee}}dp\\
	&=c_1c_2\text{Vol}(\mathfrak{P}_0)1_{V^{\vee}}\\
	&=c_1c_2 1_{V^{\vee}}\\
	&=\phi_1(z_1)\phi_2(z_2).
	\end{align*}
\end{proof}
As $\text{supp}(\phi_1*\phi_2)= \text{supp}(\phi_1\phi_2) \subseteq \mathfrak{P}_0 z_1 z_2$, so $\text{supp}(\phi_1*\phi_2)= \varnothing$ or $\mathfrak{P}_0z_1z_2$. If $\text{supp}(\phi_1*\phi_2)= \varnothing$ then it means that $(\phi_1*\phi_2)=0$. This contradicts $(\phi_1*\phi_2)(z_1z_2)=c_1c_2 \neq 0$. So $\text{supp}(\phi_1*\phi_2)= \mathfrak{P}_0z_1z_2$.\par

This implies that $\phi_1$ is invertible and $\phi_1^{-1}$ be it's inverse. Thus $\text{supp}(\phi_1^{-1})=\mathfrak{P}_0z_1^{-1}$ and $\phi_1^{-1}(z_1^{-1})=c_1^{-1}1_{V^{\vee}}$.\par

Define $\alpha \in \mathcal{H}(L,\rho_0)$ by $\text{supp}(\alpha)= \mathfrak{P}_0\zeta$ 
and $\alpha(\zeta)=1_{V^{\vee}}$.\par

\begin{proposition}\label{pro_20}
	\begin{enumerate}
		\item $\alpha^n(\zeta^n) = (\alpha(\zeta))^n$ for $n \in \mathbb{Z}.$
		\item $\text{supp}(\alpha^n)= \mathfrak{P}_0 \zeta^n \mathfrak{P}_0= \mathfrak{P}_0 \zeta^n =  \zeta^n \mathfrak{P}_0$ for $n \in \mathbb{Z}.$
	\end{enumerate}
\end{proposition}

\begin{proof}
Using lemma \ref{convolution_lemma} over and over we get,  $\alpha^n(\zeta^n) = (\alpha(\zeta))^n$ for $n \in \mathbb{Z}$ and $\text{supp}(\alpha^n)= \mathfrak{P}_0 \zeta^n \mathfrak{P}_0= \mathfrak{P}_0 \zeta^n = \zeta^n \mathfrak{P}_0$ for $n \in \mathbb{Z}$
\end{proof}\par

We know that $\mathcal{H}(L,\rho_0)$ consists of $\mathbb{C}$-linear combinations of maps $f \colon L \longrightarrow End_{\mathbb{C}}(V^{\vee})$ such that each map $f$ is supported on $\mathfrak{P}_0x\mathfrak{P}_0$ where $x \in \mathfrak{P}_0 \zeta^n$ with $n \in \mathbb{Z}$ and $f(pxp')=\rho_0^{\vee}(p)f(x)\rho_0^{\vee}(p')$ for $p,p' \in \mathfrak{P}_0$. So from Proposition ~\ref{pro_20}, $\mathcal{H}(L,\rho_0)$ is generated as a $\mathbb{C}$-algebra by $\alpha$ and $\alpha^{-1}$. Hence $\mathcal{H}(L,\rho_0)= \mathbb{C}[\alpha, \alpha^{-1}]$.\par 

\begin{proposition}\label{algebra_isomorphism_1}
	The unique algebra homomorphism $\mathbb{C}[x, x^{-1}] \longrightarrow \mathbb{C}[\alpha, \alpha^{-1}]$ given by $x \longrightarrow \alpha$ is an isomorphism. So $\mathbb{C}[\alpha, \alpha^{-1}] \simeq \mathbb{C}[x, x^{-1}]$. 
\end{proposition}

We have already shown before in sections 6.1 and 6.2 that $\mathcal{B}= \mathbb{C}[\psi,\psi^{-1}]$ is a sub-algebra of $\mathcal{A}= \mathcal{H}(G,\rho)$, where $\psi$ is supported on $\mathfrak{P}\zeta\mathfrak{P}$ and $\mathcal{B} \cong \mathbb{C}[x,x^{-1}]$. As $\mathcal{H}(L,\rho_0)= \mathbb{C}[\alpha,\alpha^{-1}]\cong \mathbb{C}[x,x^{-1}]$, so  $\mathcal{B} \cong \mathcal{H}(L,\rho_0)$ as $\mathbb{C}$-algebras. Hence $\mathcal{H}(L,\rho_0)$ can be viewed as a sub-algebra of $\mathcal{H}(G,\rho)$.\par 

Now we would like to find out how simple $\mathcal{H}(L,\rho_0)$-modules look like. Thus to understand them we need to find out how simple $\mathbb{C}[x,x^{-1}]$-modules look like.

\section{Calculation of simple $\mathcal{H}(L,\rho_0)$-modules} \label{sec_8}

Recall that $\mathcal{H}(L,\rho_0) = \mathbb{C}[\alpha, \alpha^{-1}]$. Note that $\mathbb{C}[\alpha, \alpha^{-1}] \cong \mathbb{C}[x,x^{-1}]$ as $\mathbb{C}$-algebras. It can be shown by direct calculation that the simple $\mathbb{C}[x,x^{-1}]$-modules are of the form $\mathbb{C}_{\lambda}$ for $\lambda \in \mathbb{C}^{\times}$, where $\mathbb{C}_{\lambda}$ is the vector space $\mathbb{C}$ with the $\mathbb{C}[x,x^{-1}]$-module structure given by $x.z=\lambda z$ for $z \in \mathbb{C}_{\lambda}$. \par

So the distinct simple $\mathcal{H}(L,\rho_0)$-modules(up to isomorphism)  are the various $\mathbb{C}_{\lambda}$ for $\lambda \in \mathbb{C}^{\times}$. The module structure is determined by $\alpha.z = \lambda z$ for $z \in \mathbb{C}_{\lambda}$.

\section{Final calculations to answer the question} \label{sec_9}

\subsection{Calculation of $\delta_P(\zeta)$}

Let us recall the modulus character $\delta_P \colon P \longrightarrow \mathbb{R}^{\times}_{>0}$ introduced in section 1. The character $\delta_P$ is given by $\delta_P(p)=\|det(Ad \, p)|_{\mathrm{Lie} \, U}\|_F$ for $p \in P$, where $\mathrm{Lie} \, U$ is the Lie algebra of $U$. We have

\[
U= \Bigg\lbrace
\begin{bmatrix} 
1_n & X\\
0  & 1_n
\end{bmatrix}
\mid X \in \mathrm{M}_n(E),X+^t{\overline{X}}=0
\Bigg\rbrace,
\]

\[
\mathrm{Lie} \, U= \Bigg\lbrace
\begin{bmatrix} 
0 &  X\\
0  & 0 
\end{bmatrix}
\mid X \in \mathrm{M}_n(E),X+^t{\overline{X}}=0
\Bigg\rbrace.
\]

\subsubsection{Unramified case:}

Recall $\zeta=
\begin{bmatrix}
\varpi_E 1_n & 0\\
0 & \varpi^{-1}_E 1_n
\end{bmatrix}$ in the unramified case. So

\[
(Ad \, \zeta)
\begin{bmatrix}
0 & X \\
0 & 0
\end{bmatrix}= \zeta 
\begin{bmatrix}
0 & X \\
0 & 0
\end{bmatrix}\zeta^{-1}=
\begin{bmatrix}
0 & \varpi_E^2 X\\
0 & 0
\end{bmatrix}.
\]
Hence
\begin{align*}
\delta_P(\zeta)&=\|det(Ad \, \zeta)|_{\mathrm{Lie} \, U}\|_F\\
&= \| \varpi_E^{2(dim_F(\mathrm{Lie} \, U))}\|_F\\
&=\|\varpi_E^{2n^2} \|_F\\
&=\|\varpi_F^{2n^2} \|_F\\
&= q^{-2n^2}.
\end{align*}

\subsubsection{Ramified case:}

Recall $\zeta=
\begin{bmatrix}
\varpi_E 1_n & 0\\
0 & -\varpi^{-1}_E 1_n
\end{bmatrix}$ in the ramified case and that analogously as in
unramified case we obtain $\delta_P(\zeta)= q^{-n^2}$. 

\subsection{Understanding the map $T_P$}

Let us denote the set of strongly $(\mathfrak{P},P)$-positive elements by $\mathcal{I}^{+}$. Thus

\[
\mathcal{I}^{+}=\{x \in L \mid x\mathfrak{P}_{+}x^{-1}\subseteq\mathfrak{P}_{+}, x^{-1}\mathfrak{P}_{-}x \subseteq \mathfrak{P}_{-} \}.
\] \par

where $\mathfrak{P}_{+}=\mathfrak{P} \cap U, \mathfrak{P}_{-}= \mathfrak{P} \cap \overline{U}$. We have
\[
\mathcal{H}^{+}(L, \rho_0)= \{f \in \mathcal{H}(L, \rho_0) \mid \text{supp}f \subseteq \mathfrak{P}_0\mathcal{I}^{+}\mathfrak{P}_0 \}.
\]\par

Note $\zeta \in \mathcal{I}^{+}$, so $\mathcal{H}^{+}(L, \rho_0)= \mathbb{C}[\alpha]$. The following discussion is taken from pages 612-619 in \cite{MR1643417}. Let $W$ be space of $\rho_0$. Let $f \in \mathcal{H}^{+}(L, \rho_0)$ with support of $f$ being $\mathfrak{P}_0 x\mathfrak{P}_0$ for $x \in  \mathcal{I}^{+}$. The map $F \in \mathcal{H}(G, \rho)$ is supported on $\mathfrak{P} x \mathfrak{P}$ and $f(x)=F(x)$. The algebra embedding 
\[T^{+} \colon \mathcal{H}^{+}(L, \rho_0) \longrightarrow \mathcal{H}(G, \rho)\] is given by $T^{+}(f)=F$.\par

Recall support of $\alpha \in \mathcal{H}^{+}(L, \rho_0)$ is $\mathfrak{P}_0 \zeta$. Let $T^{+}(\alpha)= \psi$, where $\psi \in \mathcal{H}(G, \rho)$ has support $\mathfrak{P} \zeta \mathfrak{P}$ and $\alpha(\zeta)=\psi(\zeta)=1_{W^{\vee}}$. As $T^{+}(\alpha)= \psi$ is invertible, so from Proposition ~\ref{pro_3} we can conclude that $T^{+}$ extends to an embedding of algebras  
\[t \colon \mathcal{H}(L, \rho_0) \longrightarrow \mathcal{H}(G, \rho).\]

Let $\phi \in \mathcal{H}(L, \rho_0)$  and $m \in \mathbb{N}$ is chosen such that $\alpha^m \phi \in \mathcal{H}^{+}(L, \rho_0)$. The map $t$ is then given by $t(\phi)= \psi^{-m}T^{+}(\alpha^m \phi)$. For  $\phi \in \mathcal{H}(L,\rho_0)$, the map 
\[t_P \colon \mathcal{H}(L,\rho_0) \longrightarrow \mathcal{H}(G, \rho)\] is given by $t_P(\phi)= t(\phi \delta_P)$, where $\phi \delta_P \in \mathcal{H}(L,\rho_0)$ and is the map \[\phi \delta_P \colon L \longrightarrow End_{\mathbb{C}}(\rho_0^{\vee})\] given by $(\phi \delta_P)(l)=\phi(l)\delta_P(l)$ for $l \in L$. As $\alpha \in \mathcal{H}(L, \rho_0)$ we have
\begin{align*}
t_P(\alpha)(\zeta) &= t(\alpha \delta_P)(\zeta)\\
&= T^{+}(\alpha \delta_P)(\zeta)\\
&= \delta_P(\zeta)T^{+}(\alpha)(\zeta)\\
&= \delta_P(\zeta)\psi(\zeta)\\
&= \delta_P(\zeta)1_{W^{\vee}}.
\end{align*}\par

Let $\mathcal{H}(L, \rho_0)$-Mod denote the category of $\mathcal{H}(L, \rho_0)$-modules and  $\mathcal{H}(G, \rho)$-Mod denote the category of $\mathcal{H}(G, \rho)$-modules. The map $t_P$ induces a functor $(t_P)_{*}$ given by

\[(t_P)_{*} \colon \mathcal{H}(L, \rho_0)-\text{Mod} \longrightarrow \mathcal{H}(G, \rho)-\text{Mod}.\] 

For $M$ an $\mathcal{H}(L, \rho_0)$-module,

\[(t_P)_{*}(M)= \mathrm{Hom}_{\mathcal{H}(L, \rho_0)}(\mathcal{H}(G, \rho),M)\] where $\mathcal{H}(G, \rho)$ is viewed as a $\mathcal{H}(L, \rho_0)$-module via $t_P$. The action of $\mathcal{H}(G,\rho)$ on $(t_P)_*(M)$ is given by 

\[h'\psi(h_1)=\psi(h_1 h')\] where $\psi \in (t_P)_*(M), h_1, h' \in \mathcal{H}(G, \rho)$.\par

Let $\tau \in \mathfrak{R}^{[L,\pi]_L}(L)$ then functor $m_L \colon \mathfrak{R}^{[L,\pi]_L}(L) \longrightarrow \mathcal{H}(L,\rho_0)-Mod$ is given by $m_L(\tau)= \mathrm{Hom}_{\mathfrak{P}_0}(\rho_0, \tau)$. The functor $m_L$ is an equivalence of categories. Let $f \in m_L(\tau), \gamma \in \mathcal{H}(L,\rho_0)$ and $w \in W$. The action of $\mathcal{H}(L,\rho_0)$ on $m_L(\tau)$ is given by $(\gamma. f)(w)= \int_L \tau(l)f(\gamma^{\vee}(l^{-1})w)dl$. Here $\gamma^{\vee}$ is defined on $L$ by $\gamma^{\vee}(l^{-1})=\gamma(l)^{\vee}$ for $l \in L$. Let $\tau' \in \mathfrak{R}^{[L,\pi]_G}(G)$ then the functor $m_G \colon \mathfrak{R}^{[L,\pi]_G}(G) \longrightarrow \mathcal{H}(G,\rho)-Mod$ is given by $m_G(\tau')=\mathrm{Hom}_{\mathfrak{P}}(\rho, \tau')$. The functor  $m_G$ is an equivalence of categories. From Corollary 8.4 in \cite{MR1643417}, the functors $m_L, m_G, Ind_P^G, (t_P)_*$  fit into the following commutative diagram:

\[
\begin{CD}
\mathfrak{R}^{[L,\pi]_G}(G)    @>m_G>>    \mathcal{H}(G,\rho)-Mod\\
@AInd_P^GAA                                    @A(t_P)_*AA\\
\mathfrak{R}^{[L,\pi]_L}(L)    @>m_L>>     \mathcal{H}(L,\rho_0)-Mod
\end{CD}
\]\par

If $\tau \in \mathfrak{R}^{[L,\pi]_L}(L)$ then from the above commutative diagram, we see that $(t_P)_*(m_L(\tau)) \cong m_G(Ind_P^G \tau)$ as $\mathcal{H}(G, \rho)$-modules. Replacing $\tau$ by $(\tau \otimes \delta_P^{1/2})$ in the above expression, $(t_P)_*(m_L(\tau \otimes \delta_P^{1/2})) \cong m_G(Ind_P^G (\tau \otimes \delta_P^{1/2}))$ as $\mathcal{H}(G, \rho)$-modules. As $Ind_P^G(\tau \otimes \delta_P^{1/2})= \iota_P^G(\tau)$, we have $(t_P)_*(m_L(\tau \otimes \delta_P^{1/2})) \cong m_G(\iota_P^G(\tau))$ as $\mathcal{H}(G, \rho)$-modules.\par

Our aim is to find an algebra embedding $T_P \colon \mathcal{H}(L,\rho_0) \longrightarrow \mathcal{H}(G, \rho)$ such that the following diagram commutes:

\[
\begin{CD}
\mathfrak{R}^{[L,\pi]_G}(G)    @>m_G>>    \mathcal{H}(G,\rho)-Mod\\
@A\iota_P^GAA                                    @A(T_P)_*AA\\
\mathfrak{R}^{[L,\pi]_L}(L)    @>m_L>>     \mathcal{H}(L,\rho_0)-Mod
\end{CD}
\]\par

Let $\tau \in \mathfrak{R}^{[L,\pi]_L}(L)$ then $m_L(\tau) \in \mathcal{H}(L,\rho_0)$- Mod. The functor $(T_P)_*$ is defined as below:

\[
(T_P)_*(m_L(\tau))= \left\lbrace \psi \colon \mathcal{H}(G,\rho) \to m_L(\tau) \; \middle|  \;
\begin{varwidth}{\linewidth}
$h\psi(h_1)= \psi(T_P(h)h_1)$ where\\
$h \in \mathcal{H}(L,\rho_0),h_1 \in \mathcal{H}(G,\rho)$
\end{varwidth}
\right \rbrace.
\]\par

From the above commutative diagram, we see that $(T_P)_*(m_L(\tau)) \cong m_G(\iota_P^G(\tau))$ as $\mathcal{H}(G, \rho)$-modules. Recall that $(t_P)_*(m_L(\tau \otimes \delta_P^{1/2})) \cong m_G(\iota_P^G(\tau))$ as $\mathcal{H}(G, \rho)$-modules. Hence we have to find an algebra embedding  $T_P \colon \mathcal{H}(L,\rho_0) \longrightarrow \mathcal{H}(G, \rho)$ such that 
$(T_P)_*(m_L(\tau)) \cong (t_P)_*(m_L(\tau \otimes \delta_P^{1/2}))$ as $\mathcal{H}(G, \rho)$-modules.\par

\begin{proposition}\label{pro_25}
	The map $T_P$ is given by $T_P(\phi)= t_P(\phi \delta_P^{-1/2})$ for $\phi \in \mathcal{H}(L, \rho_0)$ so that we have $(T_P)_*(m_L(\tau))= (t_P)_*(m_L(\tau \otimes \delta_P^{1/2}))$ as $\mathcal{H}(G,\rho)$- modules.
\end{proposition}
\begin{proof}
	Let $W$ be space of $\rho_0$. The vector spaces for $m_L(\tau \delta_P^{1/2})$ and $m_L(\tau)$ are the same. Let $f \in m_L(\tau)= \mathrm{Hom}_{\mathfrak{P}_0}(\rho_0, \tau),\gamma \in \mathcal{H}(L,\rho_0)$ and $w \in W$. Recall the action of $\mathcal{H}(L,\rho_0)$ on $m_L(\tau)$ is given by
	\[(\gamma. f)(w)= \int_L \tau(l)f(\gamma^{\vee}(l^{-1})w)dl.\] \par
	
	Let $f' \in m_L(\tau \delta_P^{1/2})= \mathrm{Hom}_{\mathfrak{P}_0}(\rho_0, \tau\delta_P^{1/2}),\gamma \in \mathcal{H}(L,\rho_0)$ and $w \in W$. Recall the action of $\mathcal{H}(L,\rho_0)$ on $m_L(\tau \delta_P^{1/2})$ is given by  \[(\gamma. f')(w)= \int_L  (\tau\delta_P^{1/2})(l)f'(\gamma^{\vee}(l^{-1})w)dl= \int_L \tau(l)\delta_P^{1/2}(l)f'(\gamma^{\vee}(l^{-1})w)dl.\] Now $f'$ is a linear transformation from space of $\rho_0$ to space of $\tau\delta_P^{1/2}$. As $\delta_P^{1/2}(l) \in \mathbb{C}^{\times}$, so $\delta_P^{1/2}(l)f'(\gamma^{\vee}(l^{-1})w)=f'(\delta_P^{1/2}(l)\gamma^{\vee}(l^{-1})w)$. Hence we have
	
	\[(\gamma. f')(w)= \int_L \tau(l)f'(\delta_P^{1/2}(l)\gamma^{\vee}(l^{-1})w)dl=  \int_L \tau(l)f'(\delta_P^{1/2}(l)\gamma(l)^{\vee}w)dl.\] Further as $\delta_P^{1/2}(l) \in \mathbb{C}^{\times}$, so  $\delta_P^{1/2}(l)(\gamma(l))^{\vee}=(\delta_P^{1/2}\gamma)(l)^{\vee}$. Therefore
	\[(\gamma.f')(w)= \int_L \tau(l)f'((\delta_P^{1/2}\gamma)(l)^{\vee}w)dl= (\delta_P^{1/2}\gamma).f'(w).\] Hence we can conclude that the action of $\gamma \in \mathcal{H}(L, \rho_0)$ on $f' \in m_L(\tau\delta_P^{1/2})$ is same as the action of $\delta_P^{1/2}\gamma \in \mathcal{H}(L, \rho_0)$  on $f' \in m_L(\tau)$. So we have  $(T_P)_*(m_L(\tau))= (t_P)_*(m_L(\tau \otimes \delta_P^{1/2}))$ as $\mathcal{H}(G,\rho)$- modules.
\end{proof}\par

From Proposition ~\ref{pro_25}, $T_P(\alpha)=t_P(\alpha \delta_P^{-1/2})$. So we have

\begin{align*}
T_P(\alpha)&=t_P(\alpha \delta_P^{-1/2})\\
&= t(\alpha \delta_P^{-1/2}\delta_P)\\
&= t(\alpha \delta_P^{1/2})\\
&= T^+(\alpha \delta_P^{1/2}).
\end{align*}\par

Hence
\begin{align*}
T_P(\alpha)(\zeta)&= T^+(\alpha \delta_P^{1/2})(\zeta)\\
&=\delta_P^{1/2}(\zeta)T^+(\alpha)(\zeta)\\
&=\delta_P^{1/2}(\zeta)\alpha(\zeta)\\
&=\delta_P^{1/2}(\zeta)1_{W^{\vee}}.
\end{align*}

Thus $T_P(\alpha)(\zeta)= \delta_P^{1/2}(\zeta)1_{W^{\vee}}$ with $\text{supp}(T_P(\alpha))=\text{supp}(t_P(\alpha))=\mathfrak{P}\zeta\mathfrak{P}$.\par

\subsection{Calculation of $(\phi_0*\phi_1)(\zeta)$}
In this section we calculate $(\phi_0*\phi_1)(\zeta)$. Let $g_i= q^{-n/2}\phi_i$ for $i=0,1$ in the unramified case and $g_i= q^{-n/4}\phi_i$ for $i=0,1$ in the ramified case. Determining $(\phi_0*\phi_1)(\zeta)$ would be useful in showing $g_0*g_1= T_P(\alpha)$ in both ramified and unramified cases. From now on, we assume without loss of generality that $\text{vol}\mathfrak{P}_0= \text{vol}\mathfrak{P}_{-}= \text{vol}\mathfrak{P}_{+}=1$. Thus we have $\text{vol}\mathfrak{P}=1$.
\begin{lemma}\label{lem_10}
	$\text{supp}(\phi_0* \phi_1)=\mathfrak{P} \zeta \mathfrak{P}=\mathfrak{P}w_0 w_1\mathfrak{P}$. 
\end{lemma}
\begin{proof} 
	We first claim that $\text{supp}(\phi_0* \phi_1)\subseteq \mathfrak{P}w_0\mathfrak{P}w_1\mathfrak{P}$. Suppose $z \in \text{supp}(\phi_0* \phi_1)$ then $(\phi_0*\phi_1)(z)= \int_G \phi_0(zr^{-1})\phi_1(r)dr \neq 0$. This would imply that there exists an $r \in G$ such that $\phi_0(zr^{-1})\phi_1(r) \neq 0$. As $\phi_0(zr^{-1})\phi_1(r) \neq 0$, this means that $\phi_0(zr^{-1}) \neq 0,\phi_1(r) \neq 0 $. But $\phi_0(zr^{-1}) \neq 0$ would imply that $zr^{-1} \in \mathfrak{P}w_0\mathfrak{P}$ and $\phi_1(r) \neq 0$ would imply that $r \in \mathfrak{P}w_1\mathfrak{P}$. So $z= (zr^{-1})(r) \in (\mathfrak{P}w_0\mathfrak{P})(\mathfrak{P}w_1\mathfrak{P})=(\text{supp}\phi_0)(\text{supp}\phi_1)= \mathfrak{P}w_0\mathfrak{P}w_1\mathfrak{P}$. Hence $\text{supp}(\phi_0* \phi_1)\subseteq \mathfrak{P}w_0\mathfrak{P}w_1\mathfrak{P}$. Let us recall $\mathfrak{P}_0, \mathfrak{P}_{+}, \mathfrak{P}_{-}$.
	
	\begin{center}
		$ \mathfrak{P}_{0}=
		\Bigg\lbrace
		\begin{bmatrix}
		a & 0\\ 
		0 & {^t}{\overline{a}}{^{-1}}\\
		\end{bmatrix}\mid a \in \mathrm{GL}_n(\mathfrak{O}_E)
		\Bigg \rbrace ,$
		\end {center}

		\begin{center}
			$\mathfrak{P}_{+}=
			\Bigg\lbrace
			\begin{bmatrix}
			1_n & X\\ 
			0 & 1_n\\
			\end{bmatrix} \mid X\in \mathrm{M}_n(\mathfrak{O}_E), X + {^t}{\overline{X}}=0
			\Bigg \rbrace ,$	
			\end {center}

			\begin{center}
				$\mathfrak{P}_{-}=
				\Bigg\lbrace
				\begin{bmatrix}
				1_n & 0\\ 
				X & 1_n\\
				\end{bmatrix} \mid X \in \mathrm{M}_n(\mathbf{p}_E), X + {^t}{\overline{X}}=0
				\Bigg \rbrace .$			
				\end {center}\par			
				
				It is easy observe that $w_0\mathfrak{P}_{-}w_0^{-1} \subseteq \mathfrak{P}_{+}, w_0\mathfrak{P}_{0}w_0^{-1}=\mathfrak{P}_{0}, w_1^{-1}\mathfrak{P}_{+}w_1 \subseteq \mathfrak{P}_{-}$. Now we have
				
				\begin{align*}
				\mathfrak{P}w_0\mathfrak{P}w_1\mathfrak{P}&= \mathfrak{P}w_0\mathfrak{P}_{-}\mathfrak{P}_{0}\mathfrak{P}_{+}w_1\mathfrak{P}\\
				&=\mathfrak{P}w_0\mathfrak{P}_{-}w_0^{-1}w_0\mathfrak{P}_{0}w_0^{-1}w_0w_1w_1^{-1}\mathfrak{P}_{+}w_1\mathfrak{P}\\
				&\subseteq\mathfrak{P}\mathfrak{P}_{+}\mathfrak{P}_{0}w_0w_1\mathfrak{P}_{-}\mathfrak{P}\\
				&=\mathfrak{P}w_0w_1\mathfrak{P}\\
				&=\mathfrak{P}\zeta\mathfrak{P}.
				\end{align*}

				So $\mathfrak{P}w_0\mathfrak{P}w_1\mathfrak{P}\subseteq \mathfrak{P}w_0w_1\mathfrak{P}=\mathfrak{P}\zeta\mathfrak{P}$. On the contrary, as $1 \in \mathfrak{P}$, so $\mathfrak{P}\zeta\mathfrak{P}=\mathfrak{P}w_0w_1\mathfrak{P} \subseteq \mathfrak{P}w_0\mathfrak{P}w_1\mathfrak{P}$. Hence we have $\mathfrak{P}w_0\mathfrak{P}w_1\mathfrak{P}= \mathfrak{P}w_0w_1\mathfrak{P}=\mathfrak{P}\zeta\mathfrak{P}$. Therefore $\text{supp}(\phi_0*\phi_1) \subseteq \mathfrak{P}w_0\mathfrak{P}w_1\mathfrak{P}= \mathfrak{P}w_0w_1\mathfrak{P}=\mathfrak{P}\zeta\mathfrak{P}$. This implies $\text{supp}(\phi_0*\phi_1)= \varnothing$ or $\mathfrak{P}\zeta\mathfrak{P}$. But if $\text{supp}(\phi_0*\phi_1)= \varnothing$ then $(\phi_0*\phi_1)=0$ which is a contradiction. Thus $\text{supp}(\phi_0*\phi_1)=\mathfrak{P}\zeta\mathfrak{P}$. 
			\end{proof}
			For $r \in \mathbb{Z}$ let
			
			\begin{center}
				$K_{-,r}=
				\Bigg\lbrace
				\begin{bmatrix}
				1_n & 0\\ 
				X & 1_n\\
				\end{bmatrix} \mid X \in \mathrm{M}_n(\mathbf{p}_E^r), X + {^t}{\overline{X}}=0
				\Bigg \rbrace$,			
			\end{center}

			\begin{center}
				$K_{+,r}=
				\Bigg\lbrace
				\begin{bmatrix}
				1_n & X\\ 
				0 & 1_n\\
				\end{bmatrix} \mid X \in \mathrm{M}_n(\mathbf{p}_E^r), X + {^t}{\overline{X}}=0
				\Bigg \rbrace$.
			\end{center}\par
			\begin{proposition}\label{prop_35}
				$(\phi_0*\phi_1)(\zeta)=\phi_0(w_0)\phi_1(w_1).$	
			\end{proposition}
			\begin{proof}			
				From Lemma \ref{lem_10}, $\text{supp}(\phi_0* \phi_1)=\mathfrak{P} \zeta \mathfrak{P}=\mathfrak{P}w_0 w_1\mathfrak{P}$. So now let us consider
				\begin{align*}
				(\phi_0*\phi_1)(\zeta)&=(\phi_0*\phi_1)(w_0w_1)\\
				&=\int_G \phi_0(y)\phi_1(y^{-1}\zeta)dy\\
				&=\int_{\mathfrak{P}w_0\mathfrak{P}}\phi_0(y)\phi_1(y^{-1}\zeta)dy.
				\end{align*}
				
				We know that $\mathfrak{P}w_0\mathfrak{P}= \underset{z \in {\mathfrak{P}w_0\mathfrak{P}/\mathfrak{P}}}{\amalg}z\mathfrak{P}$. Let $y=zp \in z\mathfrak{P}$. So we have
				\begin{align*}
				\phi_0(y)\phi_1(y^{-1}\zeta)&=\phi_0(zp)\phi_1(p^{-1}z^{-1}\zeta)\\
				&= \phi_0(z)\rho^{\vee}(p)\rho^{\vee}(p^{-1})\phi_1(z^{-1}\zeta)\\
				&=\phi_0(z)\phi_1(z^{-1}\zeta).
				\end{align*}				
				Hence \[(\phi_0*\phi_1)(\zeta)= \underset{z \in {\mathfrak{P}w_0\mathfrak{P}/\mathfrak{P}}}{\sum}\phi_0(z)\phi_1(z^{-1}\zeta) \text{Vol}{\mathfrak{P}}=\underset{z \in {\mathfrak{P}w_0\mathfrak{P}/\mathfrak{P}}}{\sum}\phi_0(z)\phi_1(z^{-1}\zeta)\] 
				
				Let $\alpha \colon \mathfrak{P}/w_0\mathfrak{P}w_0^{-1} \cap \mathfrak{P} \longrightarrow \mathfrak{P}w_0\mathfrak{P}/\mathfrak{P}$ be the map given by $\alpha(x(w_0\mathfrak{P}w_0^{-1} \cap \mathfrak{P}))= xw_0\mathfrak{P}$ where $x \in \mathfrak{P}$. We can observe that the map $\alpha$ is bijective. So $\mathfrak{P}/w_0\mathfrak{P}w_0^{-1} \cap \mathfrak{P}$ is in bijection with $\mathfrak{P}w_0\mathfrak{P}/\mathfrak{P}$.\par
				
				Hence \[(\phi_0*\phi_1)(\zeta)=  \underset{x \in {\mathfrak{P}/w_0\mathfrak{P}w_0^{-1} \cap \mathfrak{P}}}{\sum}\phi_0(xw_0)\phi_1(w_0^{-1}x^{-1}\zeta).\]
				From Iwahori factorization of $\mathfrak{P}$ we have $\mathfrak{P}=\mathfrak{P}_{-}\mathfrak{P}_{0}\mathfrak{P}_{+}= K_{-,1}\mathfrak{P}_{0}K_{+,0}$. Therefore $w_0\mathfrak{P}w_0^{-1}= ^{w_0}\mathfrak{P}=^{w_0}K_{-,1} ^{w_0}\mathfrak{P}_{0} ^{w_0}K_{+,0}= K_{+,1}\mathfrak{P}_0K_{-,0}$. So $\mathfrak{P}_0 \cap w_0\mathfrak{P}w_0^{-1}= \mathfrak{P} \cap ^{w_0}\mathfrak{P}= K_{+,1}\mathfrak{P}_0K_{-,1}$. Let $\beta \colon \mathfrak{P}/w_0\mathfrak{P}w_0^{-1} \cap \mathfrak{P} \longrightarrow K_{+,0}/K_{+,1}$ be the map given by $\beta(x(\mathfrak{P}\cap ^{w_0}\mathfrak{P}))= x_{+}K_{+,1}$ where $x \in \mathfrak{P}$ and $x=x_{+}px_{-}, x_{+} \in \mathfrak{P}_{+}, p \in \mathfrak{P}_{0}, x_{-} \in \mathfrak{P}_{-}$. We can observe that the map $\beta$ is bijective. So $\mathfrak{P}/w_0\mathfrak{P}w_0^{-1} \cap \mathfrak{P}$ is in bijection with $K_{+,0}/K_{+,1}$.\par
				
				Therefore 
				\begin{align*}
				(\phi_0*\phi_1)(\zeta)&=\underset{x_{+} \in {K_{+,0}/K_{+,1}}}{\sum}\phi_0(x_{+}w_0)\phi_1(w_0^{-1}x_{+}^{-1}\zeta)\\ &=\underset{x_{+}\in{K_{+,0}/K_{+,1}}}{\sum}\rho^{\vee}(x_{+})\phi_0(w_0)\phi_1(w_0^{-1}x_{+}^{-1}\zeta).
				\end{align*}
				As $\rho^{\vee}$ is trivial on $\mathfrak{P}_{+}$ and $x_{+} \in \mathfrak{P}_{+}$ so we have
				\begin{center}
					$(\phi_0*\phi_1)(\zeta)= \underset{x_{+} \in {K_{+,0}/K_{+,1}}}{\sum}\phi_0(w_0)\phi_1(w_0^{-1}x_{+}^{-1}\zeta).$
				\end{center}				
				
				The terms in above summation which do not vanish are the ones for which $w_0^{-1}x_{+}^{-1}\zeta \in \mathfrak{P}w_1\mathfrak{P} \Longrightarrow x_{+}^{-1} \in w_0\mathfrak{P}w_1\mathfrak{P}\zeta^{-1}\Longrightarrow x_{+} \in \zeta\mathfrak{P}w_1^{-1}\mathfrak{P}w_0^{-1} \Longrightarrow w_0^{-1}x_{+}w_0 \in w_1\mathfrak{P}w_1^{-1}\mathfrak{P}$. It is clear $w_1\mathfrak{P}w_1^{-1}\mathfrak{P}= (^{w_1}\mathfrak{P})(\mathfrak{P})$. As $^{w_1}\mathfrak{P}= {}^{w_1}K_{-,1}^{w_1}\mathfrak{P}_0^{w_1}K_{+,0}= K_{-,2}\mathfrak{P}_0K_{+,-1}$, so $w_1\mathfrak{P}w_1^{-1}\mathfrak{P}= (^{w_1}\mathfrak{P})(\mathfrak{P})=   K_{-,2}\mathfrak{P}_0K_{+,-1}\mathfrak{P}_0K_{-,1}$. Hence we have $w_0^{-1}x_{+}w_0\in K_{-,2}\mathfrak{P}_0K_{+,-1}\mathfrak{P}_0K_{-,1} \Longrightarrow w_0^{-1}x_{+}w_0=k_{-}p_{0}k_{+}k'_{-}$ where $k_{-} \in K_{-,2}, k_{+} \in K_{+,-1}, k_{-}^{'}\in K_{-,1}, p_0 \in \mathfrak{P}_0$. Hence we have $p_{0}k_{+}= k_{-}^{-1}w_0^{-1}x_{+}w_0k_{-}^{'-1}$. Now as $w_0^{-1}x_{+}w_0 \in K_{-,0}, k_{-}^{-1} \in K_{-,2}, k_{-}^{'-1}\in K_{-,1}$, so $k_{-}^{-1}w_0^{-1}x_{+}w_0k_{-}^{'-1} \in K_{-,0}$ and $p_{0}k_{+} \in \mathfrak{P}_{0}K_{+,-1}$. But we know that $K_{-,0} \cap \mathfrak{P}_{0}K_{+,-1}= 1 \Longrightarrow p_{0}k_{+}= 1 \Longrightarrow w_0^{-1}x_{+}w_0=k_{-}k_{-}' \in K_{-,1} \Longrightarrow x_{+} \in w_0K_{-,1}w_0^{-1}= K_{+,1}$. As $x_{+} \in K_{+,1}$, so only the trivial coset contributes to the above summation. Hence \[(\phi_0*\phi_1)(\zeta)=\phi_0(w_0)\phi_1(w_0^{-1}\zeta)=\phi_0(w_0)\phi_1(w_1).\]
			\end{proof}
			
			\subsection{Relation between $g_0, g_1$ and $T_P(\alpha)$}
			
			\subsubsection{Unramified case:}
			
			Recall that $\mathcal{H}(G,\rho)=\langle \phi_0, \phi_1 \rangle$ where $\phi_0$ is supported on $\mathfrak{P}w_0\mathfrak{P}$ and $\phi_1$ is supported on $\mathfrak{P}w_1\mathfrak{P}$ respectively with $\phi_i^2= q^n+(q^n-1)\phi_i$ for $i=0,1$. In this section we show that $g_0*g_1= T_P(\alpha)$, where $g_i= q^{-n/2}\phi_i$ for $i=0,1$. 
			
			\begin{proposition}\label{pro_36}
				$g_0g_1= T_P(\alpha)$.
			\end{proposition}
			\begin{proof}
				Let us choose $\psi_i \in \mathcal{H}(G,\rho)$ for $i=0,1$ such that $\text{supp}(\psi_i)= \mathfrak{P}w_i\mathfrak{P}$ for $i=0,1$. So $\phi_i$ is a scalar multiple of $\psi_i$ for $i=0,1$. Hence $\phi_i = \lambda_i \psi_i$ where $\lambda_i \in \mathbb{C}^{\times}$ for $i=0,1$. Let $\psi_i(w_i)=A \in \mathrm{Hom}_{\mathfrak{P}\cap ^{w_i}\mathfrak{P}}(^{w_i}\rho^{\vee},\rho^{\vee})$ for $i=0,1$ and $W$ be the space of $\rho$. So $A^2= 1_{W^{\vee}}$. From Proposition ~\ref{prop_35}, we have $(\psi_0*\psi_1)(\zeta)= \psi_0(w_0)\psi_1(w_1)= A^2=1_{W^{\vee}}$. Now let $\psi_i$ satisfies the quadratic relation given by $\psi_i^2 = a\psi_i+b$ where $a,b \in \mathbb{R}$ for $i=0,1$. As 
				$\psi_i^2 = a\psi_i+b \Longrightarrow (-\psi_i)^2 = (-a)(-\psi_i)+b$, so $a$ can be arranged such that $a>0$. We can see that $1 \in {\mathcal{H}(G,\rho)}$ is defined as below:
				
				\[
				1(x)=
				\begin{cases}
				0,   &\text{if $x \notin \mathfrak{P}$;}\\
				\rho^{\vee}(x) &\text{if $x \in \mathfrak{P}$.}
				\end{cases}
				\]
				
				Let us consider $\psi_i^2(1) = \int_G \psi_i(y)\psi_i(y^{-1})dy$ for $i=0,1$. Now let $y=pw_ip'$ where $p,p' \in \mathfrak{P}$ for $i=0,1$. So we have
				
				\begin{align*}
				\psi_i^2(1) &= \int_{\mathfrak{P}w_i\mathfrak{P}}\psi_i(pw_ip')\psi_i(p^{'-1}w_i^{-1}p^{-1})d(pw_ip')\\
				&=\int_{\mathfrak{P}w_i\mathfrak{P}}\rho^{\vee}(p)\psi_i(w_i)\rho^{\vee}(p')\rho^{\vee}(p^{'-1})\psi_i(w_i^{-1})\rho^{\vee}(p^{-1})d(pw_ip')\\
				&=\int_{\mathfrak{P}w_i\mathfrak{P}}\rho^{\vee}(p)\psi_i(w_i)\psi_i(w_i^{-1})\rho^{\vee}(p^{-1})d(pw_ip')\\
				&=\int_{\mathfrak{P}w_i\mathfrak{P}}\rho^{\vee}(p)\psi_i(w_i)\psi_i(w_i)\rho^{\vee}(p^{-1})d(pw_ip')\\
				&=\int_{\mathfrak{P}w_i\mathfrak{P}}\rho^{\vee}(p)A^2\rho^{\vee}(p^{-1})d(pw_ip')\\
				&=\int_{\mathfrak{P}w_i\mathfrak{P}}A^2\rho^{\vee}(p)\rho^{\vee}(p^{-1})d(pw_ip')\\
				&=A^2\text{vol}(\mathfrak{P}w_i\mathfrak{P})\\
				&=1_{W^{\vee}}\text{vol}(\mathfrak{P}w_i\mathfrak{P}).
				\end{align*}\par
				
				So $\psi_i^2(1)=1_{W^{\vee}}\text{vol}(\mathfrak{P}w_i\mathfrak{P})$ for $i=0,1$. We already know that $\psi_i^2 = a\psi_i+b$ where $a,b \in \mathbb{R}$ and for $i=0,1$. Now evaluating the expression $\psi_i^2 = a\psi_i+b$ at 1, we have 
				$\psi_i^2(1) = a\psi_i(1)+b 1(1)$. We can see that $\psi_i(1)=0$ as support of $\psi_i$ is $\mathfrak{P}w_i\mathfrak{P}$ for $i=0,1$. We have seen before that $\psi_i^2(1) = 1_{W^{\vee}}\text{vol}(\mathfrak{P}w_i\mathfrak{P})$ for $i=0,1$ and as $1 \in \mathfrak{P}, 1(1)= \rho^{\vee}(1)=1_{W^{\vee}}$. So $\psi_i^2(1) = a\psi_i(1)+b 1(1) \Longrightarrow 1_{W^{\vee}}\text{vol}(\mathfrak{P}w_i\mathfrak{P})=1_{W^{\vee}}b$ for $i=0,1$. Comparing coefficients of $1_{W^{\vee}}$ on both sides of the equation $1_{W^{\vee}}\text{vol}(\mathfrak{P}w_i\mathfrak{P})=1_{W^{\vee}}b$ for $i=0,1$ we get
				
				\begin{center}
					$b=\text{vol}(\mathfrak{P}w_i\mathfrak{P})$. 
				\end{center}\par

				As $\phi_i= \lambda_i\psi_i$ for $i=0,1$, hence $\phi_i^2= \lambda_i^2\psi_i^2=\lambda_i^2(a\psi_i+b)= (\lambda_i a)(\lambda_i \psi_i)+ \lambda_i^2 b= (\lambda_i a)\phi_i+  \lambda_i^2 b$ for $i=0,1$. But $\phi_i^2= (q^n-1)\phi_i +q^n$ for $i=0,1$. So $\phi_i^2=(\lambda_i a)\phi_i+  \lambda_i^2 b = (q^n-1)\phi_i +q^n$ for $i=0,1$. As $\phi_i$ and $1$ are linearly independent, hence $\lambda_i a= (q^n-1)$ for $i=0,1$. Therefore $\lambda_i= \frac{q^n-1}{a}$ for $i=0,1$. As $a>0, a\in \mathbb{R}$, so $\lambda_i>0, \lambda_i \in \mathbb{R}$ for $i=0,1$. Similarly, as $\phi_i$ and $1$ are linearly independent, hence $\lambda_i^2 b= q^n \Longrightarrow \lambda_i^2= \frac{q^n}{b}$ for $i=0,1$.\par 
				
				Now $\mathfrak{P}w_i\mathfrak{P}= \underset{x \in \mathfrak{P}/\mathfrak{P} \cap ^{w_i}\mathfrak{P}}{\amalg} xw_i\mathfrak{P} \Longrightarrow \text{vol}(\mathfrak{P}w_i\mathfrak{P})= [\mathfrak{P}w_i\mathfrak{P}:\mathfrak{P}]\text{vol}\mathfrak{P}=[\mathfrak{P}w_i\mathfrak{P}:\mathfrak{P}]=[\mathfrak{P}:\mathfrak{P} \cap ^{w_i}\mathfrak{P}]$ for $i=0,1$. Hence $b=\text{vol}(\mathfrak{P}w_i\mathfrak{P})=[\mathfrak{P}:\mathfrak{P} \cap ^{w_i}\mathfrak{P}]$ for $i=0,1$. Now as $\lambda_0^2= \lambda_1^2= \frac{q^n}{b} \Longrightarrow \lambda_0= \lambda_1= \frac{q^{n/2}}{b^{1/2}}=\frac{q^{n/2}}{[\mathfrak{P}:\mathfrak{P} \cap ^{w_0}\mathfrak{P}]^{1/2}}$. Therefore

				\begin{align*}
				\phi_0\phi_1 &=(\lambda_0 \psi_0)(\lambda_1\psi_1)\\
				&=\lambda_0^2\psi_0\psi_1\\
				&= \frac{q^n \psi_0 \psi_1}{[\mathfrak{P}:\mathfrak{P} \cap ^{w_0}\mathfrak{P}]}.
				\end{align*}\par
				
				We have seen before that, $\mathfrak{P}= K_{-,1}\mathfrak{P}_{0}K_{+,0}$ and $\mathfrak{P} \cap ^{w_0}\mathfrak{P}=K_{-,1}\mathfrak{P}_{0}K_{+,1}$. So
				
				\begin{align*}
				[\mathfrak{P}:\mathfrak{P} \cap ^{w_0}\mathfrak{P}]&=|\frac{K_{+,0}}{K_{+,1}}|\\
				&=|\lbrace X \in \mathrm{M}_n(k_E) \mid X+ ^t{\overline{X}}=0 \rbrace|\\
				&=(q^n)(q^2)^{\frac{(n)(n-1)}{2}}\\
				&=(q^n)(q^{n^2-n})\\
				&=q^{n^2}.
				\end{align*}\par
				
				Hence
				\begin{align*}
				(\phi_0\phi_1)(\zeta )&=\frac{q^n (\psi_0 \psi_1)(\zeta)}{[\mathfrak{P}:\mathfrak{P} \cap ^{w_0}\mathfrak{P}]}\\ 
				&=\frac{q^n (\psi_0 \psi_1)(\zeta)}{ q^{n^2}}\\
				&=q^{n-n^2}1_{W^{\vee}}.
				\end{align*}\par
				
				Recall $g_i = q^{-n/2}\phi_i$ for $i=0,1$. We know that $\phi_i^2= (q^n-1)\phi_i +q^n$ for $i=0,1$. So for $i=0,1$ we have
				
				\begin{align*}
				g_i^2 &= q^{-n}\phi_i^2\\
				&=q^{-n}((q^n-1)\phi_i +q^n)\\
				&=(1-q^{-n})\phi_i+1\\
				&=(1-q^{-n})q^{n/2}g_i+1\\
				&=(q^{n/2}- q^{-n/2})g_i+1.
				\end{align*}\par
				
				So $g_0g_1=(q^{-n/2}\phi_0)(q^{-n/2}\phi_21)= q^{-n}\phi_0\phi_1 \Longrightarrow (g_0g_1)(\zeta)= q^{-n}(\phi_0\phi_1)(\zeta)= q^{-n}q^{n-n^2}1_{W^{\vee}}= q^{-n^2}1_{W^{\vee}}$. From the earlier discussion in this section we have $T_P(\alpha)(\zeta)=\delta_P^{1/2}(\zeta)1_{W^{\vee}}$. From section 9.1, we have $\delta_P(\zeta)= q^{-2n^2}$. Hence $\delta_P^{1/2}(\zeta)= q^{-n^2}$. Therefore $(g_0g_1)(\zeta)= T_P(\alpha)(\zeta)$. So $(g_0g_1)(\zeta)= T_P(\alpha)(\zeta)$. We have $\text{supp}(T_P(\alpha))= \mathfrak{P}\zeta\mathfrak{P}$. As $\text{supp}(g_i)=\mathfrak{P}w_i\mathfrak{P}$, Lemma \ref{lem_10} gives $\text{supp}(g_0g_1)= \mathfrak{P}\zeta\mathfrak{P}$. Therefore $g_0g_1= T_P(\alpha)$.
			\end{proof}
			
			\subsubsection{Ramified case:}
			We know that $\mathcal{H}(G,\rho)=\langle \phi_0, \phi_1 \rangle$ where $\phi_0$ is supported on $\mathfrak{P}w_0\mathfrak{P}$ and $\phi_1$ is supported on $\mathfrak{P}w_1\mathfrak{P}$ respectively with $\phi_i^2= q^{n/2}+(q^{n/2}-1)\phi_i$ for $i=0,1$. In this section we show that $g_0*g_1= T_P(\alpha)$, where $g_i= q^{-n/4}\phi_i$ for $i=0,1$.
			
			\begin{proposition}\label{pro_38}
				$g_0g_1= T_P(\alpha)$.
			\end{proposition}
			\begin{proof}	
				Let us choose $\psi_i \in \mathcal{H}(G,\rho)$ for $i=0,1$ such that $\text{supp}(\psi_i)= \mathfrak{P}w_i\mathfrak{P}$ for $i=0,1$. So $\phi_i$ is a scalar multiple of $\psi_i$ for $i=0,1$. Hence $\phi_i = \lambda_i \psi_i$ where $\lambda_i \in \mathbb{C}^{\times}$ for $i=0,1$. Let $\psi_i(w_i)=A_i \in \mathrm{Hom}_{\mathfrak{P}\cap ^{w_i}\mathfrak{P}}(^{w_i}\rho^{\vee},\rho^{\vee})$ for $i=0,1$ and $W$ be the space of $\rho$. So $A_i^2= 1_{W^{\vee}}$ for $i=0,1$. From section 5.1 on page 24 in \cite{MR2276353}, we can say that $A_0= A_1$. From Proposition ~\ref{prop_35}, we have $(\psi_0*\psi_1)(\zeta)= \psi_0(w_0)\psi_1(w_1)= A_0A_1=A_0^2=1_{W^{\vee}}$. Now let $\psi_i$ satisfies the quadratic relation given by $\psi_i^2 = a_i\psi_i+b_i$ where $a_i,b_i \in \mathbb{R}$ for $i=0,1$. As $\psi_i^2 = a_i\psi_i+b_i \Longrightarrow (-\psi_i)^2 = (-a_i)(-\psi_i)+b_i$, so $a_i$ can be arranged such that $a_i>0$ for $i=0,1$. We can see that $1 \in {\mathcal{H}(G,\rho)}$ is defined as below:
				
				\[
				1(x)=
				\begin{cases}
				0,   &\text{if $x \notin \mathfrak{P}$;}\\
				\rho^{\vee}(x) &\text{if $x \in \mathfrak{P}$.}
				\end{cases}
				\]
				
			   As shown in Proposition~\ref{pro_36}, we can show that   $\psi_0^2(1)=1_{W^{\vee}}\text{vol}(\mathfrak{P}w_0\mathfrak{P})$. We already know that 
				$\psi_0^2 = a_0\psi_0+b_0$ where $a_0,b_0 \in \mathbb{R}$. Now evaluating the expression $\psi_0^2 = a_0 \psi_0+b_0$ at 1, we have 
				$\psi_0^2(1) = a_0\psi_0(1)+b_0 1(1)$. We can see that $\psi_0(1)=0$ as support of $\psi_0$ is $\mathfrak{P}w_0\mathfrak{P}$. We have seen before that $\psi_0^2(1) = 1_{W^{\vee}}\text{vol}(\mathfrak{P}w_0\mathfrak{P})$ and as $1 \in \mathfrak{P}, 1(1)= \rho^{\vee}(1)= 1_{W^{\vee}}$. So $\psi_0^2(1) = a_0\psi_i(1)+b_0 1(1) \Longrightarrow 1_{W^{\vee}}\text{vol}(\mathfrak{P}w_0\mathfrak{P})=1_{W^{\vee}}b_0$. Comparing  coefficients of $1_{W^{\vee}}$ on both sides of the equation $1_{W^{\vee}}b_0=1_{W^{\vee}}\text{vol}(\mathfrak{P}w_0\mathfrak{P})$ we get
				
				\begin{center}
					$b_0=\text{vol}(\mathfrak{P}w_0\mathfrak{P})$.
				\end{center}\par
				
				As $\phi_0= \lambda_0\psi_0$, hence $\phi_0^2= \lambda_0^2\psi_0^2=\lambda_0^2(a_0\psi_0+b_0)= (\lambda_0 a_0)(\lambda_0 \psi_0)+ \lambda_0^2 b_0= (\lambda_0 a_0)\phi_0+  \lambda_0^2 b_0$. But $\phi_0^2= (q^{n/2}-1)\phi_0 +q^{n/2}$. So $\phi_0^2=(\lambda_0 a_0)\phi_0+  \lambda_0^2 b_0= (q^{n/2}-1)\phi_0 +q^{n/2}$. As $\phi_0$ and $1$ are linearly independent, hence $\lambda_0 a_0= (q^{n/2}-1)$. Therefore $\lambda_0= \frac{q^{n/2}-1}{a_0}$. As $a_0>0, a_0\in \mathbb{R}$, so $\lambda_0>0, \lambda_0 \in \mathbb{R}$. Similarly, as $\phi_0$ and $1$ are linearly independent, hence $\lambda_0^2 b= q^{n/2} \Longrightarrow \lambda_0^2= \frac{q^{n/2}}{b_0}$. \par 
				
				Now $\mathfrak{P}w_0\mathfrak{P}= \underset{x \in \mathfrak{P}/\mathfrak{P} \cap ^{w_0}\mathfrak{P}}{\amalg} xw_0\mathfrak{P} \Longrightarrow \text{vol}(\mathfrak{P}w_0\mathfrak{P})= [\mathfrak{P}w_0\mathfrak{P}:\mathfrak{P}]\text{vol}\mathfrak{P}=[\mathfrak{P}w_0\mathfrak{P}:\mathfrak{P}]=[\mathfrak{P}:\mathfrak{P} \cap ^{w_0}\mathfrak{P}]$. Hence $b_0=\text{vol}(\mathfrak{P}w_0\mathfrak{P})=[\mathfrak{P}:\mathfrak{P} \cap ^{w_0}\mathfrak{P}]$. Now as $\lambda_0^2= \frac{q^{n/2}}{b_0} \Longrightarrow \lambda_0= \frac{q^{n/4}}{b_0^{1/2}}=\frac{q^{n/4}}{[\mathfrak{P}:\mathfrak{P} \cap ^{w_0}\mathfrak{P}]^{1/2}}$.\par
				
				We have seen before that, $\mathfrak{P}= K_{-,1}\mathfrak{P}_{0}K_{+,0}$ and $\mathfrak{P} \cap ^{w_0}\mathfrak{P}=K_{-,1}\mathfrak{P}_{0}K_{+,1}$. So
				
				\begin{align*}
				[\mathfrak{P}:\mathfrak{P} \cap ^{w_0}\mathfrak{P}]&=|\frac{K_{+,0}}{K_{+,1}}|\\
				&=|\lbrace X \in \mathrm{M}_n(k_E) \mid X+ ^t{\overline{X}}=0 \rbrace|\\
				&=q^{\frac{(n)(n-1)}{2}}\\
				&=q^{\frac{n^2-n}{2}}.
				\end{align*}\par

				So
				\begin{center}
					$\lambda_0=\frac{q^{n/4}}{[\mathfrak{P}:\mathfrak{P} \cap ^{w_0}\mathfrak{P}]^{1/2}}=\frac{q^{n/4}}{q^{\frac{n^2-n}{4}}}.$
				\end{center}\par
				
				Let us consider $\psi_1^2(1) = \int_G \psi_1(y)\psi_1(y^{-1})dy$. Now let $y=pw_1p'$ where $p,p' \in \mathfrak{P}$. So we have
				
				\begin{align*}
				\psi_1^2(1) &= \int_{\mathfrak{P}w_1\mathfrak{P}}\psi_1(pw_1p')\psi_1(p^{'-1}w_1^{-1}p^{-1})d(pw_1p')\\
				&=\int_{\mathfrak{P}w_1\mathfrak{P}}\rho^{\vee}(p)\psi_1(w_1)\rho^{\vee}(p')\rho^{\vee}(p^{'-1})\psi_1(w_1^{-1})\rho^{\vee}(p^{-1})d(pw_1p')\\
				&=\int_{\mathfrak{P}w_1\mathfrak{P}}\rho^{\vee}(p)\psi_1(w_1)\psi_1(w_1^{-1})\rho^{\vee}(p^{-1})d(pw_1p')\\
				&=\int_{\mathfrak{P}w_1\mathfrak{P}}\rho^{\vee}(p)\psi_1(w_1)\psi_1(-w_1)\rho^{\vee}(p^{-1})d(pw_1p')\\
				&=\int_{\mathfrak{P}w_1\mathfrak{P}}\rho^{\vee}(p)\psi_1(w_1)\rho^{\vee}(-1)\psi_1(w_1)\rho^{\vee}(p^{-1})d(pw_1p')\\
				&=\rho^{\vee}(-1)\int_{\mathfrak{P}w_1\mathfrak{P}}A_1^2\rho^{\vee}(p)\rho^{\vee}(p^{-1})d(pw_1p')\\
				&=\rho^{\vee}(-1)A_1^2\text{vol}(\mathfrak{P}w_1\mathfrak{P})\\
				&=\rho^{\vee}(-1)1_{W^{\vee}}\text{vol}(\mathfrak{P}w_1\mathfrak{P}).
				\end{align*}\par
				
				So $\psi_1^2(1)=1_{W^{\vee}}\text{vol}(\mathfrak{P}w_1\mathfrak{P})$. We already know that 
				$\psi_1^2 = a_1\psi_1+b_1$ where $a_1,b_1 \in \mathbb{R}$. Now evaluating the expression $\psi_1^2 = a_1 \psi_1+b_1$ at 1, we have 
				$\psi_1^2(1) = a_1\psi_1(1)+b_1 1(1)$. We can see that $\psi_1(1)=0$ as support of $\psi_1$ is $\mathfrak{P}w_1\mathfrak{P}$. We have seen before that $\psi_1^2(1) = 1_{W^{\vee}}\text{vol}(\mathfrak{P}w_1\mathfrak{P})$ and as $1 \in \mathfrak{P}, 1(1)= \rho^{\vee}(1)= 1_{W^{\vee}}$. So $\psi_1^2(1) = a_1\psi_i(1)+b_1 1(1) \Longrightarrow \rho^{\vee}(-1)1_{W^{\vee}}\text{vol}(\mathfrak{P}w_1\mathfrak{P})=1_{W^{\vee}}b_1$. Comparing coefficients of $1_{W^{\vee}}$ on both sides of the equation $1_{W^{\vee}}b_1=1_{W^{\vee}}\rho^{\vee}(-1)\text{vol}(\mathfrak{P}w_1\mathfrak{P})$ we get
				
				\begin{center}
					$b_1=\rho^{\vee}(-1)\text{vol}(\mathfrak{P}w_1\mathfrak{P})$.\par 
				\end{center}\par
				
				As $\phi_1= \lambda_1\psi_1$, hence $\phi_1^2= \lambda_1^2\psi_1^2=\lambda_1^2(a_1\psi_1+b_1)= (\lambda_1 a_1)(\lambda_1 \psi_1)+ \lambda_1^2 b_1 = (\lambda_0 a_1)\phi_1+  \lambda_1^2 b_1$. But $\phi_1^2= (q^{n/2}-1)\phi_1 +q^{n/2}$. So $\phi_1^2=(\lambda_1 a_1)\phi_1+  \lambda_1^2 b_1= (q^{n/2}-1)\phi_1 +q^{n/2}$. As $\phi_1$ and $1$ are linearly independent, hence $\lambda_1 a_1= (q^{n/2}-1)$. Therefore $\lambda_1= \frac{q^{n/2}-1}{a_1}$. As $a_1>0, a_1\in \mathbb{R}$, so $\lambda_1>0, \lambda_1 \in \mathbb{R}$. Similarly, as $\phi_1$ and $1$ are linearly independent, hence $\lambda_1^2 b= q^{n/2} \Longrightarrow \lambda_1^2= \frac{q^{n/2}}{b_1}$. \par 
				
				Now $\mathfrak{P}w_1\mathfrak{P}= \underset{x \in \mathfrak{P}/\mathfrak{P} \cap ^{w_1}\mathfrak{P}}{\amalg} xw_1\mathfrak{P} \Longrightarrow \text{vol}(\mathfrak{P}w_1\mathfrak{P})= [\mathfrak{P}w_1\mathfrak{P}:\mathfrak{P}]\text{vol}\mathfrak{P}=[\mathfrak{P}w_1\mathfrak{P}:\mathfrak{P}]=[\mathfrak{P}:\mathfrak{P} \cap ^{w_1}\mathfrak{P}]$. Hence $b_1=\text{vol}(\mathfrak{P}w_1\mathfrak{P})=[\mathfrak{P}:\mathfrak{P} \cap ^{w_1}\mathfrak{P}]$. Now as $\lambda_1^2= \frac{q^{n/2}}{b_1} \Longrightarrow \lambda_1= \frac{q^{n/4}}{b_1^{1/2}}=\frac{q^{n/4}}{[\mathfrak{P}:\mathfrak{P} \cap ^{w_1}\mathfrak{P}]^{1/2}}$.\par
				
				We have seen before that $\mathfrak{P}= K_{-,1} \mathfrak{P}_0 K_{+,0}, ^{w_1}\mathfrak{P}= K_{-,2} \mathfrak{P}_0 K_{+,-1}$. So $\mathfrak{P} \cap ^{w_1}\mathfrak{P}= K_{-,2} \mathfrak{P}_0K_{+,0}$. Hence
				
				\begin{align*}
				[\mathfrak{P}:\mathfrak{P} \cap ^{w_1}\mathfrak{P}]&=|\frac{K_{-,1}}{K_{-,2}}|\\
				&=|\lbrace X \in \mathrm{M}_n(k_E) \mid X= ^t{\overline{X}} \rbrace|\\
				&=q^{\frac{(n)(n+1)}{2}}\\
				&=q^{\frac{n^2+n}{2}}.
				\end{align*}\par
				
				So
				\begin{center}
					$\lambda_1=\frac{q^{n/4}}{[\mathfrak{P}:\mathfrak{P} \cap ^{w_1}\mathfrak{P}]^{1/2}}=\frac{q^{n/4}}{q^{\frac{n^2+n}{4}}(\rho(-1))^{1/2}}.$
				\end{center}\par
				
				Hence
				\begin{align*}
				(\phi_0\phi_1)(\zeta)&=(\lambda_0\psi_0)(\lambda_1\psi_1)(\zeta)\\
				&=(\lambda_0\lambda_1)(\psi_0\psi_1)(\zeta)\\
				&=\frac{q^{n/4}}{q^{\frac{n^2-n}{4}}}\frac{q^{n/4}}{q^{\frac{n^2+n}{4}}(\rho(-1))^{1/2}}1_W^{\vee}\\
				&=\frac{q^\frac{n-n^2}{2}1_W^{\vee}}{(\rho(-1))^{1/2}}\\
				&=\frac{q^\frac{n-n^2}{2}1_W^{\vee}}{(\rho(-1))^{1/2}}.
				\end{align*}\par
				
				As $-1 \in Z(\mathfrak{P})$ and $\rho^{\vee}$ is a representation of $\mathfrak{P}$, so $\rho^{\vee}(-1)= \omega_{\rho^{\vee}}(-1)$ where $\omega_{\rho^{\vee}}$ is the central character of $\mathfrak{P}$. Now $1=\omega_{\rho^{\vee}}(1)=(\omega_{\rho^{\vee}}(-1))^2$, so 
				$\rho^{\vee}(-1)= \omega_{\rho^{\vee}}(-1)= \pm 1$. We have seen before that $\lambda_1= \frac{q^{n/2}-1}{a_1}$ and $a_1 \in \mathbb{R}, a_1>0$, so $\lambda_1 >0$. But we know that $\lambda_1=\frac{q^{n/4}}{[\mathfrak{P}:\mathfrak{P} \cap ^{w_1}\mathfrak{P}]^{1/2}}=\frac{q^{n/4}}{q^{\frac{n^2+n}{4}}(\rho(-1))^{1/2}}$, hence $\rho^{\vee}(-1)=1$.\par
				
				Recall $g_i = q^{-n/4}\phi_i$ for $i=0,1$. We know that $\phi_i^2= (q^{n/2}-1)\phi_i +q^{n/2}$ for $i=0,1$. So for $i=0,1$ we have
				
				\begin{align*}
				g_i^2 &= q^{-n/2}\phi_i^2\\
				&=q^{-n/2}((q^{n/2}-1)\phi_i +q^{n/2})\\
				&=(1-q^{-n/2})\phi_i+1\\
				&=(1-q^{-n/2})q^{n/4}g_i+1\\
				&=(q^{n/4}- q^{-n/4})g_i+1.
				\end{align*}\par
				
				So $g_0g_1= q^{-n/2}\phi_0\phi_1 \Longrightarrow (g_0g_1)(\zeta)=q^{-n/2}(\phi_0\phi_1)(\zeta)= q^{-n/2}\frac{q^\frac{n-n^2}{2}1_W^{\vee}}{(\rho(-1))^{1/2}}=q^\frac{-n^2}{2}1_W^{\vee}$. From the earlier discussion in this section we have $T_P(\alpha)(\zeta)=\delta_P^{1/2}(\zeta)1_{W^{\vee}}$. From section 9.1, we have $\delta_P(\zeta)=  q^{-n^2}$. Hence $\delta_P^{1/2}(\zeta)= q^{-n^2/2}$. Therefore $(g_0g_1)(\zeta)= T_P(\alpha)(\zeta)$. So $(g_0g_1)(\zeta)= T_P(\alpha)(\zeta)$. We have $\text{supp}(T_P(\alpha))= \mathfrak{P}\zeta\mathfrak{P}$. As $\text{supp}(g_i)=\mathfrak{P}w_i\mathfrak{P}$, Lemma \ref{lem_10}  gives $\text{supp}(g_0g_1)= \mathfrak{P}\zeta\mathfrak{P}$. Therefore $g_0g_1= T_P(\alpha)$.
			\end{proof}
			
			\subsection{Calculation of $m_L(\pi\nu)$}
			
			Note $\pi\nu$ lies in $\mathfrak{R}^{[L,\pi]_L}(L)$. Recall $m_L$ is an equivalence of categories. As $\pi\nu$ is an irreducible representation of $L$, it follows that $m_L(\pi\nu)$ is a simple $\mathcal{H}(L,\rho_0)$-module. In this section, we identify the simple $\mathcal{H}(L,\rho_0)$-module corresponding to $m_L(\pi\nu)$. Calculating $m_L(\pi\nu)$ will be useful in answering the question in next section.\par
			
			From section 2.5, we know that $\pi =c$-$Ind_{\widetilde{\mathfrak{P}_0}}^L \widetilde{\rho_0}$, where $\widetilde{\mathfrak{P}_0}=\langle \zeta \rangle \mathfrak{P}_0, \widetilde{\rho_0}(\zeta^k j)=\rho_0(j)$ for $j \in \mathfrak{P}_0, k \in \mathbb{Z}$. Let us recall that $\nu$ is unramified character of $L$ from section 1. Let $V$ be space of $\pi \nu$ and $W$ be space of $\rho_0$. Recall $m_L(\pi\nu)= \mathrm{Hom}_{\mathfrak{P}_0}(\rho_0,\pi \nu)$. Let $f \in \mathrm{Hom}_{\mathfrak{P}_0}(\rho_0,\pi \nu)$. As $\mathfrak{P}_0$ is a compact open subgroup of $L$ and $\nu$ is an unramified character of $L$, so $\nu(j)=1$ for $j \in \mathfrak{P}_0$. We already know that $\alpha \in \mathcal{H}(L,\rho_0)$ with support of $\alpha$ being $\mathfrak{P}_0\zeta$ and $\alpha(\zeta)= 1_{W^{\vee}}$. Let $w \in W$ and we have seen in section 9.2 that the way  $\mathcal{H}(L,\rho_0)$ acts on $\mathrm{Hom}_{\mathfrak{P}_0}(\rho_0,\pi \nu)$ is given by:

			\begin{align*}
			(\alpha.f)(w)&= \int_L(\pi \nu)(l)f(\alpha^{\vee}(l^{-1})w)dl\\
			&=\int_L(\pi \nu)(l)f((\alpha(l)){\vee}w)dl\\
			&=\int_{\mathfrak{P}_0}(\pi \nu)(p\zeta)f((\alpha(p\zeta)){\vee}w)dp\\
			&=\int_{\mathfrak{P}_0}(\pi \nu)(p\zeta)f((\rho_0^{\vee}(p)\alpha(\zeta)){\vee}w)dp\\
			&=\int_{\mathfrak{P}_0}(\pi \nu)(p\zeta)f((\rho_0^{\vee}(p)1_{W^{\vee}}){\vee}w)dp\\
			&=\int_{\mathfrak{P}_0}(\pi \nu)(p\zeta)f((\rho_0^{\vee}(p)){\vee}w)dp\\
			&=\int_{\mathfrak{P}_0}\pi(p\zeta) \nu(p\zeta)f((\rho_0^{\vee}(p)){\vee}w)dp\\
			&=\int_{\mathfrak{P}_0}\pi(p\zeta) \nu(\zeta)f((\rho_0^{\vee}(p)){\vee}w)dp.\\
			\end{align*}\par

			Now $\langle, \rangle \colon W \times W^{\vee} \longrightarrow \mathbb{C}$ is given by: $\langle w, \rho_0^{\vee}(p)w^{\vee}\rangle = \langle \rho_0(p^{-1})w, w^{\vee} \rangle$ for $p \in \mathfrak{P}_0, w \in W$. So we have $(\rho_0^{\vee}(p))^{\vee}=\rho_0(p^{-1})$ for $p \in \mathfrak{P}_0$. Hence
			
			\begin{align*}
			(\alpha.f)(w)&=\int_{\mathfrak{P}_0}\pi(p\zeta) \nu(\zeta)f(\rho_0(p^{-1})w) dp.
			\end{align*}\par
			
			As $f \in \mathrm{Hom}_{\mathfrak{P}_0}(\rho_0,\pi \nu)$, so $(\pi \nu)(p)f(w) = f(\rho_0(p)w)$ for $p \in \mathfrak{P}_0, w \in W$. Hence
			
			\begin{align*}
			(\alpha.f)(w)&=\nu(\zeta)\int_{\mathfrak{P}_0}\pi(p\zeta)(\pi \nu)(p^{-1})f(w)dp\\
			&=\nu(\zeta)\int_{\mathfrak{P}_0}\pi(p\zeta)\pi(p^{-1})\nu(p^{-1})f(w)dp\\
			&=\nu(\zeta)\int_{\mathfrak{P}_0}\pi(p\zeta)\pi(p^{-1})f(w)dp.
			\end{align*}\par
			
			Now as $\pi = c$-$Ind_{\widetilde{\mathfrak{P}_0}}^L \widetilde\rho_0$ and $\widetilde{\mathfrak{P}_0}=\langle \zeta \rangle \mathfrak{P}_0, \widetilde{\rho_0}(\zeta^k j)=\rho_0(j) $ for $j \in \mathfrak{P}_0, k \in \mathbb{Z}$, so $\pi(p\zeta)=\pi(p)\widetilde{\rho_0}(\zeta)=\pi(p)\rho_0(1)=\pi(p)1_{W^{\vee}}$. Therefore
			
			\begin{align*}
			(\alpha.f)(w)&=\nu(\zeta)\int_{\mathfrak{P}_0}\pi(p)\pi(p^{-1})f(w)dp\\
			&=\nu(\zeta)f(w)\text{Vol}(\mathfrak{P}_0)\\
			&=\nu(\zeta)f(w)
			\end{align*}
			
			So $(\alpha.f)(w)=\nu(\zeta)f(w)$ for $w \in W$. So $\alpha$ acts on $f$ by multiplication by $\nu(\zeta)$. Recall for $\lambda \in \mathbb{C}^{\times}$, we write $\mathbb{C}_{\lambda}$ for the $\mathcal{H}(L,\rho_0)$-module with underlying abelian group $\mathbb{C}$ such that $\alpha.z=\lambda z$ for $z \in \mathbb{C}_{\lambda}$. Therefore $m_L(\pi \nu) \cong \mathbb{C}_{\nu(\zeta)}$.\par

			\section{Answering the question}
			
			Recall the following commutative diagram which we described earlier.
			
			\[
			\begin{CD}
			\mathfrak{R}^{[L,\pi]_G}(G)    @>m_G>>    \mathcal{H}(G,\rho)-Mod\\
			@A\iota_P^GAA                                    @A(T_P)_*AA\\
			\mathfrak{R}^{[L,\pi]_L}(L)    @>m_L>>     \mathcal{H}(L,\rho_0)-Mod
			\end{CD}
			\tag{CD}\]\par
			
			Observe that $\pi\nu$ lies in $\mathfrak{R}^{[L,\pi]_L}(L)$. From the above commutative diagram, it follows that $\iota_P^G(\pi\nu)$ lies in $\mathfrak{R}^{[L,\pi]_G}(G)$ and $m_G(\iota_P^G(\pi\nu))$ is an $\mathcal{H}(G,\rho)$-module. Recall $m_L(\pi \nu) \cong \mathbb{C}_{\nu(\zeta)}$ as $\mathcal{H}(L,\rho_0)$-modules. From the above commutative diagram, we have $m_G(\iota_P^G(\pi\nu)) \cong (T_P)_*(\mathbb{C}_{\nu(\zeta)})$ as $\mathcal{H}(G,\rho)$-modules. Thus to determine the unramified characters $\nu$ for which $\iota_P^G(\pi\nu)$ is irreducible, we have to understand when $(T_P)_*(\mathbb{C}_{\nu(\zeta)})$ is a simple $\mathcal{H}(G,\rho)$-module.\par

			Using notation on page 438 in \cite{MR2531913}, we have  $\gamma_1=\gamma_2= q^{n/2}$ for unramified case when $n$ is odd and $\gamma_1=\gamma_2= q^{n/4}$ for ramified case when $n$ is even. As in Proposition 1.6 of \cite{MR2531913}, let $\Gamma=\{\gamma_1\gamma_2,-\gamma_1\gamma_2^{-1}, -\gamma_1^{-1}\gamma_2,(\gamma_1\gamma_2)^{-1}\}$. So by Proposition 1.6 in \cite{MR2531913}, 
			$(T_P)_*(\mathbb{C}_{\nu(\zeta)})$ is a simple $\mathcal{H}(G,\rho)$-module $\Longleftrightarrow \nu(\zeta) \notin \Gamma$. Recall $\pi= c $-$ Ind_{Z(L)\mathfrak{P}_0}^L \widetilde{\rho_0}$ where $\widetilde{\rho_0}(\zeta^k j)=\rho_0(j)$ for $j \in \mathfrak{P}_0, k \in \mathbb{Z}$ and  $\rho_0=\tau_{\theta}$ for some regular character $\theta$ of $l^{\times}$ with $[l:k_E]=n$. Hence we can conclude that $\iota_P^G(\pi \nu)$ is irreducible for the unramified case when $n$ is odd $\Longleftrightarrow \nu(\zeta) \notin \{q^n,q^{-n},-1\}$, $\theta^{q^{n+1}}= \theta^{-q}$ and  $\iota_P^G(\pi \nu)$ is irreducible for the ramified case when $n$ is even $\Longleftrightarrow \nu(\zeta) \notin \{q^{n/2},q^{-n/2},-1\}$, $\theta^{q^{n/2}}= \theta^{-1}$.\par
			
			Recall that in the unramified case when $n$ is even or in the ramified case when $n$ is odd we have  $N_G(\rho_0)= Z(L) \mathfrak{P}_0$. Thus $\mathfrak{I}_G(\rho)= \mathfrak{P}(Z(L)\mathfrak{P}_0)\mathfrak{P}= \mathfrak{P}Z(L)\mathfrak{P}$.\par
			
			From Corollary 6.5 in \cite{MR1486141} which states that if $\mathfrak{I}_G(\rho) \subseteq \mathfrak{P}L\mathfrak{P}$ then \[T_P \colon \mathcal{H}(L,\rho_0) \longrightarrow \mathcal{H}(G,\rho)\] is an isomorphism of $\mathbb{C}$-algebras. As we have $\mathfrak{I}_G(\rho)=\mathfrak{P}Z(L) \mathfrak{P}$ in the unramified case when $n$ is even or in the ramified case when $n$ is odd, so $\mathcal{H}(L,\rho_0) \cong \mathcal{H}(G,\rho)$ as $\mathbb{C}$-algebras. So from the commutative diagram (CD) on the previous page, we can conclude that $\iota_P^G(\pi \nu)$ is irreducible for any unramified character $\nu$ of $L$. That proves Theorem \ref{the_1}. 
		
\nocite{MR1235019}
\nocite{MR1371680}
\nocite{MR1486141}
\nocite{MR1643417}
\nocite{MR2276353}
\nocite{MR2531913}
\nocite{MR1266747}
\nocite{Thomas2014}
\nocite{green}
\nocite{heiermann_2011}
\nocite{heiermann_2017}
\nocite{murnaghan_repka_1999}
\bibliographystyle{plain}
\bibliography{sandeep_arxiv}

\end{document}